\documentclass[12pt,a4paper,reqno,twoside]{amsbook}
\usepackage{amssymb,a4wide,times,mathptmx,epic,eepic}
\providecommand{\BBb}[1]{{\mathbb{#1}}}

\providecommand{\cal}[1]{{\mathcal{#1}}}   

\newcommand{\ang}[1]{\langle#1\rangle}

\newcommand{\Bcirc}{\overset{\lower 1.5pt%
              \hbox{$@,@,@,@,@,\scriptscriptstyle\circ$}}B{}}
\newcommand{\Binfty}{\overset{\lower 1.5pt%
              \hbox{$@,@,@,@,@,\scriptscriptstyle\infty$}}B{}}
\newcommand{\bigdot}{\mathbin{\raise.65\jot\hbox{$\scriptscriptstyle\bullet$}}}

\newcommand{\C}{{\BBb C}}

\newcommand{\dual}[2]{\langle\,#1,\,#2\,\rangle}
\newcommand{\Dual}[2]{\bigl\langle\,#1,\,#2\,\bigr\rangle}

\newcommand{\erd}{\overset{\lower 1pt\hbox{\large.}}{e}
                  \overset{\lower 1pt\hbox{\large.}}{r}}

\newcommand{\Fcirc}{\overset{\lower 1.5pt%
               \hbox{$@,@,@,@,@,\scriptscriptstyle\circ$}}F{}}
\newcommand{\fracc}[2]{{
                \textstyle\frac{#1}{\raise 1pt\hbox{$\scriptstyle #2$}}}}

\newcommand{\fracnp}{\fracc np}

\newcommand{\fracci}[2]{{\frac{#1}{\raise 1pt\hbox{$\scriptscriptstyle #2$}}}}
\newcommand{\fracpi}{\fracci1p}

\newcommand{\im}{\operatorname{i}}
\renewcommand{\Im}{\operatorname{Im}}

\newcommand{\lap}{\operatorname{\Delta}}
\newcommand{\loc}{\operatorname{loc}}
\newcommand{\lOm}{\ell_{\Omega}}

\newcommand{\mlap}{-\!\operatorname{\Delta}}

\newcommand{\nrm}[2]{\|#1\|_{#2}}
\newcommand{\Nrm}[2]{\bigl\|#1\bigr\|_{#2}}

\newcommand{\order}{\operatorname{order}}
\newcommand{\op}[1]{\operatorname{#1}}
\newcommand{\OP}{\operatorname{OP}}
\newcommand{\N}{\BBb N}

\renewcommand{\Re}{\operatorname{Re}}

\newcommand{\R}{{\BBb R}}
\newcommand{\Rn}{{\BBb R}^{n}}

\newcommand{\rOm}{r_{\Omega}}
{
\end{minipage}%
\par\medskip\noindent}%
\newcounter{enmcount}\renewcommand{\theenmcount}{{\rm\arabic{enmcount}}}

\newcounter{rmcount}\renewcommand{\thermcount}{{\rm\roman{rmcount}}}
\newenvironment{rmlist}{%
\begin{list}{{\rm(\thermcount)}}{\setlength{\labelwidth}{\leftmargin}%
\usecounter{rmcount}}}{\end{list}}
\newcounter{Rmcount}\renewcommand{\theRmcount}{{\rm\Roman{Rmcount}}}
\newenvironment{Rmlist}{%
\begin{list}{{\rm(\theRmcount)}}{\setlength{\labelwidth}{\leftmargin}%
\usecounter{Rmcount}}}{\end{list}}

\newcommand{\scal}[2]{(\,#1\,|\, #2\,)}
\newcommand{\Scal}[2]{\bigl(\,#1\mid #2\,\bigr)}

\newcommand{\singsupp}{\operatorname{sing\,supp}}
\newcommand{\supp}{\operatorname{supp}}
\newcommand{\Z}{\BBb Z}

\renewcommand{\check}[1]{\overset{{\scriptscriptstyle \vee}}{#1}}
\renewcommand{\hat}[1]{\overset{{\scriptscriptstyle \wedge}}{#1}}

\numberwithin{section}{chapter}
\numberwithin{equation}{chapter}
\newtheorem{thm}{Theorem}
\numberwithin{thm}{section}
\newtheorem{prop}[thm]{Proposition}
\newtheorem{lem}[thm]{Lemma}
\newtheorem{cor}[thm]{Corollary}
\theoremstyle{definition}
\newtheorem{defn}[thm]{Definition}

\theoremstyle{remark}
\newtheorem{rem}[thm]{Remark}
\newcommand{\OPT}{\widetilde{\operatorname{OP}}}
\newcommand{\Bbar}{\overline{B}}
\title{On the Theory of Type $\mathbf{1},\mathbf{1}$-Operators}
\author{~\\[-\baselineskip]\large{\rm by}\\[0.75\baselineskip]
{\bf Jon Johnsen}\\[7\baselineskip]
\begin{figure}[h]
\setlength{\unitlength}{0.00075in}
{
\begin{picture}(5694,4764)(-170,-10)
\thicklines
\path(-250,2037)(5712,2037)
\path(5592,2007)(5712,2037)(5592,2067)
\put(5880,2037){\makebox(0,0)[c]{$\xi$}}
\path(2712,0)(2712,4737)
\path(2742.000,4617.000)(2712.000,4737.000)(2682.000,4617.000)
\put(2712,4880){\makebox(0,0)[c]{$\eta$}}
\thinlines
\put(2893.731,2017.961){\arc{1437.042}{6.2567}{7.0072}}
\put(3600,1700){\makebox(0,0)[l]{$\arctan(\frac1{1+\varepsilon})$}}
\put(320,4520){\makebox(0,0)[l]{$\xi+\eta=0$}}
\path(237,4512)(4512,237)
\texture{80222222 22555555 55808080 80555555 55222222 22555555 55880888 8555555 
	55222222 22555555 55808080 80555555 55222222 22555555 55080808 8555555 
	55222222 22555555 55808080 80555555 55222222 22555555 55880888 8555555 
	55222222 22555555 55808080 80555555 55222222 22555555 55080808 8555555 }
\shade\path(5412,237)(4062,1137)(3162,1137)(3612,237)
\path(5412,237)(4062,1137)(3162,1137)(3612,237)
\dashline{60.000}(2712,2037)(4062,1137)
\path(2622,2937)(2802,2937) \put(2850,2970){\makebox(0,0)[l]{$\varepsilon^{-1}$}}
\path(4512,237)(4737,12)
\path(2622,1137)(2802,1137)\put(2580,1180){\makebox(0,0)[r]{$-\varepsilon^{-1}$}}
\texture{55888888 88555555 5522a222 a2555555 55888888 88555555 552a2a2a 2a555555 
	55888888 88555555 55a222a2 22555555 55888888 88555555 552a2a2a 2a555555 
	55888888 88555555 5522a222 a2555555 55888888 88555555 552a2a2a 2a555555 
	55888888 88555555 55a222a2 22555555 55888888 88555555 552a2a2a 2a555555 }
\shade\path(12,3837)(1362,2937)(2262,2937)(1812,3837)
\path(12,3837)(1362,2937)(2262,2937)(1812,3837)
\dottedline{45}(12,3837)(-258,4017)
\dottedline{45}(1812,3837)(1722,4017)
\dottedline{45}(3612,237)(3702,57)
\dottedline{45}(5412,237)(5682,57)
\end{picture}
}

\end{figure}
}
\begin{document}
\maketitle
%
\begin{center}
{\bf\large On the Theory of Type $\mathbf{1},\mathbf{1}$-Operators}
{~\\[\baselineskip]{\rm by}\\[\baselineskip]
{\bf Jon Johnsen\\
Department of Mathematical Sciences\\
Aalborg University\\
Fredrik Bajers Vej 7G\\
DK--9220 Aalborg {\O}st} 
\\[2\jot]
{E-mail:\tt jjohnsen@math.aau.dk}}
\end{center}
\vfill

\noindent
The author's doctoral dissertation, defended successfully at Aalborg University on 17 June 2011. 
Updated in October 2016.
\bigskip

\noindent
Copyright \copyright 2016 by the author.
\vspace{1cm}

\noindent
{\footnotesize The cover illustration visualizes H{\"o}rmander's microlocalisations around non-compact parts of the twisted diagonal, as used in his analysis of pseudo-differential operators of type $1,1$.}

%
%
%
%
%

\newpage
\setcounter{page}{3}
\renewcommand{\thepage}{\roman{page}}
\chapter*{Contents}
\contentsline {chapter}{\tocchapter {Chapter}{}{Dissertation preface}}{v}
\contentsline {chapter}{\tocchapter {Chapter}{}{Scientific work of Jon Johnsen}}{vii}
\contentsline {chapter}{\tocchapter {Chapter}{1}{Introduction}}{1}
\contentsline {section}{\tocsection {}{1.1}{Basics}}{1}
\contentsline {section}{\tocsection {}{1.2}{The historic development}}{2}
\contentsline {section}{\tocsection {}{1.3}{Application to non-linear boundary value problems}}{7}
\contentsline {section}{\tocsection {}{1.4}{The definition of type $1,1$-operators}}{9}
\contentsline {chapter}{\tocchapter {Chapter}{2}{Preliminaries}}{11}
\contentsline {section}{\tocsection {}{2.1}{Notions and notation}}{11}
\contentsline {section}{\tocsection {}{2.2}{Scales of function spaces}}{12}
\contentsline {chapter}{\tocchapter {Chapter}{3}{The general definition of type $1,1$-operators}}{15}
\contentsline {section}{\tocsection {}{3.1}{Definition by vanishing frequency modulation}}{15}
\contentsline {section}{\tocsection {}{3.2}{Consequences for type $1,1$-operators}}{17}
\contentsline {chapter}{\tocchapter {Chapter}{4}{Techniques for pseudo-differential operators}}{21}
\contentsline {section}{\tocsection {}{4.1}{Pointwise estimates of pseudo-differential operators}}{21}
\contentsline {section}{\tocsection {}{4.2}{The spectral support rule}}{24}
\contentsline {section}{\tocsection {}{4.3}{Stability of extended distributions under regular convergence}}{26}
\contentsline {subsection}{\tocsubsection {}{4.3.1}{Other extensions}}{28}
\contentsline {chapter}{\tocchapter {Chapter}{5}{Review of qualitative results}}{31}
\contentsline {section}{\tocsection {}{5.1}{Consistency among extensions}}{31}
\contentsline {subsection}{\tocsubsection {}{5.1.1}{Extension to functions with compact spectrum}}{31}
\contentsline {subsection}{\tocsubsection {}{5.1.2}{Extension to slowly growing functions}}{31}
\contentsline {subsection}{\tocsubsection {}{5.1.3}{Extension by continuity}}{32}
\contentsline {subsection}{\tocsubsection {}{5.1.4}{Extensions through paradifferential decompositions}}{33}
\contentsline {section}{\tocsection {}{5.2}{Maximality of the definition by vanishing frequency modulation}}{34}
\contentsline {section}{\tocsection {}{5.3}{The maximal smooth space}}{36}
\contentsline {section}{\tocsection {}{5.4}{The pseudo-local property of type $1,1$-operators}}{37}
\contentsline {section}{\tocsection {}{5.5}{Non-preservation of wavefront sets}}{38}
\contentsline {section}{\tocsection {}{5.6}{The support rule and its spectral version}}{40}
\contentsline {chapter}{\tocchapter {Chapter}{6}{Continuity results}}{43}
\contentsline {section}{\tocsection {}{6.1}{Littlewood--Paley decompositions of type $1,1$-operators}}{43}
\contentsline {subsection}{\tocsubsection {}{6.1.1}{Dyadic corona decompositions of symbols and operators}}{43}
\contentsline {subsection}{\tocsubsection {}{6.1.2}{Calculation of symbols and remainder terms}}{46}
\contentsline {section}{\tocsection {}{6.2}{The twisted diagonal condition}}{47}
\contentsline {section}{\tocsection {}{6.3}{The twisted diagonal condition of order $\sigma $}}{50}
\contentsline {subsection}{\tocsubsection {}{6.3.1}{Localisation along the twisted diagonal}}{50}
\contentsline {subsection}{\tocsubsection {}{6.3.2}{The self-adjoint subclass $\mathaccentV {tilde}07ES^d_{1,1}$}}{51}
\contentsline {subsection}{\tocsubsection {}{6.3.3}{Paradifferential decompositions for the self-adjoint subclass}}{53}
\contentsline {section}{\tocsection {}{6.4}{Domains of type $1,1$-operators}}{55}
\contentsline {subsection}{\tocsubsection {}{6.4.1}{Proof of Theorem\nonbreakingspace 6.4.1\hbox {}}}{56}
\contentsline {section}{\tocsection {}{6.5}{General Continuity Results}}{57}
\contentsline {section}{\tocsection {}{6.6}{Direct estimates for the self-adjoint subclass}}{60}
\contentsline {chapter}{\tocchapter {Chapter}{7}{Final remarks}}{65}
\contentsline {chapter}{\tocchapter {Chapter}{}{Bibliography}}{67}
\contentsline {chapter}{\tocchapter {Chapter}{}{Resum\'e (Danish summary)}}{71}

\chapter*{Dissertation Preface}
The following is identical to the dissertation submitted on 
November 1, 2010, except that the references \mbox{[18]} and \mbox{[19]} 
on page \mbox{viii} have been updated (and similarly in the bibliography). 

\bigskip\noindent
\emph{Aalborg, 9 May 2011
\hfill
Jon Johnsen}

\vspace{1in}

\begin{center}
 {\bf\large  Preface to the 2016-edition}
\end{center}
\bigskip
\noindent
The changes made in the 2016-edition of this dissertation are only minor.
First of all the mathematical misprints communicated at the defense on 17 June 2011 have been corrected.
Secondly a few typos and issues in the text have been improved.
Thirdly the references \mbox{[18]} and \mbox{[19]} have been updated again, especially the latter, which 
subsequently gave rise to the two publications \mbox{[19a]} and \mbox{[19b]} that now have been added for 
the reader's sake on page \mbox{viii}. 

However, \mbox{[19a]} and \mbox{[19b]} have not been added to the bibliograhy at the end, since the
text still refers to the technical report \mbox{[Joh10c]} (which is the same as \mbox{[19]}) in order to preserve 
the original exposition in the dissertation.

\bigskip\noindent
\emph{Aalborg, 14 October 2016
\hfill
Jon Johnsen}

\chapter*{Scientific work of Jon Johnsen}

\begin{itemize}

\item[{[1]}]
\newblock {\em {The stationary Navier--Stokes equations in $L_p$-related
  spaces}}.
\newblock PhD thesis, University of Copenhagen, Denmark, 1993.
\newblock {Ph.D.-series {\bf 1}}.

\item[{[2]}]
\newblock {\em Pointwise multiplication of Besov and Triebel--Lizorkin spaces}.
\newblock {Math. Nachr.}, {\bf 175} (1995), 85--133.

\item[{[3]}]
\newblock {\em Regularity properties of semi-linear boundary problems
  in Besov and Triebel--Lizorkin spaces}.
\newblock In {{Journ\'ees ``\'equations deriv\'ees partielles'', St.~Jean
  de~Monts, 1995}}, pages XIV1--XIV10. Grp.~de Recherche CNRS no.~1151, 1995.

\item[{[4]}]
\newblock {\em Elliptic boundary problems and the Boutet de Monvel calculus in
  Besov and Triebel--Lizorkin spaces}.
\newblock {Math. Scand.}, {\bf 79} (1996), 25--85.

\item[{[5]}]
(with T.~Runst)
\newblock {\em Semilinear boundary problems of composition type in
  $L_p$-related spaces}.
\newblock {Comm.~ P.~D.~E.}, {\bf 22} (1997), 1283--1324.

\item[{[6]}]
\newblock {\em On spectral properties of Witten-Laplacians, their
  range projections and Brascamp--Lieb's inequality}.
\newblock {Integr.~equ.~oper.~theory}, {\bf 36} (2000), 288--324.

\item[{[7]}]
\newblock {\em Traces of {B}esov spaces revisited}.
\newblock {Z. Anal. Anwendungen}, {\bf 19} (2000), 763--779.

\item[{[8]}]
(with W.~Farkas and W.~Sickel)
\newblock {\em Traces of anisotropic Besov--Lizorkin--Triebel
  spaces---a complete treatment of the borderline cases}.
\newblock {Math. Bohemica}, {\bf 125} (2000), 1--37.

\item[{[9]}]
\newblock {\em Regularity results and parametrices of semi-linear
  boundary problems of product type}. 
\newblock In D.~Haroske and H.-J. Schmeisser, editors, {\em {Function spaces,
  differential operators and nonlinear analysis.}}, pages 353--360.
  Birkh{\"a}user, 2003.

\item[{$^{\pmb{*}}$[10]}]
\newblock {\em Domains of type $1,1$ operators: a case for Triebel--Lizorkin
  spaces}.
\newblock {{C. R. Acad. Sci. Paris S\'er. I Math.}}, {\bf 339} (2004),
 115--118.

\item[{$^{\pmb{*}}$[11]}]
\newblock {\em Domains of pseudo-differential operators: a case for the
  Triebel--Lizorkin spaces}.
\newblock {{J. Function Spaces Appl.}}, {\bf 3} (2005), 263--286.

\item[{[12]}]
(with W.~Sickel)
\newblock {\em A direct proof of {S}obolev embeddings for quasi-homogeneous
  {Lizorkin--Triebel} spaces with mixed norms.}
\newblock {J. Function Spaces Appl.}, {\bf 5} (2007), 183--198.

\item[{[13]}] (with B. Sloth Jensen and Chunyan Wang)
\newblock {\em Moment evolution of Gaussian and geometric Wiener diffusions;}
\newblock In B. Sloth Jensen, T. Palokangas, editors,
    \emph{Stochastic Economic Dynamics}, pages 57-100.
\newblock Copenhagen Business School Press 2007, Fredriksberg, Denmark.

\item[{[14]}]
(with W.~Sickel)
\newblock {\em On the trace problem for {Lizorkin--Triebel} spaces with mixed norms.}
\newblock {Math. Nachr.}, {\bf 281} (2008), 1--28.

\item[{[15]}]
\newblock {\em Parametrices and exact paralinearisation of semi-linear boundary
  problems}.
\newblock {Comm. Part. Diff. Eqs.}, {\bf 33} (2008), 1729--1787.

\item[{$^{\pmb{*}}$[16]}]
\newblock{\em Type {$1,1$}-operators defined by vanishing frequency
  modulation}. 
\newblock In L.~Rodino and M.~W. Wong, editors, {\em New Developments in
  Pseudo-Differential Operators}, volume 189 of {Operator Theory: Advances
  and Applications}, pages 201--246. Birkh{\"a}user, 2008.

\item[{[17]}]
\newblock {\em Simple proofs of nowhere-differentiability for {Weierstrass's}
  function and cases of slow growth.} 
\newblock {J.~Fourier Anal. Appl.} {\bf 16} (2010), 17--33.

\item[{$^{\pmb{*}}$[18]}]
\newblock {\em Pointwise estimates of pseudo-differential operators}. 
\newblock 
Journal of Pseudo-Differential Operators and Applications, {\bf 2} (2011), 377--398.
\newblock (Originally Tech.\ Report R-2010-12, Aalborg University.)

\item[{$^{\pmb{*}}$[19]}]
\newblock {\em Type $1,1$-operators on spaces of temperate distributions.} 
\newblock Tech. Report R-2010-13, Aalborg University, 2010. 
\newline
(Available at {\tt http://vbn.aau.dk/files/38938995/R-2010-13.pdf})

\item[{[19a]}]
\newblock {\em Lp-theory of type 1,1-operators.} 
\newblock Math.\ Nachr., {\bf 286} (2013), 712--729. 
\newline
\newblock DOI:10.1002/mana.201300313

\item[{[19b]}]
\newblock {\em Fundamental results for pseudo-differential operators of type $1,1$.} 
\newblock Axioms 5 (2016), 13 (37 pages).
\newblock  DOI:10.3390/axioms5020013
\end{itemize}

\bigskip\noindent
The five entries marked by $^*$ in the above list constitute 
the author's doctoral dissertation.

Note made in 2016-edition: Subsequently \mbox{[18]} was published as stated, while \mbox{[19]} resulted in 
the two articles \mbox{[19a]} and \mbox{[19b]}.

\chapter{Introduction}
 \setcounter{page}{1}  
\renewcommand{\thepage}{\arabic{page}}
\noindent
In this presentation of the subject it is assumed that the
reader is familiar with basic concepts of Schwartz' distribution theory;
Section~\ref{NN-ssect} below gives a summary of this and notation used
throughout. 

\section{Basics}\noindent
An operator of type $1,1$ is a special example of a pseudo-differential
operator, whereby the latter is 
the mapping $u\mapsto a(x,D)u$ defined on Schwartz functions $u(x)$, ie 
on the $u\in \cal S(\Rn)$, by the classical Fourier integral
\begin{equation}
  a(x,D)u(x)= (2\pi)^{-n}\int_{\Rn} e^{\im x\cdot \eta} 
  a(x,\eta)\hat u(\eta)\,d\eta.
  \label{axDu-eq}
\end{equation}
Hereby its \emph{symbol} $a(x,\eta)$ could in general be of type $\rho,\delta$ 
for $0\le \delta\le \rho\le 1$ and, say of order $d\in\R$.
This means that $a(x,\eta)$ is in $C^\infty (\Rn\times \Rn)$ and satisfies
L.~H{\"o}rmander's condition that for all multiindices $\alpha,\beta\in
\N_0^n$ there is a constant $C_{\alpha,\beta}$ such that
\begin{equation}
  |D^\alpha_\eta D^\beta_x a(x,\eta)|\le C_{\alpha,\beta}
  (1+|\eta|)^{m-\rho|\alpha|+\delta|\beta|},
  \quad\text{for}\quad x\in \Rn,\ \eta\in \Rn.
  \label{Srd-eq}
\end{equation}
Such symbols constitute the Fr\'echet space 
$S^d_{\rho,\delta}(\Rn\times \Rn)$. The map
$a(x,D)u$ is also written $\OP(a(x,\eta))u$.

The classical case is $\rho=1$, $\delta=0$, that gives a framework for 
partial differential operators with bounded $C^\infty$ coefficients on $\Rn$. 
For example, when
\begin{equation}
p(x,D)=\sum_{|\alpha|\le d}a_{\alpha}(x)D^{\alpha}
\end{equation}
is applied to $u=\cal F^{-1}\cal Fu$, it is seen at once that
$p(x,D)$ 
has symbol $p(x,\eta)=\sum_{|\alpha|\le d}a_\alpha(x)\eta^{\alpha}$, 
which belongs to $S^d_{1,0}(\Rn\times \Rn)$.
It is well known that this allows inversion of $p(x,D)$ modulo
smoothing operators if $p(x,\eta)$ is elliptic, ie if 
$|p(x,\eta)|\ge c|\eta|^d>0$ for $|\eta|\ge1$.

\bigskip

\emph{A type $1,1$-operator}
is the more general case with $\rho=1$, $\delta=1$ in \eqref{Srd-eq}. 
A basic example of such symbols is due to C.~H.~Ching~\cite{Chi72}; 
it results by taking a 
unit vector $\theta\in \Rn$ and some auxiliary function 
$A\in C^\infty_0(\Rn)$ for which
$A(\eta)\ne0$ only holds in the corona
$\tfrac{3}{4}\le |\eta|\le \tfrac{5}{4}$ and setting 
\begin{equation}
  a_\theta(x,\eta)=\sum_{j=0}^\infty 2^{jd}\exp(-\im 2^j\theta\cdot x)
   A(2^{-j}\eta).
  \label{Ching-eq}
\end{equation}
This symbol is $C^\infty $ since there is at most one non-trivial term 
at each point $(x,\eta)$; it belongs to $S^d_{1,1}$ because $x$-derivatives of
the exponential function increases the order of growth with respect to $\eta$,
since $|2^j\theta|\approx |\eta|$ on $\supp A(2^{-j}\cdot)$.

Type $1,1$-operators are interesting
because they have important applications to non-linear maps and non-linear
partial differential operators, as indicated below,\,---\,but
this is undoubtedly also the origin of this operator class's peculiar 
properties.

To give a glimpse of this, it is recalled that
elementary estimates show that the mapping 
$\OP\colon (a,u)\mapsto a(x,D)u$  in \eqref{axDu-eq} is 
bilinear and  continuous
\begin{equation}
  S^d_{1,1}(\Rn\times \Rn)\times \cal S(\Rn)\to \cal S(\Rn).
  \label{SdS-eq}
\end{equation}
Beyond this, difficulties emerge when one tries to extend a given type 
$1,1$-operator $a(x,D)$ in a
consistent way to $\cal S'(\Rn)\setminus \cal S(\Rn)$. It is also
a tricky task to determine the subspaces $E$ with 
\begin{equation}
 \cal S(\Rn)\subset E\subset \cal S'(\Rn)  
\end{equation}
to which $a(x,D)$ extends.
Conversely, already when $E$  is fixed as $E=L_2(\Rn)$,
there is no known characterisation of \emph{symbols} of the
type $1,1$-operators that extend to $E$. 

Above all, the main technical difficulty of type $1,1$-operators is that
they can change every frequency in $u(x)$, ie every $\eta\in \supp\hat u$,
to the frequency $\xi=0$\,---\,intuitively this can be understood from
\eqref{Ching-eq} because 
the factor $e^{-\im x\cdot 2^{j}\theta}$ oscillates as much as 
$e^{\im x\cdot \eta}$ in \eqref{axDu-eq}.

Consequently, at every singular point $x_0$ of $u$
they may change the high frequencies causing the singularity, hence change
its nature (known as non-preservation of wavefront sets). 
However, from this perspective it might seem surprising
that they cannot \emph{create} singularities; for open sets $\Omega\subset \Rn$
this means that
\begin{equation}
  \quad\text{$u$ is $C^\infty $ in $\Omega$}\quad
\implies
  \quad\text{$a(x,D)u$ is $C^\infty $ in $\Omega$}.
  \label{C8-eq}
\end{equation}
(This is known as the \emph{pseudo-local} property). 
As \eqref{C8-eq} obviously holds true whenever $a(x,D)$ in \eqref{axDu-eq}
is applied to a Schwartz function, cf \eqref{SdS-eq}, it is clear that 
\eqref{C8-eq} pertains to the $u\in \cal S'\setminus \cal S$ on which
$a(x,D)$ can be defined, and that \eqref{axDu-eq} alone is of little use in
the proof of \eqref{C8-eq}.

Besides the challenge of describing the unusual properties of type
$1,1$-operators, they also have interesting applications
as recalled in the next two sections.

\section{The historic development}
\label{review-ssect}\noindent
The review below is mainly cronological and deliberately brief, 
but hopefully it can serve the reader as a point of reference in
Chapters~\ref{vfm-sect}--\ref{final-sect}. 
The author's contributions are given in footnotes where comparisons make sense
(a thorough review will follow in Section~\ref{dozen-ssect} below).

\bigskip

Symbols of type $\rho,\delta$ were introduced in 1966 in a seminar on
hypoelliptic equations by L.~H{\"o}rmander \cite{Hrd}. 
(Unlike the definition of $S^d_{\rho,\delta}(\Rn\times \Rn)$ in 
\eqref{Srd-eq},
the estimates were local in $x$ as customary at that time.)

The pathologies of type $1,1$-operators were revealed around 1972--73 when 
C.~H.~Ching \cite{Chi72} in his thesis gave examples of symbols
$a_\theta(x,\eta)$ in $S^0_{1,1}$ for which
the corresponding operators are unbounded in $L^2(\Rn)$.
Essentially these symbols had the form in \eqref{Ching-eq}.

Moreover, E.~M.~Stein
showed $C^s_*$-boundedness, $s>0$, for all operators of order $d=0$,
in lecture notes from Princeton University (1972-73).
This result is now available in \cite[VII.\S 1.3]{Ste93}, 
albeit with a misprint
in the reference to the lecture notes (as noticed in \cite{JJ08vfm}).

Afterwards C.~Parenti and L.~Rodino \cite{PaRo78} discovered that some type
$1,1$-operators do not preserve wavefront sets.%
\footnote{This is extended to all $d\in \R$, $n\in \N$ in \cite[Sect.~3.2]{JJ08vfm}
with exact formulae for the wavefront sets.}
As the background for this,
the pseudo-local property of type $1,1$-operators was anticipated in
\cite{PaRo78} with an incomplete argument.%
\footnote{The first full proof appeared in \cite[Thm.~6.4]{JJ08vfm}.}

Around 1980, Y.~Meyer \cite{Mey80,Mey81}
obtained the fundamental property that a composition
operator $u\mapsto F(u)$, for a fixed $C^\infty$-function $F$ with 
$F(0)=0$, acting on $u\in\bigcup_{s>n/p} H^s_{p}(\Rn)$, can be
written 
\begin{equation}
  F(u)=a_u(x,D)u  
\end{equation}
for a specific $u$-dependent symbol $a_u\in S^0_{1,1}$. Namely, when
$1=\sum_{j=0}^\infty \Phi_j$ is a Littlewood--Paley partition of unity,
then $a_u(x,\eta)$ is an elementary symbol in the sense of R.~R.~Coifman and
Y.~Meyer~\cite{CoMe78}, ie it is given by the formula
\begin{equation}
  a_u(x,\eta)=\sum_{j=0}^\infty m_j(x)\Phi_j(\eta)
\end{equation}
with the smooth multipliers   
\begin{equation}
  m_j(x)=\int_0^1 F'(\sum_{k<j}\Phi_k(D)u(x)+t\Phi_j(D)u(x))\,dt.
\end{equation}
This gave a convenient proof of the fact that the non-linear map 
$u\mapsto F(u)$ sends $H^s_p(\Rn)$ into itself for $s>n/p$. 
Indeed, this follows as
Y.~Meyer for general $a\in S^d_{1,1}$, 
using reduction to elementary symbols,
established continuity 
\begin{equation}
 H^{t+d}_r(\Rn)\xrightarrow[]{a(x,D)} 
 H^t_r(\Rn)\qquad\text{for $t>0$, $1<r<\infty$}. 
  \label{Meyer-eq}
\end{equation}
So for $a=a_u$ and $t=s$, $r=p$ this yields at once that $F(u)=a_u(x,D)u$
also belongs to $H^s_p$. For integer $s$ this could also be seen directly by
calculating derivatives up to order $s$ of $F(u)$,
but for non-integer $s>n/p$, this use of pseudo-differential operators  
is a particularly elegant proof method.%
\footnote{In \cite[Sect.~9]{JJ08vfm} these results are deduced from
the precise definition of type $1,1$-operators in \cite{JJ08vfm}, 
together with a straightforward proof of \emph{continuity} on $H^s_p$ of 
$u\mapsto F\circ u$ in Theorem~9.4 there.}

It was also realised then that type $1,1$-operators show up in J.-M.~Bony's 
paradifferential calculus \cite{Bon} and microlocal inversion 
together with propagations of singularites for non-linear partial
differential equations of the form $F(x,u(x),\dots,\partial^\alpha_xu(x))=0$. 

In the wake of this, in  1983,  G.~Bourdaud proved
boundedness on the Besov space $B^{s}_{p,q}(\Rn)$ for
$s>0$, $p,q\in [1,\infty ]$ in his thesis, cf \cite{Bou83,Bou88}.
He also gave a simplified proof of \eqref{Meyer-eq},
and noted that by duality and interpolation every type $1,1$-operator
\begin{equation}
  a(x,D)\colon C^\infty_0(\Rn)\to\cal D'(\Rn)  
\end{equation}
with $d=0$ is bounded on $H^s_p(\Rn)$ for all real $s$, $1<p<\infty $, in
particular on $L_2$, if its adjoint 
$a(x,D)^*\colon C^\infty_0(\Rn)\to\cal D'(\Rn)$ is also of type $1,1$.

Denoting this subclass of symbols by $\tilde S^0_{1,1}$, or more generally
\begin{equation}
  \OP(\tilde S^d_{1,1})=\OP(S^d_{1,1})\cap \OP(S^d_{1,1})^*,
\end{equation}
he proved that $\OP(\tilde S^0_{1,1})$ is a maximal self-adjoint subalgebra
of $\BBb{B}(L_2(\Rn))\cap \OP(S^0_{1,1})$. Hence self-adjointness suffices
for $L_2$-boundedness, but it is \emph{not} necessary:

G.~Bourdaud also showed that the auxiliary function $A$ in 
Ching's counter-example can be chosen for $n=1$ so that  
$a_\theta(x,D)$ \emph{does} belong to $\BBb{B}(L_2)\cap \OP(S^0_{1,1})$ even
though neither $a_\theta(x,D)^*$ nor $a_\theta(x,D)^2$ is of type $1,1$.

In addition G.~Bourdaud analysed the borderline $s=0$ and showed that
every $a(x,D)$ of order $0$ is bounded $B^0_{p,1}(\Rn)\to L_p(\Rn)$ for all
$p\in [1,\infty ]$, where the Besov space $B^0_{p,1}$ 
is slightly smaller than $L_p$; whilst $a_\theta(x,D)$ 
was proven unbounded on $B^0_{2,1}$.%
\footnote{In \cite{JJ05DTL} this was
sharpened in an optimal way to continuity $F^0_{p,1}\to L_p$, where the
Lizorkin--Triebel space $F^0_{p,1}$ fulfils $B^0_{p,1}\subset
F^0_{p,1}\subset L_p$ with strict inclusions for $1<p<\infty $.} 

In their fundamental paper on the $T1$-theorem G.~David and J.-L.~Journ\'e
\cite{DaJo84} concluded that $T=a(x,D)\in \OP(S^0_{1,1})$ is
bounded on $L_2$ if and only if $T^*(1)\in \op{BMO}(\Rn)$, the space of
functions (modulo constants) of bounded mean oscillation. 
(Formally this condition is weaker than G.~Bourdaud's $T^*\in \OP(S^0_{1,1})$; 
but none of these are expressed in terms of the symbol.)
Inspired by this, G.~Bourdaud \cite{Bou88} noted that certain singular
integral operators and hence every
$a(x,D)\in \OP(S^0_{1,1})$ extends to a map $\cal O_M(\Rn)\to\cal D'(\Rn)$,
where $\cal O_M$ denotes the space of $C^\infty $-functions of polynomial
growth.%
\footnote{In \cite[Thm.~2.6]{JJ10tmp} this was generalised to a map from the
maximal space of smooth functions, more precisely to a map 
$C^\infty \bigcap\cal S'\to C^\infty $ that moreover leaves $\cal O_M$
invariant.} 

Concerning $L_p$-estimates,
T.~Runst \cite{Run85ex} treated continuity 
in the more general Besov spaces $B^{s}_{p,q}$ for $p\in \,]0,\infty]$ 
and in Lizorkin--Triebel spaces $F^{s}_{p,q}$ for $p\in \,]0,\infty[\,$,
although the necessary control of the frequency changes created by $a(x,D)$
was not quite achieved in \cite{Run85ex}.%
\footnote{This flaw was explained and remedied in \cite[Rem.~5.1]{JJ05DTL}
and supplemented by $F^{s}_{p,q}$ and $B^{s}_{p,q}$ continuity
results for operators fulfilling L.~H{\"o}rmander's twisted diagonal
condition; with a further extension to operators in the self-adjoint
subclass $\OP(\tilde S^d_{1,1})$ to follow in \cite{JJ10tmp}.}
J.~Marschall \cite{Mar91} worked on further generalisations to the weighted,
anisotropic cases.%
\footnote{\cite{Mar91} contains flaws similar to \cite{Run85ex} as explained in
\cite[Rem.~5.1]{JJ05DTL}; \cite[Rem.~4.2]{JJ05DTL} also pertains to
\cite{Mar91}.}

L.~H{\"o}rmander treated type $1,1$-operators four times, first 
in lecture notes \cite{H87} from
University of Lund (1986--87); the results appeared in \cite{H88}
with important improvements in \cite{H89} the year after. 
When the notes were published after a decade \cite{H97}, the chapter
on type $1,1$-operators was rewritten with a new presentation including
the results from \cite{H89} and a few additional conclusions.

He sharpened G.~Bourdaud's analysis of $a_\theta(x,D)$
by proving that continuity $H^s\to\cal D'$ for $s\le0$
only holds if $s> -r$ where $r$ is the order of the zero
of the auxiliary function $A$ at the point $\theta$ on the unit sphere.%
\footnote{$a_\theta(x,D)$ can moreover be taken \emph{unclosable} in $\cal S'$, 
cf \cite[Sect.~3.1]{JJ08vfm}, where it was also shown that extension to $d\in
\R$ and $\theta\ne0$ was useful for a precise version of the
non-preservation of wavefront sets observed in \cite{PaRo78}.}

Moreover, L.~H{\"o}rmander characterised the $s\in \R$ (except for a limit
point $s_0$) for which a given $a(x,D)\in \OP(S^d_{1,1})$ 
extends by continuity to a bounded operator $H^{s+d}\to H^s$.
More precisely he obtained a largest interval $\,]s_0,\infty [\,\ni s$
together with constants $C_s$ such that
\begin{equation}
  \nrm{a(x,D)u}{H^s}\le C_s\nrm{u}{H^{s+d}}\quad\text{for all}\quad
  u\in \cal S(\Rn);
  \label{Hest-eq}
\end{equation}
and conversely that existence of such a $C_s$ implies $s\ge s_0$.

In order to give conditions in terms of the symbols, L.~H{\"o}rmander
introduced,
as a novelty in the analysis of pseudo-differential operators, 
the \emph{twisted diagonal}
\begin{equation}
  \cal T=\{\,(\xi,\eta)\in \Rn\times\Rn\mid \xi+\eta=0\,\}.
\end{equation} 
This was shown to play an important role, for if
eg the partially Fourier transformed symbol
$\hat a(\xi,\eta):=\cal F_{x\to\xi}a(x,\eta)$ 
vanishes in a conical neighbourhood of a non-compact part of $\cal T$, that
is, if  for some $B\ge 1$,
\begin{equation}
 B (|\xi+\eta|+1)< |\eta|\implies \hat a(x,\eta)=0,
  \label{tdc-cnd}
\end{equation}
then $a(x,D)\colon H^{s+d}\to H^s$ is continuous for
every $s\in \R$ (ie $s_0=-\infty $). 

Moreover, continuity for all $s>s_0$ was shown in \cite{H89} 
to be equivalent to the twisted diagonal condition of order $\sigma=s_0$,
which is a specific asymptotic behaviour of 
$\hat a(\xi,\eta)$ at $\cal T$. 
This is  formulated in the style of the fundamental
Mihlin--H{\"o}rmander multiplier theorem: there is a constant
$c_{\alpha,\sigma}$  such that for $0<\varepsilon<1$,
\begin{equation}
    \sup_{R>0,\;x\in \Rn}R^{-d}\big(
  \int_{R\le |\eta|\le 2R} |R^{|\alpha|}D^\alpha_{\eta}a_{\chi,\varepsilon}
  (x,\eta)|^2\,\frac{d\eta}{R^n}
  \big)^{1/2}
  \le c_{\alpha,\sigma} \varepsilon^{\sigma+n/2-|\alpha|}.
  \label{wtdc-cnd}
\end{equation}
Hereby $a_{\chi,\varepsilon}(x,\eta)$ denotes a specific 
localisation of $a(x,\eta)$ to a conical neighbourhood of $\cal T$.
Cf Section~\ref{wtdc-ssect} below. 

L.~H{\"o}rmander also characterised the case
$s_0=-\infty $ as the one with symbol in the class
$\tilde S^d_{1,1}$ and as the one where \eqref{wtdc-cnd} holds for
all $\sigma\in \R$; 
roughly speaking such symbols vanish to infinite order at $\cal T$.
A concise presentation was given in \cite[Thm.~9.4.2]{H97}.

For operators  with additional properties, a symbolic calculus was also
developed together with microlocal regularity results at non-characteristic
points as well as a sharp G{\aa}rding inequality. Although important
for the general theory of type $1,1$-operators, this is, however an area
adjacent to the present one. 
So is Chapter~10--11 in \cite{H87,H97}
where the paradifferential calculus, linearisation and propagation of
singularities of J.-M.~Bony \cite{Bon} is exposed with consistent use of
type $1,1$-operators. 
(A partly similar approach was used by M.~Taylor \cite{Tay91} and in the
treatment  of P.~Auscher and M.~Taylor \cite{AuTa95}
of commutator estimates by paradifferential operators.)

Shortly after \cite{H88,H89},
R.~Torres \cite{Tor90} also estimated $a(x,D)u$ for $u\in \cal S(\Rn)$,
using the atoms and molecules of 
M.~Frazier and B.~Jawerth \cite{FJ1,FJ2}.
This gave unique extensions by continuity to maps 
$A\colon F^{s+d}_{p,q}(\Rn)\to F^{s}_{p,q}(\Rn)$ for all $s$ so large that, 
for all multiindices $\gamma$,
\begin{equation}
  0\le|\gamma|<\max(0,\frac np-n,\frac nq-n)-s
  \implies \cal F(a(x,D)^*x^\gamma) \in \cal E'(\Rn).
\end{equation}
Obviously this refers
to the adjoint $a(x,D)^*\colon \cal S'\to \cal S'$, which in general
is an even less understood operator than those of type $1,1$.
However, as noted in \cite{Tor90}, this implies vanishing of $D^\gamma_\xi\hat
a(\xi,-\xi)$ for large $\xi$ if the symbol has compact support in $x$. 
L.~Grafakos and R.~Torres \cite{GrTo99} made a similar study in
corresponding homogeneous Besov and Lizorkin--Triebel spaces, using symbols
in the homogeneous symbol class $\dot S^d_{1,1}$, defined by removing
``$1+$'' from \eqref{Srd-eq} for 
$a(x,\eta)\in C^\infty (\Rn\times (\Rn\setminus \{0\}))$.

G.~Garello \cite{Gar94,Gar98} worked on an anisotropic version of the 
results in
\cite{PaRo78,H88,H89} for locally estimated symbols, although with
flawed arguments for the non-preservation of wavefront sets.%
\footnote{This was noted in \cite[p.~214]{JJ08vfm}.} 

A.~Boulkhemair \cite{Blk95,Blk99} worked (in a general context) 
on the use of symbols $a\in S^d_{1,1}$ in the Weyl calculus, ie in
$\op{Op}^W(a)=(2\pi)^{-n}\iint e^{\im(x-y)\cdot \eta}a(\tfrac{x+y}{2},\eta)
u(y)\,dy\,d\eta$. 
It was shown for Ching's symbol $a_\theta$ with $d=0$, cf \eqref{Ching-eq}, 
that when $A(\theta)=1$ also the operator 
$\op{Op}^W(a_\theta)$ is bounded on $H^s$ if and only if $s>0$.
In addition it was observed that
Weyl operators are worse for type $1,1$-symbols 
since certain $b(x,D)\in \op{Op}^W(S^0_{1,1})$
are unbounded on $H^s$ for \emph{every}
$s\in\R$; as noted with credit to J.~M.~Bony, this results for 
$b=\Re a_{\theta}$ or $b=\Im a_\theta$ 
because $\op{Op}^W(b)^*=\op{Op}^W(\bar b)$.
Condition \eqref{tdc-cnd} was shown to
split into two similar conditions (pertaining to $\eta\pm \frac{1}{2}\xi=0$)
that give boundedness in $H^s$ for $s>0$ and $s<0$, 
hence for all $s$ when both hold.

Very recently, J.~Hounie and R.~A.~dos Santos Kapp \cite{HoSK09} utilised
atomic decompositions of the local Hardy space $h_p(\Rn)$, which identifies
with $F^0_{p,2}$ for $0<p<\infty $, to derive existence of $h_p$-bounded 
extensions of $a(x,D)$ in the self-adjoint subclass of order $d=0$ from the
$L_2$-estimates of L.~H{\"o}rmander~\cite{H89,H97}.%
\footnote{As a special case of \cite[Thm.~7.9]{JJ10tmp} 
it was shown that every $a(x,D)$
in $\OP(\tilde S^0_{1,1})$ is continuous
\begin{equation*}
  h_p(\Rn)\to F^{s'}_{p,2}(\Rn)
\end{equation*}
for every $s'<0$ if $0<p\le 1$. For $1<p<\infty$ this was also shown for $s'=0$
in \cite[Thm.~7.5]{JJ10tmp}.
}

\bigskip

The above review summarises the scientific contributions, which
resulted from the author's search in the literature for
works devoted to type $1,1$-operators. 

The review is intended to be complete, and the
contributions of the author from 2004-2009
\cite{JJ04DCR,JJ05DTL,JJ08vfm,JJ10pe,JJ10tmp} are described accordingly.

It is clear (from the review) that 
a general definition of $a(x,D)u$ for a given symbol $a\in
S^d_{1,1}(\Rn\times\Rn)$ has not been described in the previous literature.
The estimates of L.~H{\"o}rmander \cite{H88,H89}, cf \eqref{Hest-eq},  
gave a uniquely defined bounded
operator $A\colon H^{s+d}\to H^{s}$; and an extension of $A$ to 
$\bigcup_{s>s_0} H^{s+d}(\Rn)$ for some limit $s_0$ or possibly even
$s_0=-\infty$, depending on $a(x,\eta)$.
Similarly the approach of R.~Torres 
could at most define $A$ on $\bigcup F^{s}_{p,q}(\Rn)$. 

Later elementary arguments in \cite[Prop.~1]{JJ05DTL} gave that every type
$1,1$-operator is defined on $\cal F^{-1}\cal E'(\Rn)$, and even on
$C^\infty \bigcap \cal S'$.
These spaces clearly contain all polynomials 
$\sum_{|\alpha|\le k}c_\alpha x^\alpha$
that do not 
belong to $\bigcup H^s$, nor to $\bigcup F^{s}_{p,q}$.

This development therefore only emphasises the need for a unifying point of
view, that is, 
a general definition of type $1,1$-operators without reference to spaces
other than $\cal S'(\Rn)$.

\section{Application to non-linear boundary value problems}
\label{bvp-ssect}\noindent
In addition to the applications developed by Y.~Meyer \cite{Mey80,Mey81} and
J.-M.~Bony \cite{Bon},
type $1,1$-operators were recently used by the author
in the analysis of semi-linear boundary problems \cite{JJ08par}. 
More precisely, their pseudo-local property was shown to be useful for the
derivation of local regularity improvements.

To explain this, one can as a typical example 
consider a perturbed $k$-harmonic Dirichl\'et problem in a 
bounded $C^\infty $-region $\Omega\subset \Rn$,
\begin{equation}
  \begin{split}
  (\mlap)^k u+u^2&=f \quad\text{in}\quad \Omega,
\\
  \gamma_0 u&=\varphi_0\quad\text{on}\quad \partial\Omega,
\\
  &\ \:\vdots
\\
  \gamma_{k-1}u&=\varphi_{k-1}\quad\text{on}\quad \partial\Omega.
  \end{split}
\end{equation}
Here $\lap=\partial^2_{x_1}+\dots +\partial^2_{x_n}$ 
denotes the Laplacian while $\gamma_j$ stands for the normal
derivative of order $j$ at the boundary.

For such problems the parametrix construction of \cite{JJ08par} yields the
solution formula 
\begin{equation}
  u= P^{(N)}_u(R_k f +K_0\varphi_0+\dots +K_{k-1}\varphi_{k-1})
    +(R_kL_u)^Nu,
  \label{param-eq}
\end{equation}
where the parametrix $P^{(N)}_u$ is the $u$-dependent linear operator 
\begin{equation}
 P^{(N)}_u=I +R_kL_u+\dots +(R_k L_u)^{N-1}.
\end{equation}
Here it was a crucial point of \cite{JJ08par} to use the so-called
\emph{exact paralinearisation} $L_u$ of $u^2$ as a main
ingredient. In effect this means that $L_u$ 
is a localised type $1,1$-operator, as reviewed in \eqref{rAl-eq} below. 
(This is a result from \cite[Thm.~5.15]{JJ08par}, but it would lead too far
to explain its deduction from the rather technical paralinearisation.)
With a convenient sign convention $L_u$
fulfils $-L_u(u)=u^2$. 

Moreover, the other terms $R_k$, $K_0$,\dots ,$K_{k-1}$ in the formula are
the solution operators of the linear problem.%
\footnote{The operators $R_k$, $K_0$,\dots , $K_{k-1}$ can be explicitly
described in local coordinates at the boundary $\partial\Omega$. This is the
subject of the calculus of L.~Boutet de Monvel \cite{BM71} of
pseudo-differential boundary operators; it has been amply described
eg in works of G.~Grubb \cite{G2,G1,G97,G09}. The calculus was exploited in
\cite{JJ08par} but details are left out here because it would be too far
from the topic of type $1,1$-operators.}
It is perhaps instructive to reduce to the linear case by formally setting
$L_u\equiv 0$ above: this shows that the parametrix $P^{(N)}_u$ and the
remainder $(R_kL_u)^N$ simply
modify $u$ in the presence of the non-linear term.

Formula \eqref{param-eq} also has the merit of showing \emph{directly} 
that the regularity of $u$ will be
uninfluenced by the non-linear term $u^2$. Or more precisely,
$u$ will belong to the same Sobolev space $H^s_p$ as the corresponding
linear problem's solution $v$, ie
\begin{equation}
  v=R_k f+K_0\varphi_0+\dots +K_{k-1}\varphi_{k-1}.
\end{equation}
Indeed, in \eqref{param-eq} the parametrix $P^{(N)}_u$ is applied to $v$,
but it is
of order $0$ for every $N$, hence sends each Sobolev space $H^s_p$ into
itself; 
while the remainder $(R_m L_u)^{N}u$ will be
in $C^k(\overline{\Omega})\subset H^s_p(\overline{\Omega})$ 
for some fixed $k$ if $N$ is taken large enough
(in both cases because $R_k L_u$ will have negative order if the given
solution $u$ a priori meets a rather weak regularity assumption; cf
\eqref{rAl-eq} below). 
These inferences may be justified using parameter
domains as in \cite{JJ08par}, to keep track of the spaces on which various
steps are valid. 

Moreover, to explain the usefulness of type $1,1$-operators here, it is
noted that in subregions $\Xi\Subset \Omega$, extra regularity properties of
$f$ carry over to $u$. Eg, if $f|_{\Xi}$ is $C^\infty $ so is $u|_{\Xi}$. 
Other examples involve improvements in $\Xi$ of eg the Sobolev space regularity.

Such local properties can also be deduced from formula \eqref{param-eq},
because $L_u$ factors through a specific type $1,1$-operator $A_u$
(this is in itself a minor novelty, because of the boundary).
That is, when $\rOm$ denotes restriction to $\Omega$ and 
$\lOm$ stands for a linear extension operator from $\Omega$, then
\begin{equation}
  L_u=\rOm \circ A_u \circ \lOm,  \qquad A_u\in \OP(S^{d} _{1,1}); 
  \label{rAl-eq}
\end{equation}
here the order $d\ge (\tfrac{n}{p_0}-s_0)_+$ if $u$ is given in $H^{s_0}_{p_0}$,
though with strict inequality if $s_0=n/p_0$.

To exploit this, one may simply take cut-off functions $\psi,\chi\in
C^\infty_0(\Xi)$ with $\chi=1$ around $\supp\psi$.
Insertion of these into
\eqref{param-eq}, cf \cite[Thm.~7.8]{JJ08par}, gives
\begin{multline}
  \psi u= \psi P^{(N)}_u(R_k (\chi f))
+\psi P^{(N)}_u(R_k ((1-\chi) f)) 
\\
+\psi P^{(N)}_u(K_0\varphi_0+\dots +K_{k-1}\varphi_{k-1})
    +\psi (R_k L_u)^Nu.
  \label{param'-eq}
\end{multline}
As desired
$\psi u$ has the same regularity as the \emph{first} term on the
right-hand side. Indeed, the last term has the same regularity
as the first if $N$ is large, and\,---\,since
the set of pseudo-local operators is invariant under sum and composition, so
that pseudo-locality of $A_u$ by \eqref{rAl-eq} carries over to $P^{(N)}$\,---\,the disjoint
supports of $\psi$ and $1-\chi$ will imply that the second term is $C^\infty$; 
the $K_j\varphi_j$ always contribute $C^\infty$-functions in the interior, 
to which set $\psi $ localises while $P^{{(N)}}$ is pseudo-local.

Therefore the \emph{pseudo-local} property of $A_u$ will lead easily to improved
regularity of $u$ locally in $\Xi$, to the extent this is permitted by the
data $f$. Hence it was a serious drawback that the literature
had not established pseudo-locality in the $1,1$-context.

But motivated by the above application in \eqref{param'-eq},
the pseudo-local property of general type $1,1$-operators was 
proved recently by the author in \cite{JJ08vfm}. The only previous work
mentioning this subject was that of
C.~Parenti and L.~Rodino \cite{PaRo78}, who three decades ago 
anticipated the result but merely gave an incomplete argument, partly because
they did not assign a specific meaning to $a(x,D)u$ 
for $u\in \cal S'\setminus C^\infty_0$.

\section{The definition of 
type $\mathbf{1},\mathbf{1}$-operators}\noindent
As seen at the end of the last two sections, it will be well motivated to
introduce a general definition of type $1,1$ \emph{operators}.

This was
first done rigorously in \cite{JJ08vfm}, 
taking into account that in some cases they can
only be defined on proper subspaces $E\subset \cal S'(\Rn)$.
Indeed, it was proposed to stipulate that
$u$ belongs to the domain $D(a(x,D))$ and to set
\begin{equation}
  a(x,D)u:= \lim_{m\to\infty } \OP(\psi(2^{-m}D_x)a(x,\eta)\psi(2^{-m}\eta))u
  \label{a11-id}
\end{equation}
if this limit exists, say in $\cal D'(\Rn)$, for all the
$\psi\in C^\infty_0(\Rn)$ with $\psi=1$ in a neighbourhood of the origin 
and if it does not depend on such $\psi$.

The definition, its consequences and the techniques developed are discussed in
the author's contributions
\cite{JJ04DCR,JJ05DTL,JJ08vfm,JJ10pe,JJ10tmp},
where the first is an early announcement of the results in the second.
These works are summarised in Chapter~\ref{vfm-sect}. 

\chapter{Preliminaries}
  \label{prel-sect}
\section{Notions and notation}
\label{NN-ssect}\noindent
As usual $t_{\pm}=\max(0,\pm t)$ will denote the positive and negative part
of $t\in\R$; and $[t]$ will stand for the largest integer $k\in\Z$ such that
$k\le t$. The characteristic function of a set $M\subset\Rn$ is denoted
$1_M$; by $M\Subset \Rn$ it is indicated that the subset $M$ is precompact.

The Lebesgue spaces $L_p(\Rn)$ with $0<p\le \infty $ consist of the 
(equivalence classes of) measurable functions having finite (quasi-)norm
$\nrm{f}{p}=(\int_{\Rn}|f(x)|^p\,dx)^{1/p}$ for $0<p<\infty $, respectively 
$\nrm{f}{\infty }=\op{ess}\sup_{\Rn} |f|$. 

In general $\nrm{f+g}{p}\le 2^{(\fracpi-1)_+}(\nrm{f}{p}+\nrm{g}{p})$ 
for $0<p<\infty $. Hence for $0<p<1$ the map $f\mapsto
\nrm{f}{p}$ is only a quasi-norm, but it  does have a subadditive power as
$\nrm{f+g}{p}^p\le \nrm{f}{p}^p+\nrm{g}{p}^p$ for $0<p<1$.

For every multiindex $\alpha\in \N_0^n$ it is convenient to set 
$x^\alpha=x_1^{\alpha_1}\dots x_n^{\alpha_n}$ and to introduce the
differential operator 
$D^\alpha=(-\im)^{|\alpha|}\partial^{\alpha_1}_{x_1}\dots
\partial^{\alpha_n}_{x_n}$ where $|\alpha|=\alpha_1+\dots +\alpha_n$. 

The space of smooth functions with  compact support is denoted 
by $C^\infty_0(\Omega)$ or $\cal D(\Omega)$, 
when $\Omega\subset\Rn$ is open;
$\cal D'(\Omega)$ is the dual space of distributions on $\Omega$.
Throughout
$\ang{u,\varphi}$ denotes the action of $u\in\cal D'(\Omega)$
on $\varphi\in C^\infty_0(\Omega)$. Therefore  $\dual{\cdot }{\cdot }$ is a
bilinear form; the sesquilinear form $\scal{\cdot }{\cdot }$ is used for
the action of conjugate linear functionals on $C^\infty_0$ and $\cal S$, 
consistently with the \emph{inner}
product on the Hilbert space $L_2(\Rn)$ (both $\dual{\cdot }{\cdot }$ and
$\scal{\cdot }{\cdot }$ are called scalar products for convenience). 

The space of slowly increasing functions, ie $C^\infty$-functions $f$
fulfilling $|D^\alpha f(x)|\le c_\alpha\ang{x}^{N_\alpha}$ for all
mulitindices $\alpha$ is written $\cal O_M(\Rn)$;
hereby $\ang{x}=(1+|x|^2)^{1/2}$.

The Schwartz space of rapidly decreasing $C^\infty$-functions 
is written $\cal S$ or $\cal S(\Rn)$, while its dual space $\cal S'(\Rn)$ 
constitutes the space of tempered distributions.
The Fourier transformation of $u$ is denoted by 
$\cal Fu(\xi)=\hat u(\xi)=\int_{\Rn}
e^{-ix\cdot\xi}u(x)\,dx$, with inverse  $\cal F^{-1}v(x)=\check v(x)$. 

The subspace $\cal E'(\Rn)$ consists of the distributions of compact
support; it is the dual of $C^\infty (\Rn)$. The \emph{spectrum} of $u\in
\cal S'$ is by definition $\supp \cal Fu$; hence $\cal F^{-1}(\cal E')$ is the
space af distributions with compact spectrum (though it equals $\cal F(\cal
E')$ as a set, the slightly more pedantic $\cal F^{-1}\cal E'$ is preferred
to emphasize the role of the Fourier transformation).

Pseudo-differential operators are given on $\cal S(\Rn)$
by \eqref{axDu-eq}, with symbols
fulfilling \eqref{Srd-eq}. On
$S^d_{\rho,\delta}=S^d_{\rho,\delta}(\Rn\times \Rn)$ the Frech\'et
topology is defined by a family of seminorms $p_{\alpha,\beta}(a)$, that are
given as the smallest possible constants $C_{\alpha,\beta}$ in
\eqref{Srd-eq}. For short $S^\infty _{\rho,\delta}:=\bigcup_{d\in\R}
S^d_{\rho,\delta}$ is used for the set of all symbols (of type $\rho,\delta$).
The symbol class $S^{-\infty }:=
\bigcap_{d} S^d_{1,0}=\bigcap_{d,\rho,\delta}S^d_{\rho,\delta}$ defines the
smoothing operators; they are
bounded $H^s\to H^t$ for all $s,t\in\R$.

The pseudo-differential operators $a(x,D)$ are in
bijective correspondence with their distribution kernels, that are given by 
\begin{equation}
  K(x,y)=\cal F^{-1}_{\eta\to x-y}a(x,\eta).  
\end{equation}
By definition the kernel satisfies the kernel relation
\begin{equation}
  \dual{a(x,D)\psi}{\varphi}=\dual{K}{\varphi\otimes \psi}
\quad\text{for all}\quad
\varphi,\psi\in C^\infty_0(\Rn)  .
\end{equation}
As customary, the support
$\supp K\subset \Rn\times \Rn$ is seen as a relation mapping sets in $\Rn_y$
to other sets in $\Rn_x$. More precisely, each subset $M\subset \Rn_y$ is
mapped to
\begin{equation}
 \supp K\circ M=\{\,x\in \Rn\mid \exists y\in M\colon (x,y)\in \supp K\,\}.  
\end{equation}

The singular support of $u\in \cal D'$, denoted $\singsupp u$, is the
complement of the largest open set on which $u$ acts a $C^\infty$-function.
The wavefront set $\op{WF}(u)$ is the complement of those $(x,\xi)\in
\Rn\times (\Rn\setminus \{0\})$ for which $\cal F(\varphi u)$ decays rapidly
in a conical neighbourhood of $\xi$ for some $\varphi\in C^\infty_0$ for
which $\varphi(x)\ne0$.

Every pseudo-differential operator considered here  is continuous 
$a(x,D)\colon \cal S(\Rn)\to\cal S(\Rn)$, hence
has a continuous adjoint $a(x,D)^*\colon \cal S'(\Rn)\to\cal S'(\Rn)$ with
respect to the scalar product $\scal{\cdot }{\cdot }$; 
this fulfils
\begin{equation}
  \scal{a(x,D)^*\varphi}{\psi}=\scal{\varphi}{a(x,D)\psi},
  \qquad \varphi,\psi\in \cal S(\Rn).
\end{equation}
Its restriction $a(x,D)^*\colon \cal S(\Rn)\to\cal S'(\Rn)$ is also
continuous, hence is a pseudo-differential operator by Schwartz' kernel
theorem; cf \cite[18.1]{H}. More precisely,
\begin{equation}
  a(x,D)^*=\OP(b(x,\eta))\quad\text{for}\quad
  b(x,\eta)=e^{\im D_x\cdot D_\eta}\bar a(x,\eta).
\end{equation}
The adjoint symbol $e^{\im D_x\cdot D_\eta}\bar a(x,\eta)$ is also written 
$a^*(x,\eta)$, so $\OP(a(x,\eta))^*=\OP(a^*(x,\eta))$.

\section{Scales of function spaces}\noindent
The Sobolev spaces $H^s_p(\Rn)$ are defined for $s\in \R$ and $1<p<\infty $
as $\OP(\ang{\xi}^{-s})(L_p)$, with
$\nrm{f}{H^s_p}=\nrm{\OP(\ang{\xi}^{-s})f}{p}$. The special case $p=2$ is
written as $H^s(\Rn)$ or $H^s$ for simplicity.

The H{\"o}lder class $C^s(\Rn)$ is for non-integer $s>0$ defined as
the functions $f\in C^{[s]}(\Rn)$ having finite norm
\begin{equation}
  |f|_s=\sum_{|\alpha|\le [s]}\nrm{D^\alpha f}{\infty }+
        \sum_{|\alpha|=[s]}\sup_{x\ne y}|D^\alpha f(x)-D^\alpha f(y)|
        |x-y|^{[s]-s}.
\end{equation}
To get an interpolation invariant half-scale $C^s_*(\Rn)$, $s>0$, 
it is well known
that one should fill in for $s\in \N$ by means of the Zygmund condition. Eg
the space $C^1_*$ consists of the $f\in C(\Rn)\cap L_\infty (\Rn)$ for which
\begin{equation}
  |f|_1=\nrm{f}{\infty }+\sup_{y\ne0}
  \sup_{x\in \Rn}|f(x+y)+f(x-y)-2f(x)|/|y| <\infty .
\end{equation}
These spaces appear naturally as a part of a full scale of
H{\"o}lder--Zygmund spaces $C^s_*(\Rn)$ defined for $s\in\R$; as explained
in eg \cite[Sc.~8.6]{H97}.

However, all the $H^s_p$ and $C^s_*$ spaces are contained in two more
general scales, namely the 
Besov spaces $B^{s}_{p,q}(\Rn)$ and Lizorkin--Triebel spaces
$F^{s}_{p,q}(\Rn)$, that are well adapted to harmonic analysis.
They are recalled below.

First a Littlewood--Paley decomposition is constructed using a
function $\tilde \Psi$ in $C^\infty(\R)$ for which $\tilde \Psi(t)\equiv0$ 
and $\tilde \Psi(t)\equiv1$
holds for $t\ge 2$ and $t\le 1$, respectively;
then $\Psi(\xi)=\tilde \Psi(|\xi|)$ and 
$\Phi=\Psi-\Psi(2\cdot )$ gives the partition of unity 
$1=\Psi(\xi)+\sum_{j=1}^\infty \Phi(2^{-j}\xi )$.
For brevity it is here convenient to set 
$\Phi_0=\Psi$ and $\Phi_j=\Phi(2^{-j}\cdot )$ for $j\ge 1$.

Then, for a {\em smoothness indices $s\in\R$}, 
\emph{integral-exponent $p\in\left]0,\infty\right]$} 
and  \emph{sum-exponent} $q\in\left]0,\infty\right]$, 
the {Besov} space $B^{s}_{p,q}(\Rn)$ 
is defined to consist of the $u\in \cal S'(\Rn)$ for which
\begin{equation}
 \Nrm{u}{B^{s}_{p,q}} := 
 \big(\sum_{j=0}^\infty 2^{sjq} 
  (\int_{\Rn}|\Phi_j(D)u(x)|^p\,dx)^{\tfrac{q}{p}} 
  \big)^{\frac1q} <\infty.
 \label{bspq-id} 
\end{equation}
(As usual the norm in $\ell_q$ should be replaced by the supremum over $j\in
\N_0$ in case $q=\infty $.)

Similarly the {Lizorkin--Triebel} space $F^{s}_{p,q}(\Rn)$ is defined as the
$u\in \cal S'(\Rn)$ such that
\begin{equation}
  \Nrm{u}{F^{s}_{p,q}}:=
  \big (\int_{\Rn}
  (\sum_{j=0}^\infty 2^{sjq} 
  |\Phi_j(D)u(x)|^q)^{\fracci pq}\,dx
  \big)^{\frac 1p} <\infty.
 \label{fspq-id}
\end{equation}
Throughout it will be tacitly understood that $p<\infty$ whenever
Lizorkin--Triebel spaces are under consideration.

The spaces are described in eg \cite{RuSi96,T2,T3,Y1}. 
They are quasi-Banach
spaces with the quasi-norms  given by the finite expressions in
\eqref{bspq-id} and \eqref{fspq-id}; and Banach spaces if both $p\ge 1$ and
$q\ge 1$.

In general $u\mapsto \nrm{u}{}^{\lambda}$ is subadditive for 
$\lambda\le\min(1,p,q)$, so $\nrm{f-g}{}^\lambda$ is a metric on each space
in the $B^{s}_{p,q}$- and $F^{s}_{p,q}$-scales.

There are a number of embeddings of these spaces, like the simple ones
$F^{s}_{p,\infty}\hookrightarrow F^{s-\varepsilon}_{p,q}$ for $\varepsilon>0$ and
$F^{s}_{p,q}\hookrightarrow F^{s}_{p,r}$ for $q\le r$. The Sobolev embedding
theorem takes the form
\begin{equation}
  F^{s_0}_{p_0,q_0}\hookrightarrow F^{s_1}_{p_1,q_1}
  \quad\text{for}\quad s_0-\fracc n{p_0}=s_1-\fracc n{p_1},\  
  p_0<p_1.
\end{equation}
The analogous results are valid for the $B^{s}_{p,q}$ spaces, provided that
$q_0\le q_1$. Moreover,
\begin{equation}
  B^{s}_{p,\min(p,q)}\hookrightarrow F^{s}_{p,q}
  \hookrightarrow B^{s}_{p,\max(p,q)}.
\end{equation}

Among the well-known identifications it should be mentioned that
\begin{gather}
  H^s_p=F^s_{p,2}\quad\text{for}\quad s\in\R,\ 1<p<\infty,
\label{Hsp-id} \\
  C^s_*=B^s_{\infty ,\infty }\quad\text{for}\quad s\in\R.
  \label{Cs*-id}
\end{gather}
In particular this means that
\begin{equation}
    H^s=F^s_{2,2}=B^s_{2,2}\quad\text{for}\quad s\in\R.
\end{equation}
One interest of this is that statements proved for all $B^{s}_{p,q}$
are automatically valid for the Sobolev spaces $H^s$ by specialising to
$p=q=2$, as well as for the H{\"o}lder-Zygmund spaces $C^s_{*}$ by setting 
$p=q=\infty$.
(Much of the literature on partial differential equations has focused on
these two scales, with two rather different types of arguments.)

Among the other relations, it could be mentioned that 
$F^0_{p,2}(\Rn)$ equals the local Hardy space $h_p(\Rn)$ for $0<p<\infty $.
\cite{T3} has ample information on these identifications, and also on the
extension of $F^{s}_{p,q}$ to $p=\infty $; this is not considered here.

\begin{rem}   \label{BF-rem}
The quasi-norms of $B^{s}_{p,q}$ and $F^{s}_{p,q}$ depend of course on the
choice of the Littlewood--Paley decomposition; cf \eqref{bspq-id} and
\eqref{fspq-id}. It is well known that different choices yield equivalent
quasi-norms, which may be seen with a multiplier argument. However, a slight
extension of this shows that the above assumption on $\tilde \Psi(t)$ can be
completely weakened, that is, any $\tilde \Psi\in C^\infty_0(\R)$ equalling
$1$ around $t=0$ will lead
to an equivalent quasi-norm (cf the framework for
Littlewood--Paley decompositions in Section~\ref{LPa-ssect} below).
This is convenient for the treatment of type $1,1$-operators in
$B^{s}_{p,q}$ and $F^{s}_{p,q}$ spaces.
\end{rem}

\chapter{The general definition of type 
$\mathbf{1},\mathbf{1}$-operators}
\label{vfm-sect}\noindent
This section gives a brief description of the author's contributions;
for the sake of readability, the statements will occasionally only address
the main cases.  
A more detailed account can be found in the subsequent sections 
(and in the papers, of course). 

\section{Definition by vanishing frequency modulation}\noindent
As the background for Definition~\ref{vfm-defn} below, 
it is recalled that the very first result on
type $1,1$-operators was the counter-example by C.~H.~Ching \cite{Chi72}, who
showed that there exists $a_\theta(x,\eta)$ in $S^0_{1,1}$, cf
\eqref{Ching-eq}, for which the operator $a_\theta(x,D)$ does not have a
continuous extension to $L_2$. 

For later reference, this is now explicated with a refined version of order $d$.
\begin{lem}[{\cite[Lem.~3.2]{JJ08vfm}}]
  \label{cex-lem}
Let $a_{\theta}(x,\eta)$
be given as in \eqref{Ching-eq} for $d\in \R$ and with $|\theta|=1$ 
and $A=1$ on the ball $B(\theta,\tfrac{1}{10})$.
Taking $v\in \cal S(\Rn)$ with $\emptyset\ne\supp\hat v\subset
B(0,\tfrac{1}{20})$, then 
\begin{equation}
  v_N=v(x)\sum_{j=N}^{N^2} \frac{e^{\im 2^jx\cdot\theta}}{j2^{jd}\log N}  
\end{equation}
defines a sequence of Schwartz functions with the properties 
\begin{equation}
  \begin{split}
    \nrm{v_N}{H^d} &
  \le c\nrm{v}{2}(\sum_{j=N}^\infty j^{-2})^{1/2}   \searrow 0,
\\
  a_{\theta}(x,D)v_N(x)&=\tfrac{1}{\log N}
  (\tfrac{1}{N}+\tfrac{1}{N+1}+\dots+\tfrac{1}{N^2})v(x) 
  \xrightarrow[N\to\infty]{~} v(x) \quad\quad\text{in $\cal S(\Rn)$}.
  \end{split}
\end{equation}
Consequently $a_{\theta}(x,D)$ is unbounded $H^d\to L_2$ and 
unclosable in $\cal S'(\Rn)\times \cal D'(\Rn)$. 
\end{lem}

Later in 1983, G.~Bourdaud \cite{Bou83} showed
in his doctoral dissertation that every $a(x,D)\in \OP(S^0_{1,1})$ is
bounded on $L_2(\Rn)$ if also its adjoint $a(x,D)^*$ is of type $1,1$.
Hence $a_\theta(x,D)$ above fulfils $a_{\theta}(x,D)^*\notin
\OP(S^0_{1,1})$, so this adjoint need not send $\cal S(\Rn)$ into itself.

This has two important consequences: first of all, while $a(x,D)$ as usual
does have the ``double'' adjoint $a^*(x,D)^*=\OP(a^*(x,\eta))^*$ as an
extension, the latter is not necessarily defined on the entire space $\cal
S'(\Rn)$ when $a(x,\eta)$ is of type $1,1$.
In fact,
already for $a_\theta(x,D)^*$ it can be shown explicitly that its image of
$\cal S(\Rn)$ contains functions in $\cal S'\setminus \cal S$ (see eg
\cite[(3.4),(3.9)]{JJ08vfm}), whence 
$a^*(x,D)^*$ is defined on a proper subspace of $\cal S'$. 

Secondly, if one tries to see
$u\in \cal S'(\Rn)$ as a limit $u=\lim_{k\to\infty }u_k$ for Schwartz
functions $u_k$, one cannot hope to get a useful definition by setting
\begin{equation}
  a(x,D)u=\lim_{k\to\infty }\OP(a)u_k.
  \label{au_k-eq}
\end{equation}
Indeed, this would not always give a linear operator, as 
$a_\theta(x,D)$ is \emph{unclosable};  cf Lemma~\ref{cex-lem}. This is obviously
important also because it shows that a type $1,1$-operator cannot be given an
extended definition just by closing its graph $G(a(x,D))$ as a subset of 
$\cal S'\times \cal D'$\,---\,and nor can one hope
to give a definition by other means and obtain a closed operator in general.

In view of this, and especially in comparison with \eqref{au_k-eq}, 
it is perhaps not surprising that 
\cite{JJ08vfm} proposes a regularisation of the symbol instead:
\begin{equation}
  a(x,D)u(x)=\lim_{m\to\infty }\OP(b_m(x,\eta))u(x).
  \label{axDu-id}
\end{equation}
However, the precise choice of the approximating symbol 
$b_m(x,\eta)$ is decisive here.

To prepare for the formal definition, 
a \emph{modulation} function $\psi$ will in the sequel mean an arbitrary 
$\psi\in C^\infty_0(\Rn)$ 
equal to $1$ in a neighbourhood of the origin. Then, after setting
$\hat a(\xi,\eta)=\cal F_{x\to\xi} a(x,\eta)$ for symbols, 
the following notation is used throughout
\begin{equation}
  a^m(x,\eta)= \cal F^{-1}_{\xi\to x}[\psi(2^{-m}\xi)\hat a(\xi,\eta)].
  \label{am-id}
\end{equation}
One can then take $b_m(x,\eta)=a^m(x,\eta)\psi(2^{-m}\eta)$, which is in
$S^{-\infty }$, so that $b_m(x,D)u$ is  defined for every $u\in \cal S'$. 
It is easy to see that if $a\in S^d_{1,1}$ then
$b_m\to a$ in $S^{d+1}_{1,1}$ 
for $m\to\infty$; cf \cite[Lem.~2.1]{JJ08vfm}.
 
To make the dependence on $\psi$ explicit, set
\begin{equation}
  a_{\psi}(x,D)u=\lim_{m\to\infty } \OP(a^m(x,\eta)\psi(2^{-m}\eta))u.
  \label{apsi-eq}
\end{equation}

\begin{defn}
  \label{vfm-defn}
For every symbol $a\in S^{d}_{1,1}(\Rn\times \Rn)$ the distribution 
$u\in \cal S'(\Rn)$ belongs to the
domain $D(a(x,D))$ if the above limit $a_{\psi}(x,D)u$ exists in $\cal
D'(\Rn)$ for every modulation function $\psi$ and if, in addition, 
this limit is independent of such $\psi$. In this case
\begin{equation}
  a(x,D)u=a_{\psi}(x,D)u.
\end{equation}
\end{defn}

In \cite{JJ08vfm} this was termed the definition of $a(x,D)$ by 
\emph{vanishing frequency modulation}, since all high frequencies are cut
off, both in $u(y)$ and in the symbol's dependence on $x$.

To explain the notation, note first that $D$ appears in two meanings when
the domain is denoted by $D(a(x,D))$. 
Moreover, \eqref{axDu-eq} may be written out as
\begin{equation}
    a(x,D)u(x)= (2\pi)^{-n}\int_{\Rn}\int_{\Rn} e^{\im (x-y)\cdot \eta} 
  a(x,\eta) u(y)\,dy\,d\eta.
\end{equation}
Here $u$ is seen as a function of $y$; accordingly the dual variable is
denoted by $\eta$. Clearly $a(x,D)u(x)$  depends on
$x$, whence its Fourier transform is written as a function of
$\xi\in \Rn$. Likewise, when $\cal F_{x\to\xi}$ is applied to $a(x,\eta)$,
one obtains $\hat a(\xi,\eta)$.

The modulation parameter is throughout denoted by $m\in \N$. The modulation
function is denoted by $\psi$, or $\Psi$ if more than one is considered
simultaneously. Moreover, with $u^m=\psi(2^{-m}D)u$ and $a^m(x,\eta)$ as
defined above, Definition~\ref{vfm-defn} is for
convenience often expressed in short form as
\begin{equation}
  a(x,D)u=\lim_{m\to\infty } a^m(x,D)u^m.
  \label{amum-eq}
\end{equation} 
This may look self-contradicting, however, for $a^m(x,D)$ is
just another type $1,1$-operator. But as $u^m\in \cal F^{-1}\cal E'(\Rn)$,
it will be clear below (from the general extension to $\cal F^{-1}\cal E'$)
that $a^m(x,D)u^m$ is defined and equals $\OP(a^m(x,\eta)\psi(2^{-m}\eta))u$.

Definition~\ref{vfm-defn} 
is actually just a rewriting of the usual one, which is suitable 
for type $1,1$-symbols as a point of departure, for if $u\in \cal S$ it
follows from the continuity in \eqref{SdS-eq} that
$a(x,D)u=\OP(a(x,\eta))u$. It also gives back the usual operator
$\OP(a(x,\eta))u$ on $\cal S'$ whenever $a\in S^d_{1,0}$, for it is well
known that this is equal to the limit $a_{\psi}(x,D)u$.

Formally Definition~\ref{vfm-defn} is reminiscent of oscillatory integrals,
as exposed by for example 
X.~St.-Raymond \cite{SRay91}, now with the natural proviso (as $\delta=1$)
that $u\in D(a(x,D))$ when the regularisation yields a limit independent of
the integration factor.

Of course, $a(\cdot ,\eta)$ is not modified here with
an integration factor proper, but rather with the Fourier
multiplier $\psi(2^{-m}D_x)$.
This obvious difference is emphasized because $\psi(2^{-m}D_x)$
later gives 
easy access to Littlewood--Paley analysis of $a(x,D)$.

For other remarks on the feasibility of the frequency modulation, in
particular the relation to pointwise multiplication, the reader
may refer to \cite[Sect.~1.2]{JJ08vfm}.

\section{Consequences for type $\mathbf{1},\mathbf{1}$-operators}
\label{dozen-ssect}\noindent
Although
the definition by vanishing frequency modulation is rather unusual
(which is unavoidable), 
it does have a dozen important properties:
\begin{Rmlist}  \addtolength{\itemsep}{1\jot}
  \item   \label{uni-pro}
Definition~\ref{vfm-defn} unifies
4 previous extensions of type $1,1$-operators.
  \item   \label{max-pro} The resulting densely defined map
$a(x,D)\colon \cal S'(\Rn)\to \cal D'(\Rn)$ 
is \emph{maximal} among the extensions $\widetilde{\OP}(a(x,\eta))$ 
that are stable under vanishing frequency modulation as well as
compatible with $\OP(S^{-\infty })$.
  \item   \label{C8S'-pro}
Every operator $a(x,D)$ of type $1,1$ restricts to a map
\begin{equation}
  a(x,D)\colon C^\infty (\Rn)\bigcap\cal S'(\Rn)\to C^\infty (\Rn),
  \label{C8S'-eq}
\end{equation}
where $C^\infty \bigcap\cal S'$ is the maximal subspace of smooth
functions. Moreover, $\cal O_M(\Rn)$ is invariant under $a(x,D)$.
  \item   \label{psloc-pro}
Every operator $a(x,D)$ of type $1,1$ is pseudo-local.
  \item   \label{nwf-pro} 
Some type $1,1$-operators do not preserve wavefront sets, eg
\eqref{Ching-eq} gives
\begin{align}
  \op{WF}(u)&= \Rn\times \R_+\theta
\\
  \op{WF}(a_{2\theta}(x,D)u)&= \Rn\times \R_+(-\theta)
\end{align}
for $|\theta|=1$, $A(\eta)=1$ around $\eta=\theta$ 
and a product $u(x)=v(x)f(\theta\cdot x)$
with a suitable $v\in \cal F^{-1}C^\infty_0$ and an
oscillating factor $f(t)=\sum_{j=0}^\infty 2^{-jd}e^{\im 2^jt}$, 
which for $0<d\le 1$ is  
Weierstrass's continuous nowhere differentiable function.
  \item   \label{supp-pro}
The operators satisfy the \emph{support rule}, respectively the
\emph{spectral} support rule,
\begin{align}
  \supp a(x,D)u&\subset \supp K\circ \supp u,
  \label{sKau-eq}
\\
  \supp \cal F a(x,D)u&\subset \supp \cal K\circ \supp\cal F u,
  \label{ssKau-eq}
\end{align}
where $K$ is the distribution kernel of $a(x,D)$, 
whereas $\cal K$ is that of $\cal Fa(x,D)\cal F^{-1}$.
  \item   \label{LPa-pro}
The auxiliary function $\psi$ in Definition~\ref{vfm-defn} allows 
a direct transition to Littlewood--Paley analysis of $a(x,D)u$, which in
particular gives the well-known paradifferential decomposition, cf
\eqref{a123-eq}, 
\begin{equation}
  a(x,D)u=a^{(1)}(x,D)u+a^{(2)}(x,D)u+a^{(3)}(x,D)u.
\end{equation}
  \item  \label{tdc-pro} 
The operator $a(x,D)$ is everywhere defined and continuous 
\begin{equation}
  a(x,D)\colon \cal S'(\Rn)\to \cal S'(\Rn)
  \label{aS'-eq}
\end{equation}
if $a(x,\eta)$ satisfies H{\"o}rmander's 
\emph{twisted diagonal condition;}
ie, if for some $B\ge 1$
\begin{equation}
  \hat a(\xi,\eta)=0\quad\text{whenever}\quad
  B(|\xi+\eta|+1)< |\eta|.
  \label{tdc-id}
\end{equation}
  \item   \label{wtdc-pro}
The continuity in \eqref{aS'-eq} more generally holds 
in the self-adjoint subclass $\OP(\tilde S^\infty _{1,1})$, 
ie if $a(x,D)$ fulfils H{\"o}rmander's
twisted diagonal condition of order $\sigma$ for every $\sigma\in \R$.
  \item  \label{dom-pro} 
Every $a(x,D)$ of order $d$ is for $p\in [1,\infty [\,$ and $q\le 1$ a
continuous map  
\begin{equation}
  a(x,D)\colon F^{d}_{p,q}(\Rn)  \to L_p(\Rn);
\end{equation}
for $a_\theta(x,D)$ 
from \eqref{Ching-eq} this is optimal
within the scales $B^{s}_{p,q}$ and  
$F^{s}_{p,q}$ of Besov and Lizorkin--Triebel spaces. 
(These contain $C^s$ and $H^s_p$,  respectively.)
  \item   \label{Fspq-pro}
Every $a(x,D)$ in $\OP(S^d_{1,1})$ is continuous, for $s>\max(0,\fracnp-n)$,
$0<p< \infty $, $0<q\le \infty $,
\begin{equation}
  a(x,D)\colon  F^{s+d}_{p,q}(\Rn)\to F^{s}_{p,r}(\Rn)
\quad\text{if}\quad r\ge q,\ r>n/(n+s).
\end{equation}
This holds for all $s\in \R$ and $r=q$ when $a(x,\eta)$ fulfils the twisted
diagonal condition \eqref{tdc-id}, and if $p>1$, $q>1$ also when
$a(x,\eta)\in \tilde S^d_{1,1}(\Rn\times \Rn)$.

These properties extend 
to the scale $B^{s}_{p,q}(\Rn)$ for $0<p\le \infty $, $r=q$.

  \item  \label{Fspq'-pro}
When $a(x,\eta)$ is in $\tilde S^d_{1,1}(\Rn\times \Rn)$,
cf \eqref{wtdc-pro}, and $0<p\le 1 $, $0<q\le \infty $,
\begin{equation}
  a(x,D)\colon F^{s+d}_{p,q}(\Rn)\to F^{s'}_{p,q}(\Rn)
\quad\text{for arbitrary}\quad s'<s\le \fracnp-n.
\end{equation}
This extends  verbatim to the $B^{s}_{p,q}$-scale.
\end{Rmlist}

Definition~\ref{vfm-defn}
together with the properties \eqref{uni-pro}-\eqref{Fspq'-pro} 
constitute the author's main
contribution to the theory of type $1,1$-operators.

Among the above items, 
\eqref{nwf-pro} and \eqref{dom-pro} amount to sharpenings 
of results in the existing literature. 
The other ten results are rather more substantial, 
as eg both \eqref{uni-pro}--\eqref{max-pro} and the $\cal S'$-continuity in
\eqref{tdc-pro}--\eqref{wtdc-pro} 
have not been treated at all hitherto.

Further comments on \eqref{uni-pro}--\eqref{Fspq'-pro} follow below. 
For convenience the properties \eqref{uni-pro}--\eqref{supp-pro} will be
reviewed in Chapter~\ref{results-sect} in corresponding sections 
\ref{uni-ssect}--\ref{supp-ssect}, whereas the more technical results in 
\eqref{LPa-pro}--\eqref{Fspq'-pro} are described separately in 
Chapter~\ref{resultsII-sect}.

\bigskip 

Behind the type $1,1$-results \eqref{uni-pro}--\eqref{Fspq'-pro} above,
there are at least three new \emph{techniques}:
\begin{rmlist}
    \item { Pointwise estimates of pseudo-differential operators. \/} 
  \item { The spectral support rule of pseudo-differential operators.\/}
  \item{ Stability of extended distributions under regular convergence.\/}
\end{rmlist}
These tools are useful already for classical pseudo-differential
operators, so they are reviewed first, in the next chapter.

\chapter{Techniques for pseudo-differential operators}
\noindent
The results in this chapter are interesting already for a classical symbol, ie
for $a(x,\eta)$ in $S^d_{1,0}(\Rn\times \Rn)$, to which the reader may
specialise if desired. However, it is convenient to state them for symbols in 
$S^d_{1,\delta}$ with $0\le\delta<1$, ie when 
\begin{equation}
  |D^\beta_x D^\alpha_\eta a(x,\eta)|\le C_{\alpha,\beta}
  (1+|\eta|)^{d-|\alpha|+\delta|\beta|}.
\end{equation}
In this way, the extra precaution that would be needed for $\delta=1$ is unnecessary here,
although the results extend directly to type $1,1$-operators, 
unless otherwise is mentioned.

\section{Pointwise estimates of pseudo-differential operators}
\label{pe-ssect}\noindent
It seems to be a new observation, that 
the value of $a(x,D)u(x)$ can be
estimated at each point $x\in \Rn$ thus:
\begin{equation}
  |a(x,D)u(x)|\le  c u^*(x) \quad\text{when}\quad \supp\hat u\Subset \Rn.
  \label{auu*-eq}
\end{equation}
Hereby $u^*$ is the maximal function
of Peetre--Fefferman--Stein type; that is,
\begin{equation}
 u^*(x)=u^*(N,R;x)=\sup_{y\in \Rn}\frac{|u(x-y)|}{(1+R|y|)^{N}}
\end{equation}
with $R>0$ chosen so that $\supp\hat
u$ is contained in the closed ball $\overline{B}(0,R)$. The parameter $N>0$
can eg be larger than the order of $\hat u$, so that $u^*(x)<\infty $ holds
by the Paley--Wiener--Schwartz Theorem.

The above inequality is really a consequence of the following
\emph{factorisation inequality}, shown in \cite[Thm.~4.1]{JJ10pe}.
This involves a
cut-off function $\chi\in C^\infty_0(\Rn)$ that should equal $1$ in a
neighbourhood of $\supp\hat u\Subset \Rn$:
\begin{gather}
  |a(x,D) u(x)|\le F_a(N,R;x)\cdot u^*(N,R;x) 
  \label{Fau*-ineq} \\
  F_a(N,R;x)=\int_{\Rn} (1+R|y|)^{N}|\cal F^{-1}_{\eta\to y}
    (a(x,\eta)\chi(\eta))|\,dy
  \label{Fa-eq}
\end{gather}
This simply means that the action of $a(x,D)$ on $u$ can be decomposed, at the
unimportant price of an estimate, into a product   
where the \emph{entire} dependence on the symbol lies in the ``$a$-factor''
$F_a(N,R;x)$, also called the symbol factor.

The symbol factor $F_a$
only depends vaguely on $u$ through $N$ and $R$.
(Eg $N=[n/2]+1$ works for all  $u\in \bigcup H^s(\Rn)$, so then $N$ plays no
role.)
Formula \eqref{Fa-eq} shows that $F_a$ is a weighted $L_1$-norm of a 
regularisation of the distribution kernel $K$. 
In general $F_a\in C^0\cap L_\infty (\Rn)$, so
together \eqref{Fau*-ineq}--\eqref{Fa-eq} yield \eqref{auu*-eq}.

In the exploitation of \eqref{Fau*-ineq},
it is rather straightforward to control the maximal function $u^*(x)$ with
polynomial bounds. Eg, if $N$ is greater than the order of $\hat u$, the
Paley--Wiener--Schwartz Theorem gives $|u(y)|\le c(1+|y|)^N\le 
(1+|x|)^N(1+R|x-y|)^N$ when $R\ge 1$, so that in this case
\begin{equation}
  u^*(N,R;x)\le c (1+|x|)^N.
  \label{u*xN-eq}
\end{equation}
Moreover, the
maximal operator $u\mapsto u^*$ is bounded with respect to the $L_p$-norm on
$L_p\bigcap \cal F^{-1}\cal E'$,
\begin{equation}
  \int_{\Rn} u^*(N,R;x)^p\,dx \le C_p \int_{\Rn} |u(x)|^p\,dx,
\qquad 0< p\le \infty,\quad N>n/p .
  \label{u*Lp-ineq}
\end{equation}
Consequently the `trilogy' \eqref{Fau*-ineq}, \eqref{Fa-eq}, \eqref{u*Lp-ineq}
leads at once to bounds
of pseudo-differential operators on $L_p\bigcap \cal F^{-1}\cal E'$,
\begin{equation}
  \int|a(x,D)u(x)|^p\,dx\le \nrm{F_a}{\infty }^p 
  \int u^*(x)^p\,dx \le C_p\nrm{F_a}{\infty }^p\int |u(x)|^p\,dx.
  \label{LpFE-eq}
\end{equation}
As $\nrm{F_a}{\infty }=\sup|F_a(N,R;\cdot)|$ depends on $R$, this extends to
all $u\in L_p$ only if $a(x,\eta)$ has further properties. But it is
noteworthy that the above boundedness holds whenever $0<p\le\infty$, so it was
stated as a result in \cite[Cor.~4.4]{JJ10pe}, and in the type
$1,1$-context in \cite[Thm.~6.1]{JJ10pe}; cf Remark~\ref{LpFE-rem} below. 

With a little more effort,
mainly by renouncing on the compact spectrum of $u$, a transparent proof of
the fact that $a(x,D)$ is a map $\cal O_M\to \cal O_M$ was also obtained in
this way; cf \cite[Cor.~4.3]{JJ10pe}. However, for type $1,1$-operators, this
result requires another proof because it is not clear a priori that
$\cal O_M$ is contained in $D(a(x,D))$; cf Section~\ref{C8-ssect} below.

These estimates of $a(x,D)$ are a bit paradoxical because the map 
$u\mapsto u^*$ is non-linear; but this is just a minor drawback as 
\eqref{u*Lp-ineq} was shown by elementary means in \cite{JJ10pe}. 
(The previous proofs of 
\eqref{u*Lp-ineq} in the literature invoke $L_p$-boundedness of 
the Hardy--Littlewood maximal function.)

\begin{rem}   \label{Marschall-rem}
It deserves to be mentioned that somewhat different pointwise estimates were
introduced by J.~Marschall in his thesis \cite{Mar85} and exploited in 
eg \cite{Mar91,Mar95,Mar96}. 
For symbols $b(x,\eta)$ in $L_{1,\loc}(\R^{2n})\cap
\cal S'(\R^{2n})$ with support in $\Rn\times \Bbar(0,2^k)$ and $\supp\cal F
u\subset \Bbar(0,2^k)$, $k\in\N$, Marschall's inequality states that
\begin{equation}
  |b(x,D)v(x)|\le c\Nrm{b(x,2^k\cdot )}{\dot B^{n/t}_{1,t}} M_t u(x),
  \qquad 0<t\le 1.
  \label{Marschall-ineq}
\end{equation}
Here $M_tu(x)=\sup_{r>0} (r^{-n}\int_{B(x,r)}|u(y)|^t\,dy)^{1/t}$ 
is the Hardy--Littlewood maximal function of $u$, when $t=1$, while the norm 
of the homogeneous Besov space $\dot B^{n/t}_{1,t}$ falls on the dilated
symbol $a(x,2^k\cdot)$ parametrised by $x$.
Under the natural condition that the right-hand side is in $L_{1,\loc}(\Rn)$
it was proved in \cite{JJ05DTL}, to which the reader is referred for details;
some shortcomings in Marschall's exposition in eg \cite{Mar96} were pointed out in 
\cite[Rem.~4.2]{JJ05DTL}.
Cf also \cite[Rem.~4.11]{JJ10pe} and \cite[Rem.~7.3]{JJ10tmp}. Marschall's
inequality is mentioned merely for the sake of completeness; it is not
feasible for the general study of type $1,1$-operators.
\end{rem}

In addition to the above observation that the symbol factor $F_a(x)$ is a
bounded continuous function, basic properties of the Fourier transformation
yield the following estimate, that is reminiscent of the
Mihlin--H{\"o}rmander multiplier condition:
\begin{thm}[{\cite[Thm.~4.5]{JJ10pe}}]
  \label{Fa-thm}
Let the symbol factor $F_a(N,R;x)$ be given by   
\eqref{Fa-eq} for parameters $R,N>0$, with the auxiliary function
taken as $\chi=\psi(R^{-1}\cdot)$ for $\psi\in C^\infty_0(\Rn)$ equal to
$1$ in a set with non-empty interior. Then it holds for all $x\in \Rn$ that
\begin{equation}
 0\le F_a(x) \le c_{n,k} \sum_{|\alpha|\le k} 
  (\int_{R\supp \psi} |R^{|\alpha|}D^{\alpha}_\eta a(x,\eta)|^2
    \,\frac{d\eta}{R^n})^{1/2}
  \label{FaMH-eq}
\end{equation}
when $k$ is the least integer satisfying $k>N+n/2$.
\end{thm}

Although the above result has a straightforward proof, it nevertheless 
deserves to be presented as a theorem because it has a very central role.
On the one hand, this will be clear later in the proof of
Theorem~\ref{sigma-thm}, where it allows an exploitation of the profound
condition on the twisted diagonal of L.~H{\"o}rmander, which is
phrased with similar integrals. 

On the other hand, it
is also most convenient for the more standard Littlewood--Paley analysis of
pseudo-differential operators; but in this connection it applies through
its corollaries given below.

First of all, more refined estimates
in terms of symbol seminorms yield 
$\nrm{F_a}{\infty }=\cal O(R^{d'})$ for $d'=\max(d, [N+n/2]+1)$.
However, the exponent can be much improved here in case the auxiliary
function in the symbol factor is supported in a corona:

\begin{cor}[{\cite[Cor.~3.4]{JJ10pe}}]
  \label{Fa-cor} 
Let $a(x,\eta)$ be given in $S^{d}_{1,\delta}(\Rn\times\Rn)$ whilst 
$N$, $R$ and $\psi$ have the
same meaning as in Theorem~\ref{Fa-thm}.
When $R\ge 1$ and $k>N+n/2$, $k\in \N$,
then there is a seminorm $p$ on $S^d_{1,\delta}$ and some 
$c_k>0$ independent of $R$ such that 
\begin{equation}
 0\le F_a(x) \le c_k p(a) R^{\max(d,k)} \quad\text{for all}\quad x\in
\Rn. 
  \label{Rmax-eq}
\end{equation}
Moreover, if $\supp\psi$ is contained in a corona
\begin{equation}
  \{\,\eta\mid \theta_0 \le|\eta|\le \Theta_0 \,\},
  \label{cpsi-eq}
\end{equation}
and 
$\psi(\eta)=1$ holds for $\theta_1\le |\eta|\le \Theta_1$, whereby
$0\ne\theta_0<\theta_1<\Theta_1<\Theta_0$,
then
\begin{equation}
  0\le F_a(x)\le c'_kR^{d}p(a) \quad\text{for all}\quad x\in \Rn,
  \label{Rd-eq}
\end{equation}
with $c'_k=c_k\max(1,\theta_0^{d-k},\theta_0^d)$.
\end{cor}

The above asymptotics for $R\to\infty$ can be further reinforced when
$a(x,\eta)$ has vanishing 
moments with respect to $x$, eg if $\hat a(\cdot ,\eta)$ is zero around
$\xi=0$. A simple result of this type is obtained by subjecting the symbol to
a frequency modulation in its $x$-dependence, using a Fourier multiplier
$\varphi(Q^{-1}D_x)$ that depends on a second spectral quantity $Q$:

\begin{cor}[{\cite[Cor.~4.9]{JJ10pe}}]
  \label{Fa0-cor}
When $a_{Q}(x,\eta)=\varphi(Q^{-1}D_x)a(x,\eta)$ for some $a\in
S^d_{1,\delta}$ and $\varphi\in C^\infty_0(\Rn)$ with $\varphi=0$
in a neighbourhood of $\xi=0$, then there is a seminorm $p$ on $S^d_{1,\delta}$
and constants $c_M$, depending only on $M$, $n$, $N$, $\psi$ and $\varphi$,
such that for $R\ge1$, $M>0$, $Q>0$,
\begin{equation}
  0\le F_{a_{Q}}(N,R;x)\le c_M p(a)Q^{-M} R^{\max(d+\delta M,[N+n/2]+1)}.
\end{equation}
Here ${d+\delta M}$ can replace the maximum 
when the auxiliary function $\psi$ in
$F_{a_{Q}}$ fulfils the corona condition in Corollary~\ref{Fa-cor}.
\end{cor}

Not surprisingly, it is very convenient to have 
an adaptation of \eqref{u*xN-eq} to the frequency modulated symbols
appearing in Definition~\ref{vfm-defn}. One such result is

\begin{prop}[{\cite[Prop.~3.5]{JJ10tmp}}]
  \label{cutoff-prop}
For $a(x,\eta)$ in $S^d_{1,\delta}(\Rn\times \Rn)$ and arbitrary $\Phi$, $\Psi\in
C^\infty_0(\Rn)$, for which $\Psi$ is constant in a neighbourhood of the
origin and is supported by $\Bbar(0,R)$ for $R\ge 1$, 
there is a constant $c>0$ such that for all $k\in\N$, $N\ge \order_{\cal S'} 
(\cal Fv)$, 
\begin{equation}
  \big|\OP\big(\Phi(2^{-k}D_x)a(x,\eta)\Psi(2^{-k}\eta)\big)v(x)\big|
  \le  c 2^{k(N+d)_+}(1+|x|)^N.
\end{equation}
Here the positive part $(N+d)_+=\max(0,N+d)$ is redundant when
$0\notin\supp\Psi$. 
\end{prop}

One of the points here is that the cutoff functions $\Phi$, $\Psi$ can be
rather arbitrary, and that $c$ is independent of $k$. The temperate order
denoted
$\order_{\cal S'}$ in the proposition is for $u\in\cal S'$ introduced as the
smallest integer $N$ such that $u$ fulfils the estimate
\begin{equation} \label{Sorder-eq}
  |\dual{u}{\psi}|\le c\sup\{\,(1+|x|)^N|D^\alpha\psi(x)|\mid x\in\Rn,\
  |\alpha|\le N \,\},\quad\text{for}\quad \psi\in\cal S.
\end{equation}
Clearly one has $\order_{\cal S'}(u)\ge \order(u)$, 
but the notion plays only a minor technical role.

\bigskip

The inequalities \eqref{Fau*-ineq}, \eqref{u*Lp-ineq} are in fact relatively
easy to show, but the passage to estimates in Sobolev spaces $H^s_p$ requires
Littlewood--Paley decompositions (which works well, cf \eqref{LPa-pro}). 
However, when treating
these, the results of the next section are most convenient:

\section{The spectral support rule}
\label{srule-ssect}\noindent
Seen as a temperate distribution, $a(x,D)u$ has a spectrum consisting
of the frequencies belonging to $\supp \cal F(a(x,D)u)$. Concerning this
one has as a new result the \emph{spectral support rule}, which in case 
$\supp\hat u\Subset \Rn$ states that
\begin{equation}
  \supp\cal F a(x,D)u \subset 
   \bigl\{\,\xi+\eta \bigm| (\xi,\eta)\in\supp\hat a(\cdot,\cdot),
  \ \eta\in\supp\hat u \,\bigr\}.
  \label{srule-eq}
\end{equation}
Cf the original statements  in 
\cite{JJ05DTL,JJ08vfm} or \cite[App.~B]{JJ10tmp} for more general versions.

It is instructive to note that \eqref{srule-eq} also can be written as
\begin{equation}
 \supp \cal Fa(x,D)\cal F^{-1}\hat u\subset \supp\cal K\circ \supp\hat u, 
  \label{srule'-eq}
\end{equation}
where $\cal K$ denotes the distribution kernel of $\cal F a(x,D)\cal F^{-1}$, 
ie of the conjugation of $a(x,D)$ by the Fourier transformation, that
also appears on the left-hand side. Clearly this resembles the rule for 
$\supp a(x,D)u$; cf \eqref{sKau-eq}. It is also related to the
well-known formula for symbols $a\in\cal S(\Rn\times\Rn)$,
\begin{equation}
  \cal F a(x,D)u(x)=(2\pi)^{-n}\int \hat a(\xi-\eta,\eta)\hat u(\eta)\,d\eta,
  \quad u\in\cal S.
\end{equation}
Indeed, an inspection shows that
\begin{equation}
  \cal K(\xi,\eta)=(2\pi)^{-n}\hat a(\xi-\eta,\eta)
  =(2\pi)^{-n}\cal F_{(x,y)\to(\xi,\eta)}K(\xi,-\eta).
\end{equation}
Therefore \eqref{srule-eq}--\eqref{srule'-eq} are plausible,
since this shows that $\hat a$ essentially gives the full frequency content 
of the kernel $K$.

The result in \eqref{srule-eq} 
is a novelty already for classical $a(x,\eta)$.
It holds trivially if $a(x,\eta)$ is an elementary symbol, which were
introduced in 1978 by R.~Coifman and Y.~Meyer \cite{CoMe78} 
specifically for the purpose of controlling the spectrum 
$\supp \cal F a(x,D)u$ in Littlewood--Paley analysis of $a(x,D)u$.
Indeed, elementary symbols are by definition given as a series of products
\begin{equation}
  a(x,\eta)=\sum_{j=0}^\infty m_j(x)\Phi_j(\eta)
\end{equation}
whereby $(m_j)$ is a sequence in $L_\infty (\Rn)$ and
$1=\sum_{j=0}^\infty \Phi_j$ is a Littlewood--Paley partition of unity,
that is $\Phi_j$ is in $C^\infty $ with support where
$2^{j-1}\le |\eta|\le 2^{j+1}$ for $j\ge 1$.
For such symbols in $S^d_{1,0}$ every $u\in \cal
F^{-1}\cal E'(\Rn)$ gives a finite sum
\begin{equation}
  a(x,D)u=\sum m_j(x)\Phi_j(D)u,
\end{equation}
for which the support rule for convolutions immediately yields
\begin{equation}
\begin{split}
    \supp \cal F(a(x,D)u)
&= \supp\big((2\pi)^{-n}\sum\hat m_j*(\Phi_j\hat u) \big)
\\
  &\subset  \bigcup \bigl\{\, \xi+\eta \bigm|  \xi\in \supp \hat m_j,\ 
       \eta\in \supp \Phi_j\cap \supp \hat u\,\bigr\}
\\
  &\subset \bigl\{\, \xi+\eta \bigm|  (\xi,\eta)\in \supp \hat a,\ 
       \eta\in\supp \hat u\,\bigr\}.
  \end{split}
\end{equation}
This shows that the spectral support rule holds for elementary symbols.

However, it should be mentioned that there is an equally simple proof
for \emph{arbitrary} symbols $a\in S^d_{1,0}$: 
When $v\in C^\infty_0(\Rn)$ has support disjoint from $\supp\cal K\circ
\supp\hat u$ and $\supp\hat u$ is compact,
then it is clear that $\op{dist}(\supp \cal K, \supp(v\otimes \hat u))>0$. 
So by mollification, say $\hat u_\varepsilon=\varphi_\varepsilon*\hat u$ 
for some $\varphi\in C^\infty_0(\Rn)$ with $\hat \varphi(0)=1$,
$\varphi_\varepsilon=\varepsilon^{-n}\varphi(\cdot /\varepsilon)$,
all sufficiently small $\varepsilon>0$ give
\begin{equation}
  \supp\cal K\bigcap \supp v\otimes \hat u_\varepsilon=\emptyset.
\end{equation} 
Therefore \eqref{srule'-eq} follows at once, since
\begin{equation}
  \dual{\cal Fa(x,D)\cal F^{-1}\hat u}{v}
 = \lim_{\varepsilon\to0}   \dual{\cal Fa(x,D)\cal F^{-1}\hat u_\varepsilon}{v}
 = \lim_{\varepsilon\to0}\dual{\cal K}{v\otimes \hat u_\varepsilon}=0
\end{equation}
is obtained simply by using that $\cal Fa(x,D)\cal F^{-1}$ is continuous in
$\cal S'$ and that $\hat u_\varepsilon\in C^\infty_0(\Rn)$.
(This argument is taken from \cite[App.~B]{JJ10tmp}.)

The spectral support rule \eqref{srule-eq} was probably anticipated by some,
but seemingly neither formulated nor proved. Indeed, in their works on
$L_p$-estimates, J.~Marschall \cite{Mar91,Mar96} and
T.~Runst~\cite{Run85ex}  both tacitly avoided elementary symbols and 
as needed stated consequences of \eqref{srule-eq}, albeit without adequate
arguments; cf the remarks in \cite{JJ05DTL}. 
Anyhow, due to \eqref{srule-eq}, the cumbersome reduction to elementary
symbols is usually unnecessary.

Generalisations to the case in which $\supp\hat u$ need not be compact (in
which case one should take the closure of the right-hand sides of
\eqref{srule-eq}--\eqref{srule'-eq}) and to the case of type $1,1$-operators
also exist, cf Section~\ref{supp-ssect} below. 
However, the proofs for these cases were based on some subtle parts of
distribution theory: 

\section{Stability of extended distributions under regular convergence}
\label{regular-ssect}\noindent
If $u,f\in \cal D'(\Rn)$ only ``overlap'' in a mild way, more precisely,
\begin{align}
  \supp u \cap \supp f &\Subset \Rn
  \label{ufs-eq}
  \\
  \singsupp u\cap \singsupp f&=\emptyset,
  \label{ufss-eq}
\end{align}
and $\supp u$ is compact, it is natural and classical
(cf \cite[Sect.~3.1]{H}) that $fu$ is
defined in $\cal D'(\Rn)$, whence $\dual{u}{f}$ can be defined using $\supp
u\Subset \Rn$ as
\begin{equation}
  \dual{u}{f}=\dual{fu}{1}.
  \label{fu1-id}
\end{equation}
It is easy to see that this well-known extension of the distribution $u$, or
rather 
of the scalar product $\dual{\cdot }{\cdot }$ 
is discontinuous in general. Eg $f=0$ can be
approached by $f_\nu=\exp(-\nu x^2)$ in $\cal D'(\R)$, that for $u=\delta_0$
gives $fu=0\ne \delta_0=\lim f_\nu u$, hence for the scalar product yields
$\dual{u}{f}=0\ne1=\lim \dual{u}{f_\nu}$. 

However, the extension does have an important property of stability:

\begin{thm}  \label{ext-thm}
For the above extension it holds that
\begin{equation}
  \dual{u}{f_\nu}\to \dual{u}{f} \quad\text{for}\quad \nu\to\infty ,
  \label{ufnu-eq}
\end{equation}
provided $f_\nu\in  C^\infty(\Rn)$ and 
$ f_{\nu}\xrightarrow[\nu\to\infty ]{~} f $
both in $\cal D'(\Rn)$ and in $C^\infty (\Rn\setminus \singsupp f)$.
\end{thm}

The full set of results is collected in \cite[Thm.~7.2]{JJ08vfm}. Eg it is
possible to have convergence of $(f_\nu)$ in the topology of $C^\infty $ over a smaller
open set if only this contains the singular support of $u$ (which is
unfulfilled for $(f_\nu)$ in the above example).

Sequences as in Theorem~\ref{ext-thm} have been used repeatedly for type
$1,1$-operators,
so the following notion is introduced, inspired by a reference to
$\Rn\setminus \singsupp f$ as the regular set of $f$:

\begin{defn}
  \label{regconv-defn}
A sequence $f_\nu\in C^\infty (\Rn)$ is said to converge \emph{regularly} to
the distribution $f\in \cal D'(\Rn)$ whenever $f=\lim_\nu f_\nu$ holds 
in $\cal D'(\Rn)$ as well as in $C^\infty (\Rn\setminus \singsupp f)$, that
is, if for $\nu\to\infty $,
\begin{gather}
  \dual{f_\nu-f}{\varphi} \to0 \quad\text{for all}\quad \varphi\in
C^\infty_0(\Rn) 
\\
  \sup_{x\in K}|D^\alpha f_\nu(x)-D^\alpha f(x)|\to 0
\quad\text{for all $\alpha\in \N_0^n$,}\quad 
  K\Subset \Rn\setminus \singsupp f.
\end{gather}
\end{defn}
This definition was made (implicitly) in connection with
\cite[Thm.~7.2]{JJ08vfm}. 
The result below shows that mollification automatically yields 
regular convergence, for which reason it was termed the 
Regular Convergence Lemma in \cite[Lem.~6.1]{JJ08vfm}:

\begin{lem}
  \label{regconv-lem}
Let $u\in \cal S'(\Rn)$ be given and take a sequence $\varepsilon_k\to 0^+$ 
and $\psi\in \cal S(\Rn)$. Then 
\begin{equation}
 \psi(\varepsilon_k D) u\to \psi(0)\cdot u
  \quad\text{for}\quad k\to\infty
  \label{psiku-eq}
\end{equation}
in the Fr\'echet space $C^\infty(\Rn\setminus\singsupp u)$. 
If $\cal F^{-1}\psi\in C^\infty_0$ this extends to all $u\in \cal D'$, 
provided $\psi(\varepsilon_k D)u$ is replaced by 
$\cal F^{-1}(\psi(\varepsilon_k\cdot ))*u$. 
\end{lem}

The last part of this lemma is easy to deduce, using a cutoff
function equal to $1$ on a neighbourhood of the given compact set, where the
derivatives should converge uniformly. Only the $\cal S'$-part requires a
more explicit proof.

However, despite the Regular Convergence Lemma's content, the broader notion
of regular convergence is convenient 
because such sequences are
invariant under eg linear coordinate changes, 
multiplication by cutoff functions and tensor products $f\mapsto f\otimes g$
when $g\in C^\infty $.

These remarks are useful in connection with Schwartz' kernel formula. Recall
that for a continuous operator  
$A\colon \cal S'(\Rn)\to\cal S'(\Rn)$, its distribution
kernel $K\in \cal S'(\Rn\times \Rn)$ satisfies, 
for all $u$, $v\in \cal S(\Rn)$,
\begin{equation}
  \dual{Au}{v}=\dual{K}{v\otimes u}.
  \label{AK-eq}
\end{equation}
First of all, this can be related to the vanishing
frequency modulation adopted for type $1,1$-operators.
Indeed, when 
$a(x,D)u=\lim_{m\to\infty }\OP(\psi(2^{-m}D_x)a(x,\eta)\psi(2^{-m}\eta))u$
and the $m^{\op{th}}$ term is written $A_mu$, 
then $A_m$ has distribution kernel $K_m(x,y)$ given by a convolution
conjugated by a change of coordinates (cf \cite[Prop.~5.11]{JJ08vfm}), namely
\begin{equation}
  K_m(x,y)=4^{mn}\cal F^{-1}(\psi\otimes\psi)(2^{m}\cdot )*
  (K\circ \left(\begin{smallmatrix}I&0\\I&-I\end{smallmatrix}\right))(x,x-y).
\end{equation}
Because of the Regular Convergence Lemma, this $C^\infty $-function
converges \emph{regularly} to $K$ for $m\to\infty $.
Therefore $K_m\to K$ in $\cal S'(\Rn)$ as well as in 
$C^\infty (\Rn\setminus \{\,x=y\,\})$.

However, with a suitable cutoff function this gives convergence in the
Schwartz space:

\begin{prop}[{\cite[Prop.~6.3]{JJ08vfm}}]
  \label{KmKS-prop}
If $f\in C^\infty (\Rn\times \Rn)$ has bounded derivatives of any order
with $\supp f$ bounded with respect to $x$ and disjoint from the diagonal, then
\begin{equation}
  f(x,y)K_m(x,y)\xrightarrow[m\to\infty ]{~}
  f(x,y)K(x,y)
  \quad\text{in}\quad \cal S(\Rn\times \Rn).
\end{equation}
\end{prop}

It is noteworthy that the proof of this plausible proposition relies on the
mentioned 
less trivial part of the Regular Convergence Lemma, in which the function
$\cal F^{-1}\psi$ there is in $ \cal S\setminus C^\infty_0$.
Certainly Proposition~\ref{KmKS-prop} sheds light on the limit in
Definition~\ref{vfm-defn}, but it is also a convenient proof ingredient later.

Secondly, \eqref{AK-eq} is by \eqref{ufnu-eq} easily extended  
to the pairs $(u,v)\in \cal S'(\Rn)\times C^\infty_0(\Rn)$ fulfilling
\begin{align}
  \supp K\bigcap \supp(v\otimes u)&\Subset \Rn\times \Rn,
  \label{sKvu-eq}
\\
  \singsupp K\bigcap \singsupp(v\otimes u)&=\emptyset.
  \label{ssKvu-eq}
\end{align}
\begin{thm}[{\cite[Thm.~7.4]{JJ08vfm}}] 
  \label{AK-thm}
If $A\colon \cal S'(\Rn)\to\cal S'(\Rn)$ is continuous and \eqref{sKvu-eq},
\eqref{ssKvu-eq} are fulfilled,  then $\dual{Au}{v}=\dual{K}{v\otimes u}$
holds with extended action of the scalar product. This extends to the $\cal
D'$-case. 
\end{thm}

It is illuminating to give the short argument:
the right-hand side of \eqref{AK-eq} is defined according to
\eqref{sKvu-eq}--\eqref{ssKvu-eq} and the extension of
$\dual{\cdot }{\cdot }$ in \eqref{fu1-id}, 
so it only remains to verify the equality in \eqref{AK-eq} 
under the assumptions \eqref{sKvu-eq}--\eqref{ssKvu-eq}. 

For this one can clearly take 
$\kappa$, $\chi\in C^\infty_0(\Rn)$ such that $\kappa=1$ on
$\supp v$ and $\kappa(x)\chi(y)=1$ on the compact set in \eqref{sKvu-eq}.
By letting $u_\nu\in C^\infty (\Rn)$ tend regularly to $u$, cf
Lemma~\ref{regconv-lem},  
the convergence $v\otimes u_\nu\to v\otimes u$ is also regular, so one finds
from Theorem~\ref{ext-thm},
\begin{equation}
  \begin{split}
  \dual{K}{v\otimes u}&= \dual{(\kappa\otimes \chi)K}{v\otimes u}=
   \lim_{\nu\to\infty }\dual{(\kappa\otimes \chi)K}{v\otimes u_\nu}
\\
  &=\lim_{\nu\to\infty }\dual{K}{(\kappa v)\otimes (\chi u_\nu)}_{\cal
S'\times \cal S}
  =\lim_{\nu\to\infty }\dual{A(\chi u_\nu)}{\kappa v}_{\cal S'\times \cal S}
  =\dual{Au}{v}.
  \end{split}
\end{equation}
This proves the theorem.

As consequences it should be pointed out
that the support rule \eqref{sKau-eq} follows at once from \eqref{AK-eq}
for $A=a(x,D)\in S^d_{1,0}(\Rn\times \Rn)$ by taking 
$v\in C^\infty_0(\Rn)$ with
support disjoint from that of $\supp K\circ \supp u$. 
It is noteworthy that also the \emph{spectral} support rule 
\eqref{ssKau-eq},\eqref{srule'-eq}
follows in this way for $A=\cal F a(x,D)\cal F^{-1}$,
for this is also continuous on $\cal S'$ for such $a(x,D)$.
 
Type $1,1$-operators requires some additional efforts due to the limit
$m\to\infty $ in \eqref{axDu-id}. The main line is the same as the above,
that roughly speaking applies for each $m$; for \eqref{sKau-eq}
the convergence in Proposition~\ref{KmKS-prop} was sufficient,
cf \cite[Sect.~7--8]{JJ08vfm}.
For the spectral support rule \eqref{srule-eq} the
passage to the limit $m\to\infty $ required an extra assumption ($\cal
S'$-convergence in \eqref{axDu-id}), but still the main ingredient
was stability under regular convergence in the kernel formula.

\subsection{Other extensions}
Among the many possible extensions of $\dual{\cdot }{\cdot }$, it is
particularly relevant to recall the one related to the space $\cal
D'_\Gamma$ consisting of the $u\in \cal D'$ with $\op{WF}(u)\subset \Gamma$,
whereby $\Gamma\subset \Rn\times (\Rn\setminus \{0\})$ 
is a fixed closed, conical set (ie $\Gamma$ is
invariant under scaling by positive reals in the second entry). 
$\cal D'_\Gamma$ is given a stronger topology than the relative by adding the
seminorms 
\begin{equation}
  p_{\varphi,N,V}(u)=\sup_{\eta\in V}(1+|\eta|)^{N}
  |\widehat{\varphi u}(\eta)|, \qquad N=1,2,\dots ,
\end{equation}
where $\varphi\in C^\infty_0$ and the closed cone $V\subset \Rn$ run through
those with $\Gamma\bigcap(\supp \varphi\times V)=\emptyset$.

For cones $\Gamma_1$, $\Gamma_2$ such that $(x,-\eta)\notin \Gamma_2$
whenever $(x,\eta)\in \Gamma_1$, there is an extension of 
$\dual{\cdot }{\cdot }$ to a bilinear map $\cal D'_{\Gamma_1}\times \cal
E'_{\Gamma_2}\to\C$, which is sequentially continuous in each variable;
this is eg explained in the notes of A.~Grigis and J.~Sj{\"o}strand
\cite[Prop.~7.6]{GrSj94}.
Obviously the wavefront condition expressed via $\Gamma_1$, $\Gamma_2$
is weaker than disjointness of the singular supports.

On the other hand, any sequence 
$f_\nu\in C^\infty (\Rn)$ such that
$f_\nu\to f$ in $\cal D'_\Gamma$ automatically tends regularly to $f$ in the
sense of  
Definition~\ref{regconv-defn}, for the supremum over $K$ there goes to $0$
because it can be estimated by $p_{\varphi,|\alpha|+n+1,\Rn}(f_\nu-f)$ when
$\varphi=1$ on $K$ and $\supp\varphi\cap \singsupp f=\emptyset$ (allowing
$V=\Rn$), using that $\cal F$ is bounded from $L_1$ to $L_\infty $. 

The incompatibility of the two extensions becomes clearer by noting that
$\dual{u}{f}$ is defined whenever the product $fu$ makes
sense in $\cal E'$; cf \eqref{fu1-id}. Eg one may use the product
$\pi(f,u)$ defined formally by regarding
$f$ as a (non-smooth) symbol independent of $\eta$ 
(cf Remark~1.1 in \cite{JJ08vfm}, or the author's paper \cite{JJ94mlt} devoted
to $\pi(f,u)$).  
It is well known that $\pi(\delta_0,H)=\delta_0/2$, when $H=1_{\R_+}$ is the
Heaviside function, so from this one finds
$\dual{\delta_0}{H}=\dual{\frac{1}{2}\delta_0}{1}=\frac{1}{2}$. 
(As $\op{WF}(\delta_0)=\{0\}\times \R$, wavefront sets are not useful
here.) 

However, as the point of the regular convergence is to simplify (and
to emphasise the essential), this notion should be well motivated.

\begin{rem}
Both parts of the Regular Convergence Lemma 
could have been known since the 1950's in view of its content, of course.
The same could be said about the stability in Theorem~\ref{ext-thm} and the
resulting kernel convergence in Proposition~\ref{KmKS-prop} as well as the
extended kernel formula in Theorem~\ref{AK-thm}.
But it has not been possible to track any evidence of this, neither written
nor as folklore.
\end{rem}

\chapter{Review of qualitative results}
\label{results-sect}\noindent
This chapter gives a detailed account of the results summarised in items
\eqref{uni-pro}--\eqref{supp-pro} in Section~\ref{dozen-ssect}. 
The review follows the order there.

For convenience $a(x,\eta)$ denotes an arbitrary
symbol in $S^d _{1,1}(\Rn\times \Rn)$.

\section{Consistency among extensions}
\label{uni-ssect}\noindent
The definition by vanishing frequency modulation has the merit of giving
back most, if not all, of the previous extensions of type $1,1$-operators. 
This is reviewed in the subsections below.

\subsection{Extension to functions with compact spectrum}
First of all there was in \cite{JJ04DCR,JJ05DTL}
a mild extension of $a(x,D)$ to $\cal F^{-1}\cal E'(\Rn)$.
The extension is rather elementary, but is easy to explain with a point of
view from \cite{JJ08vfm}: the defining integral may be seen 
as a scalar product for $u\in \cal F^{-1}C^\infty_0(\Rn)$
\begin{equation}
  a(x,D)(x)u=\Dual{\hat u}{a(x,\cdot )
             \frac{e^{\im \dual{x}{\cdot }}}{(2\pi)^n}}_{\cal
E'\times C^\infty}.
  \label{aFC-eq}
\end{equation}
On the right-hand side one is free to insert any $\hat u\in \cal E'$, which
is consistent with \eqref{axDu-eq} because $\cal S\cap \cal F^{-1}\cal E'=
\cal F^{-1}C^\infty_0$. 

More precisely this gives an extension to a map
$\tilde a(x,D)\colon \cal S+\cal F^{-1}\cal E'\to C^\infty $ given by
\begin{equation}
  \tilde a(x,D)u=a(x,D)v+\OP(a(x,\eta)\chi(\eta))v'
  \label{aFE-eq}
\end{equation}
when $u=v+v'$ for some $v\in \cal S$ and $v'\in \cal F^{-1}\cal E'$ whilst
$\chi\in C^\infty $ is an arbitrary function equalling $1$ on neighbourhood
of $\supp\hat u$. Indeed, $\chi$ can be inserted already in \eqref{aFC-eq},
and since the resulting symbol $a(x,\eta)\chi(\eta)$ is in $S^{-\infty }$
the formula for $\tilde a(x,D)$ makes sense. The value of $\tilde a(x,D)u$
is also independent of how $v$, $v'$ are chosen, as can be seen using
linearity and \eqref{aFC-eq}.

As examples of the above extension, type $1,1$-operators are always defined on
polynomials $\sum_{|\alpha|\le m}a_\alpha x^\alpha$, 
plane waves $e^{\im x\cdot z}$ and also on the less trivial function
$\tfrac{\sin x_1}{x_1}\dots \tfrac{\sin x_n}{x_n}$, since this is equal to
$\pi^n\cal F^{-1}1_{[-1,1]^n}$.

It was verified in \cite[Cor.~4.7]{JJ08vfm} that this extension is contained
in the operator defined by vanishing frequency modulation.
However, this also results from the next section.

\subsection{Extension to slowly growing functions}   \label{GBext-ssect}
Following an early remark by G.~Bourdaud \cite{Bou87} (who treated
singular integral operators) one can obtain that every type
$1,1$ symbol $a(x,\eta)$ gives rise to a map
\begin{equation}
  \tilde A\colon \cal O_M(\Rn)\to \cal D'(\Rn).
  \label{AOD-eq}
\end{equation}
Hereby $\cal O_M$ stands for the space of $C^\infty $-functions that
together with all their derivatives have polynomial growth at infinity.

Indeed, for each $f\in \cal O_M$ one may  take
$\tilde Af$ as the distribution that for each $\varphi\in C^\infty_0(\Rn)$,
and $\chi\in C^\infty_0(\Rn)$
equal to $1$ on a neighbourhood of $\varphi$, is given by
\begin{equation}
  \dual{\tilde Af}{\varphi}= \dual{a(x,D)(\chi f)}{\varphi}
   +\iint K(x,y)(1-\chi(y))f(y)\varphi(x)\,dy\,dx.
  \label{aOC-eq}
\end{equation}
Here the distribution kernel $K(x,y)$ decays rapidly for fixed $x$ and
$|y|\to\infty $, so that the integral makes sense.
The right-hand side
gives the same value for any other such cutoff function $\tilde \chi$, for
an analogous integral defined from $\chi-\tilde \chi$
has the opposite sign of $\dual{a(x,D)((\tilde \chi-\chi)f)}{\varphi}$.
In view of this independence, and since the absolute value is less than a constant times $\sup|\varphi|$,
$\tilde Af$ defines a distribution in $\cal D'(\Rn)$. 

It can also be seen that $\tilde Af$ is smooth and
slowly increasing, and with some effort that $\tilde A$ is in fact a
restriction of $a(x,D)$ defined by vanishing frequency modulation:

\begin{prop}   \label{aOO-prop}
Each $a(x,D)$ in $\OP(S^d _{1,1})$ restricts to a map 
$\cal O_M(\Rn)\to \cal O_M(\Rn)$.
\end{prop}

This result contains the previous extension to $\cal
S+\cal F^{-1}\cal E'$ in \eqref{aFE-eq}, and it is 
rather more precise. Of course it looks like being a completion, but this is
not obvious as neither the topology on  
$\cal O_M(\Rn)$ nor continuity is involved in the statement.

The proposition is given without details here, as it is
superseeded by an extension to $C^\infty \bigcap\cal S'$, which is
derived from a closer inspection of $\tilde A$. 
However, this result follows in Theorem~\ref{aOO-thm} below, because it is 
rather more important in itself.

\begin{rem}
In a remark preceding the proof of the $T1$-theorem of G.~David and
J.-L.~Journ\'e \cite{DaJo84},   
it was explained  that
just a few properties of the distribution kernel of a continuous map $T\colon
C^\infty_0(\Rn)\to\cal D'(\Rn)$ implies that $T(1)$ is well defined modulo
constants. In particular this was applied to $T\in \OP(S^0_{1,1})$, but in
that case their extension equals the above, hence by
Proposition~\ref{aOO-prop}
also gives the same result as Definition~\ref{vfm-defn}.
\end{rem}

\subsection{Extension by continuity}
In \cite{H88,H89}, L.~H{\"o}rmander characterised the $s\in \R$
for which a given $a(x,D)\in \OP(S^d_{1,1})$ 
extends by continuity to a bounded operator $H^{s+d}\to H^s$; the only
possible exception was a certain limit point $s_0$ that was not treated, cf
\cite{H97}. 
More precisely (paraphrasing his results) 
he obtained a largest interval $\,]s_0,\infty [\,\ni s$
together with constants $C_s$ such that
\begin{equation}
  \nrm{a(x,D)u}{H^s}\le C_s\nrm{u}{H^{s+d}}\quad\text{for all}\quad
  u\in \cal S(\Rn).
\end{equation}
Here $s_0\le 0$ always holds.
Conversely existence of such a $C_s$ was shown to imply $s\ge s_0$.
More precisely, $s_0=-\sup \sigma$ when $\sigma$ runs through the
values for which the symbol fulfils \eqref{wtdc-cnd}.

R.~Torres \cite{Tor90} worked with the full scale of Lizorkin--Triebel
spaces $F^{s}_{p,q}$.
His methods relied on the framework of atoms and molecules of 
M.~Frazier and B.~Jawerth \cite{FJ1,FJ2}, but 
he also estimated $a(x,D)u$ for $u\in \cal S(\Rn)$.
This gave extensions by continuity to maps 
\begin{equation}
  A\colon F^{s+d}_{p,q}(\Rn)\to F^{s}_{p,q}(\Rn)  
\end{equation}
for $s>\max(0,\fracc np-n,\fracc nq-n)$ and $0<p<\infty $, $0<q\le \infty $,
and more generally 
for all $s$ so large that for all multiindices $\gamma$ it holds true that
\begin{equation}
  0\le|\gamma|<\max(0,\frac np-n,\frac nq-n)-s
  \implies \cal F(a(x,D)^*x^\gamma) \in \cal E'(\Rn).
\end{equation}

Of course the uniqueness of these extensions yield that they coincide with
Definition~\ref{vfm-defn} whenever the same continuity properties of $a(x,D)$
can be proved by other means. This has to a large extent been done with
paradifferential decompositions, as reviewed in the section below.

Indeed, in the $H^s_p$ context with $s>0$ and $1<p<\infty $ this was
done in \cite[Thm.~9.2]{JJ08vfm}, which covers the above result of
L.~H{\"o}rmander for $s>0$, and also for $s\in \R $ in case $a(x,D)$
fulfils the twisted diagonal condition \eqref{tdc-cnd}; and done in 
\cite[Cor.~7.6]{JJ10tmp} for $s\in\R$ when \eqref{wtdc-cnd} holds
for every $\sigma\in \R$ .
These results were in fact just special cases of the $F^{s}_{p,q}$ results in
\cite{JJ05DTL, JJ10tmp}, so for the same $s$
also the extensions of R.~Torres are restrictions of the operator defined
by vanishing frequency modulation; cf Theorems~\ref{Hsp-thm} and
\ref{FBspq-thm} below. 

However, it should be emphasized that full coherence has not yet been
obtained. For $s\le 0$ the above extensions by continuity still require
treatment when $a(x,\eta)$ only satisfies \eqref{wtdc-cnd} 
for a specific $\sigma$, for continuity has then been shown with the present
methods for $s>-\sigma+[n/2]+2$; cf Remark~7.11 in \cite{JJ10tmp}.

\subsection{Extensions through paradifferential decompositions}
  \label{para-sssect}
As is well known,
paradifferential decomposition of $a(x,D)$ yields three contributions to the
limit \eqref{apsi-eq},
\begin{equation}
  a_{\psi}(x,D)=a^{(1)}_{\psi}(x,D)+a^{(2)}_{\psi}(x,D)+a^{(3)}_{\psi}(x,D).
  \label{apsi123-id}
\end{equation}
The details of this decomposition will be given later in
Section~\ref{LPa-ssect}, for here 
it suffices to explain the philosophy behind it:
\begin{itemize}
  \item $a^{(1)}_{\psi}(x,D)u$ has a regularity that depends on $u$ alone
(usually);
  \item the last term $a^{(3)}_{\psi}(x,D)u$ has a regularity determined by
that of the symbol (usually);
  \item in between there is $a^{(2)}_{\psi}(x,D)u$ that may or may not be
defined\,---\,depending on the fine interplay of $u$ and $a(x,\eta)$.
This term is the most regular of the three (usually).
\end{itemize}
In addition to this compelling description, it should be noted that the
usefulness lies in the particular form the terms have (cf
Section~\ref{LPa-ssect}); they consist of three infinite series.
More precisely, each of these can in its turn be treated 
by methods of harmonic analysis, which are quite simple 
owing to the above decomposition of the
singularities in $a_{\psi}(x,D)u$.

In the type $1,1$ context, this decomposition goes back to Y.~Meyer
\cite{Mey80,Mey81} and G.~Bourdaud \cite{Bou83,Bou88}, but has also been
used by numerous authors in several fields.

Quite simply, \eqref{apsi123-id} induces an extension of $a(x,D)$ to those
$u\in \cal S'(\Rn)$ for which all the three mentioned series converge in $\cal
D'(\Rn)$. (This was taken as the definition of type $1,1$-operators
in \cite{JJ04DCR,JJ05DTL}, but was superseeded by Definition~\ref{vfm-defn} in 
\cite{JJ08vfm}.) 
When combining the results of this analysis, one finds eg estimates of the
form 
\begin{equation}
  \nrm{a_{\psi}(x,D)u}{H^{s}_{p}}\le 
  \sum_{j=1,2,3}\nrm{a^{(j)}_{\psi}(x,D)u}{H^{s}_{p}}
  \le C \nrm{u}{H^{s+d}_{p}}.
\end{equation}
Here it is an important point that the inequality will be shown directly for
all  
$u\in H^{s}_{p}$ (without extension by continuity). For this purpose the
pointwise estimates and the spectral support rule reviewed in 
Sections~\ref{pe-ssect}--\ref{srule-ssect} are particularly useful.

Moreover, there is here no dependence on the modulation function $\psi$, for
$\cal S$ is a dense subset on which $a_\psi(x,D)=a(x,D)$.
This way the next result was obtained in \cite{JJ05DTL}:

\begin{thm}
  \label{Hsp-thm}
Let $a(x,\eta)$ be a symbol in $S^d_{1,1}(\Rn\times \Rn)$.
Then for every $s>0$, $1<p<\infty $ the type $1,1$-operator $a(x,D)$ has
$H^{s+d}_{p}(\Rn)$ in its domain and it is
a continuous linear map
\begin{equation}
  a(x,D)\colon H^{s+d}_p(\Rn)\to H^s_p(\Rn).
\end{equation}
This property extends to all $s\in \R$ when $a(x,\eta)$ fulfils the twisted diagonal
condition \eqref{tdc-cnd}.
\end{thm}

There is also a version for the general $F^{s}_{p,q}$ spaces, as
reviewed in Theorem~\ref{FBspq-thm} below.

It is most noteworthy, though, that this result deals directly with the
operator $a(x,D)$ defined by vanishing frequency modulation. 
Phrased with a few words this is because the modulation function,
by dilation, 
gives rise to a Littlewood--Paley partition of unity that can be
inserted twice in $a_{\psi}(x,D)u$\,---\,whereafter bilinearity leads
directly to the decomposition \eqref{apsi123-id}. 
Section~\ref{LPa-ssect} below gives the details of this.

\section{Maximality of the definition by 
vanishing frequency modulation} 
\label{max-ssect}\noindent
It follows from standard results that $a(x,D)$ gives back the usual
operator, written $\OP(a)u$, if the definition is applied to some
$a(x,\eta)\in S^\infty_{1,0}$; cf \cite[Prop.~5.4]{JJ08vfm}.
In addition it is also consistent with the previous extensions, as
elucidated in Section~\ref{uni-ssect}.

But even so the approximants in \eqref{am-id} might seem rather arbitrary. 
That this is not the case is evident from the
following 
characterisation, which justifies the title of this section:

\begin{thm}[{\cite[Thm.~5.9]{JJ08vfm}}]   \label{max-thm}
The map $a\mapsto a(x,D)$ given by Definition~\ref{vfm-defn} is one among the
operator  assignments $a\mapsto\OPT(a)$ mapping each
$a\in S^\infty_{1,1}(\Rn\times\Rn)$ into a linear map from $\cal S'(\Rn)$ to
$\cal D'(\Rn)$ such that:
\begin{rmlist}
  \item   \label{scmtb-cnd}
$\OPT(\cdot)$ is \emph{compatible} with $\OP$ on $S^{-\infty}$, that is,
$\OPT(b)$ is defined for all $u\in \cal S'(\Rn)$ whenever $b(x,\eta)\in
S^{-\infty }(\Rn\times \Rn)$ and $\OPT(b)u=\OP(b)u$;
  \item  \label{stabl-cnd} 
each $\widetilde{\OP}(b)$ 
is \emph{stable} under vanishing frequency modulation, or explicitly
\begin{equation}
  \OPT(b)u=\lim_{m\to\infty}\OPT(\psi(2^{-m}D_x)b(x,\eta)\psi(2^{-m}\eta))u  
\end{equation}
for every modulation function $\psi$, $u\in D(\widetilde{\OP}(b))$
for a fixed $b\in S^\infty_{1,1}$. 
\end{rmlist}
Whenever $\OPT$ is such a map, 
then $\OPT(a)\subset a(x,D)$ holds in the sense of operator theory 
for every $a\in S^\infty_{1,1}$.
\end{thm}

It should be noted that ``from\dots to'' indicates that the operator
$\OPT(a)$ is defined just on a subspace of $\cal S'(\Rn)$, in analogy with
the corresponding formulation in the theory of unbounded operators in
Hilbert space. 
In addition the domain of $\OPT(a)$ is necessarily dense, for as a
consequence of \eqref{scmtb-cnd}--\eqref{stabl-cnd} in the theorem, 
it follows from \eqref{SdS-eq} that
$\OPT(a)$ \emph{extends} $a(x,D)$ in \eqref{axDu-eq}, that is,
$\OPT(a)u=a(x,D)u$ for all $u\in \cal S(\Rn)$.

In \cite[Sect.~4--5]{JJ08vfm} it is also confirmed
that Definition~\ref{vfm-defn} is \emph{strongly compatible} with 
eg $S^\infty_{1,0}$; ie, despite the limit procedure
it gives the result 
\begin{equation}
  a(x,D)u=\OP(a(x,\eta)\chi(\eta))u 
  \label{scomp-eq}
\end{equation}
whenever $u\in \cal S'(\Rn)$ and
$a(x,\eta)\chi(\eta)$ is in $S^\infty _{1,0}$ for some $C^\infty $-function
$\chi$ equalling $1$ on a neighbourhood of $\supp\hat u$.
In particular \eqref{scomp-eq} holds whenever $u\in \cal F^{-1}\cal
E'(\Rn)$, as was already seen in Section~\ref{uni-ssect},
for $a(x,\eta)\chi(\eta)$ belongs to $S^{-\infty }$ when $\chi\in
C^\infty_0(\Rn)$ is as above.

These and other questions of extension and compatibility are dealt with at 
length in \cite[Sect.~4--5]{JJ08vfm}. 

\begin{rem}   \label{Stein-rem}
It is of historic interest to relate Definition~\ref{vfm-defn} to that 
used by E.~M.~Stein around 1972 in the extension to H{\"o}lder--Zygmund spaces 
$C_*^s(\Rn)$ with $s>0$. For convenience the basis for this will be  
the exposition in Proposition~3 of \cite[VII,\,\S1.3]{Ste93}.

The indications will be brief, however, borrowing the Littlewood--Paley
decomposition $1=\sum_{j=0}^\infty \varphi(2^{-j}\eta)$ 
constructed from an arbitrary modulation function $\psi$
in Section~\ref{dyadic-sssect} below. 
The integer $h$ there is such that 
$\tilde \varphi_j:=\varphi(2^{-j+h+1}\cdot )+\dots +\varphi(2^{-j-h-1}\cdot )$
is identical equal to $1$ on $\supp\varphi(2^{-j}\cdot )$. 
Using this, E.~M.~Stein introduced for $a\in \OP(S^0_{1,1})$,
\begin{equation}
  \OPT(a)u = \sum_{j=0}^\infty \OP(a(x,\eta)\varphi(2^{-j}\eta))\tilde
\varphi_j(D)u 
\quad\text{for}\quad u\in C^s(\Rn).
  \label{Stein-id}
\end{equation}
With arguments from harmonic analysis it was shown in \cite{Ste93} that the
series converges in $L_\infty$ and that the induced map is bounded on
$C_*^s(\Rn)$ for $s>0$. 
(Independence of $\psi$ was tacitly by-passed in \cite{Ste93}.)
 
With hindsight \eqref{Stein-id} applies as a definition of $\OPT(a)u$ for
the $u\in \cal S'$ for which the right-hand side converges in $\cal D'$.  
Of course the $\tilde \varphi_j(D)$ are redundant, so the above amounts to
\begin{equation}
  \OPT(a)u=\lim_{k\to\infty } \OP(a(x,\eta)\psi(2^{-k}\eta))u.
  \label{Stein'-id}
\end{equation}
First of all this shows that Stein's extension implicitly relied on
vanishing \emph{partial} frequency modulation, since the symbol is not
modified by $\psi(2^{-m}D_x)$ in \eqref{Stein'-id}. 

Secondly, whether $\OPT(a)$ is a restriction of $a(x,D)$ is by
means of Theorem~\ref{max-thm} easily reduced to whether
$\OPT(a)$ is stable under vanishing frequency modulation, ie to show that
\begin{equation}
  \OPT(a)u=  \lim_{m\to\infty }\OPT(\Psi(2^{-m}D_x)a(x,\eta)\Psi(2^{-m}\eta))u
  \label{Stein''-id}
\end{equation}
for every modulation function $\Psi$, $u\in D(\OPT(a))$; which by 
definition of $\OPT$ is equivalent to
\begin{multline}
  \lim_{k\to\infty }\lim_{m\to\infty } 
    \OP(\Psi(2^{-m}D_x)a(x,\eta)\Psi(2^{-m}\eta)\psi(2^{-k}\eta))u
\\
  =\lim_{m\to\infty }\lim_{k\to\infty } 
  \OP(\Psi(2^{-m}D_x)a(x,\eta)\Psi(2^{-m}\eta)\psi(2^{-k}\eta))u.
  \label{Stein'''-id}
\end{multline} 
So the question is reduced to that of commuting the two
limits above. This is mentioned just to explicate the difficulties
in the present question, and indeed in the theory as a whole.
\end{rem}

\section{The maximal smooth space}
\label{C8-ssect}\noindent
It it turns out that every type $1,1$-operator $a(x,D)$, 
despite the pathologies it
may display, always is defined on the largest possible space of smooth
functions, which is $C^\infty \bigcap\cal S'$, of course.

The proof of this fact departs from the extension $\tilde A$ of G.~Bourdaud
recalled in Section~\ref{GBext-ssect}. To free the discussion from the slow
growth in $\cal O_M$,
one may restate the definition of $\tilde Af$ in terms of the tensor product $1\otimes f$ in 
$\cal S'(\Rn\times \Rn)$ acting on $(\varphi\otimes (1-\chi))K\in \cal
S(\Rn\times \Rn)$, ie
\begin{equation}
  \dual{\tilde Af}{\varphi}= \dual{a(x,D)(\chi f)}{\varphi}
   +\dual{1\otimes f}{(\varphi\otimes (1-\chi))K},
  \label{aOC'-eq}
\end{equation}
The advantage here is that both terms obviously make sense as long as
$f$ is smooth and temperate, ie for every 
$f\in C^\infty (\Rn)\bigcap\cal S'(\Rn)$. Moreover, the arguments in 
Section~\ref{GBext-ssect} can then essentially be repeated, which shows that
every $f\in C^\infty \bigcap\cal S'$ is mapped by $\tilde A$ to a well
defined distribution; cf \cite[Sect.~2.1.2]{JJ10tmp}.

Invoking Definition~\ref{vfm-defn}, this gives that
$a(x,D)$ always is a map defined on the
\emph{maximal} set of smooth functions $C^\infty \bigcap \cal S'$:

\begin{thm}[{\cite[Thm.~2.7]{JJ10tmp}}]   \label{aOO-thm}
Every $a(x,D)\in \OP(S^d _{1,1}(\Rn\times \Rn))$ restricts to a map
\begin{equation}
  a(x,D)  \colon C^\infty (\Rn)\bigcap \cal S'(\Rn)\to C^\infty (\Rn),
\end{equation}
which maps the subspace $\cal O_M(\Rn)$ into itself. The restriction is given by \eqref{aOC'-eq}.
\end{thm}

\begin{proof}
Let for brevity $A_m=\OP(\psi(2^{-m}D_x)a(x,\eta)\psi(2^{-m}\eta))$
with distribution kernel $K_m$, so  that $a(x,D)u=\lim_m A_mu$ when 
$u\in D(a(x,D))$.
With $f\in C^\infty \bigcap \cal S'$ and $\varphi, \chi$ as above,
this is the case for $u=\chi f\in C^\infty_0$.

Since the support of $\varphi\otimes (1-\chi)$ is
disjoint from the diagonal and bounded in the $x$-direction, it was shown
by means of the Regular Convergence Lemma in Proposition~\ref{KmKS-prop}
that in the topology of $\cal S(\Rn\times \Rn)$
\begin{equation}
  \varphi(x)(1-\chi(y))K_m(x,y) \xrightarrow[m\to\infty ]{~}\varphi(x)(1-\chi(y)) K(x,y).
\end{equation}
Exploiting these facts in \eqref{aOC'-eq} yields that
\begin{equation}
  \dual{\tilde Af}{\varphi}= \lim_m \dual{A_m(\chi f)}{\varphi}
   +\lim_m\iint K_m(x,y)(1-\chi(y))f(y)\varphi(x)\,dy\,dx.
  \label{aOCm-eq}
\end{equation}
Here the integral equals $\dual{A_m(f-\chi f)}{\varphi}$ by the kernel
relation, for $A_m\in S^{-\infty }$ and $f$ may as an element of $\cal S'$ 
be approached from $C^\infty_0$. So \eqref{aOCm-eq} yields
\begin{equation}
  \dual{\tilde Af}{\varphi}=\lim_m \dual{A_m(\chi f)}{\varphi}
   +\lim_m \dual{A_m(f-\chi f)}{\varphi} =
  \lim_m \dual{A_m f}{\varphi}.
\end{equation}
Thus $A_mf\to \tilde Af$, which is independent of $\psi$. 
Hence $\tilde A\subset a(x,D)$ as desired.

Moreover, $\tilde Af$ is smooth because $a(x,D)(f\chi)\in \cal S$ while
the other contribution in \eqref{aOC'-eq} also acts like a $C^\infty
$-function: for a suitable $\tilde \varphi\in C^\infty _0$ chosen to be $1$ around
$\supp\varphi$  the second term equals
\begin{equation}
  \int \dual{f}{(\tilde \varphi(x)(1-\chi))K(x,\cdot )}\varphi(x)\,dx,
\end{equation}
where $x\mapsto \dual{f}{\tilde
\varphi(x)(1-\chi(\cdot ))K(x,\cdot )}$ is $C^\infty $ as seen in the
verification that $g\otimes f\in \cal S'(\Rn\times \Rn)$ for $f,g\in \cal S'(\Rn)$.
Therefore $\tilde Af$ is locally smooth, so $\tilde Af \in C^\infty (\Rn)$ follows.

When in addition $f\in\cal O_M$, then $(1+|x|)^{-2N}D^\alpha \tilde Af $ is
bounded for sufficiently large $N$, for when
$r=\op{dist}(\supp\varphi,\supp(1-\chi))$ one 
finds in the second contribution to \eqref{aOC-eq} that
\begin{equation}
  \begin{split}
  (1+|y|)^{2N}|D^\alpha_x K(x,y)|&\le 
  (1+|x|)^{2N}\max(1,1/r)^{2N}(r+|x-y|)^{2N}|D^\alpha_x K(x,y)|
\\
  &\le c(1+|x|)^{2N}\sup_{x\in\Rn}\int |D^\alpha_x(2\lap_\eta)^N a(x,\eta)|\,d\eta,  
  \end{split}
\end{equation}
where the supremum is finite for $2N>d+|\alpha|+n$ whilst
$(1+|y|)^{-2N}f(y)$ is in $L_1$ for large $N$.
Hence $\tilde Af\in \cal O_M$ as claimed.
\end{proof}

In view of the theorem, the
difficulties for type $1,1$-operators are unrelated to growth at infinity 
for smooth functions.
The space $C^\infty (\Rn)\bigcap \cal S'(\Rn)$ clearly contains
functions of non-slow growth, eg
\begin{equation}
  f(x)=e^{x_1+\dots +x_n}\cos(e^{x_1+\dots +x_n}).
\end{equation}
The codomain $C^\infty $ in Theorem~\ref{aOO-thm} is of course not
contained in $\cal S'$, but this is
consistent with the use of $\cal D'$ in Definition~\ref{vfm-defn}.

\section{The pseudo-local property of 
type $\mathbf{1},\mathbf{1}$-operators}
\label{psdl-ssect}\noindent
For a classical pseudo-differential operator, 
say with symbol $a$ in $S^\infty _{1,0}$, 
it is a well known fact that $a(x,D)$ has the so-called pseudo-local
property.  
This means that it cannot create singularities, ie
\begin{equation}
  \singsupp a(x,D)u\subset \singsupp u
  \quad\text{for all}\quad u\in D(a(x,D)).
  \label{psdl-eq}
\end{equation}
In this connection the domain $D(a(x,D))$ is simply $\cal S'(\Rn)$.

For a type $1,1$-operator $a(x,D)$ the above formulation also applies, although
it is a delicate task to determine $D(a(x,D))$ exactly. 
But Definition~\ref{vfm-defn} characterises $D(a(x,D))$ as the set of
distributions $u$ for which the limit there exists; and this suffices for the
proof of 

\begin{thm}[{\cite[Thm.~6.4]{JJ08vfm}}]   \label{psdl-thm}
Every operator $a(x,D)$ in 
$\OP(S^{\infty }_{1,1}(\Rn\times \Rn))$ has the
pseudo-local property \eqref{psdl-eq}.
\end{thm}

This was partly anticipated already in 1978 
by C.~Parenti and L.~Rodino \cite{PaRo78}.
More precisely, they formulated it as a result for the case where
$a(x,D)$ is defined on the full distribution space (ie $\cal E'(\Rn)$
in their context of locally estimated symbols), 
but as justification they only observed 
that the distribution kernel $K(x,y)$ is $C^\infty $ for $x\ne y$, tacitly
leaving it to the reader to invoke the rest of the standard proof.

However, this is first of all not straightforward as the usual rules of 
pseudo-differential calculus are not available for type $1,1$-operators 
(a consequence of Ching's counter-example \cite{Chi72}), so the
question is rather more complicated than the impression \cite{PaRo78} gives.
But secondly, the remedy has to be sought, it seems, in the part
they suppressed\,---\,namely, the substitute for the rules of calculus
can be found precisely in the very definition of type $1,1$-operators, where
the vanishing frequency modulation yields a useful regularisation.

This became clear with the proof in \cite{JJ08vfm} of the above theorem. 
Indeed, 
in the three decades since \cite{PaRo78}, a proof of Theorem~\ref{psdl-thm} has
not been known. This is of course with good reason, for both sides of
\eqref{psdl-eq} are empty for $u\in \cal S$, so one has to treat $u\in \cal
S'\setminus \cal S$ directly. Obviously this requires a precise definition of
$a(x,D)$ as well as further ideas to handle the possible 
discontinuity in $\cal S'$ of $a(x,D)$.

To explain this, recall that the crux of the standard proof is to show that
$C^\infty (\Rn)$ contains a certain localised term, 
namely $\psi(x)a(x,D)(\chi_1 u)(x)$ in the notation of \cite[(6.18)]{JJ08vfm}. 
Hereby $\psi \in C^\infty_0$ and $\chi_1\in C^\infty$ have disjoint supports,
so the distribution kernel $\widetilde{K}(x,y)=\psi(x)K(x,y)\chi_1(y)$ 
of the composite map $u\mapsto \psi a(x,D)(\chi_1 u)$
belongs to $\cal S(\Rn\times \Rn)$, whence it suffices as usual to
establish that 
\begin{equation}
  \dual{\psi a(x,D)(\chi_1 u)}{\varphi}=
  \dual{\varphi\otimes u}{\widetilde{K}}
  \quad\text{for all}\quad \varphi\in C^\infty_0(\Rn).
  \label{axDloc-eq}
\end{equation}
In fact, the right-hand side equals $\int \dual{u}{\widetilde{K}(x,\cdot
)}\varphi(x)\,dx$ by the definition of the tensor product, so that
$\psi a(x,D)(\chi_1 u)=\dual{u}{\widetilde{K}(x,\cdot )}$,
where the last expression is $C^\infty $.

However, even though both sides of \eqref{axDloc-eq} make sense as they
stand, it is, because of the lack of continuity of $a(x,D)$, 
not a trivial task to show from scratch that they are equal. 
As indicated above, the details of the verification do not follow
well-trodden paths:

The proof strategy in \cite[Thm.~6.4]{JJ08vfm} was to utilise 
the regularisation that one is given gratis
from Definition~\ref{vfm-defn}.
This departs from Proposition~\ref{KmKS-prop} that gives,
because $\psi\otimes \chi_1$ has support disjoint from the diagonal, 
\begin{equation}
  \psi(x)K_m(x,y)\chi_1(y)\xrightarrow[m\to\infty ]{~}
\psi(x)K(x,y)\chi_1(y)=\tilde K(x,y)
  \quad\text{in}\quad \cal S(\Rn\times \Rn).
\end{equation}
Using this on the right-hand side of \eqref{axDloc-eq}, and approaching $u$
by a sequence $u_l$ from $C^\infty_0(\Rn)$,
the formula follows from the usual continuity properties:
\begin{equation}
  \begin{split}
  \dual{\varphi\otimes u}{\widetilde{K}}
&= \lim_m 
  \dual{\varphi\otimes u}{(\psi\otimes \chi_1)K_m}_{\cal S'\otimes \cal S}
\\
&= \lim_m\lim_l 
  \dual{\varphi\otimes u_l}{(\psi\otimes \chi_1)K_m}_{\cal S'\otimes \cal S}
\\
&= \lim_m\dual{a^m(x,D)(\chi_1 u)^m}{\psi\varphi}
=  \dual{\psi a(x,D)(\chi_1 u)}{\varphi}.  
  \end{split}
\end{equation}
In this calculation the order of the limits is essential, of course. 
But it might
be instructive to note that in case $a(x,D)$ \emph{is} continuous on $\cal
S'$, then $u_l\to u$ and frequency modulation on the left-hand side of
\eqref{axDloc-eq} gives $\lim_{l}\lim_m\dual{K_m}{u_l\otimes \varphi}$, with
limits in reverse order.

The full proof of the pseudo-local property in \cite[Thm.~6.4]{JJ08vfm} is
not much longer. However, the above Schwartz space
convergence of the distribution kernels required some preparation as noted
around Proposition~\ref{KmKS-prop}.

\section{Non-preservation of wavefront sets}
\label{nwf-ssect}\noindent
In the counter-examples based on Ching's symbol $a_{\theta}(x,\eta)$
in \eqref{Ching-eq}, the role of the exponentials is
to move all frequencies in the spectrum of the functions to a
neighbourhood of the origin. 
Therefore it is perhaps not surprising that another variant of Ching's example 
will produce frequencies $\eta$ that are moved to, say $-\eta$. 

So, although $a_\theta(x,D)$ cannot create singularities, cf
Section~\ref{psdl-ssect}, at the singular
points of $u(x)$ it may change all the high frequencies causing them.

This indicates that type $1,1$-operators need not have
the microlocal property. That is, the inclusion among wavefront sets 
\begin{equation}
 \op{WF}(a(x,D)u)\subset\op{WF}(u)  ,
\end{equation}
that always holds for pseudo-differential operators of type $1,0$,
is violated for certain symbols $a\in S^\infty_{1,1}$ and
distributions $u$. 

This was confirmed with explicit calculations in \cite[Sec.~3.2]{JJ08vfm}, 
following C.~Parenti and L.~Rodino \cite{PaRo78}
who treated $d=0$ and $n=1$. The programme they suggested was carried out
for all $d\in \R$, $n\in \N$ and arbitrary directions of $\theta$. 
As a minor improvement, the wavefront sets was explicitly 
determined, and due to the fact that a certain $v$ below  has
compact spectrum (rather than compact support as in \cite{PaRo78}) and the 
uniformly estimated symbols, the proofs were also rather cleaner.

When $\theta\in \Rn$ is fixed with $|\theta|=1$, 
one can introduce
\begin{equation}
 w_{\theta}(x)=w(\theta,d;x)= 
 \sum_{j=1}^\infty 2^{-jd}e^{\im 2^j\theta\cdot x}v(x)  
\end{equation}
for some $v\in \cal S(\Rn)$ with $\supp\hat v\subset B(0,1/20)$;
or equivalently
\begin{equation}
  \hat w_{\theta}(\eta)
  = \sum_{j=1}^\infty 2^{-jd} \hat v(\eta-2^j\theta).
  \label{wtheta-eq}
\end{equation}
This distribution has the cone $\Rn\times
(\R_+\theta)$ as its wavefront set, as shown in \cite[Prop.~3.3]{JJ08vfm}.

The counter-example arises by considering $w_{\theta}$ together with the symbol
$a_{2\theta}\in S^d_{1,1}(\Rn\times\Rn)$ 
defined by \eqref{Ching-eq} with auxiliary function $A$ fulfilling in
addition 
\begin{equation}
  A(\eta)=1 \quad\text{for}\quad
  \tfrac{9}{10}\le |\eta|\le \tfrac{11}{10}.   
  \label{A(1)-eq}
\end{equation}

\begin{prop}   \label{WF-prop}
The distributions $w(\theta,d;x)$ are in $H^s(\Rn)$ precisely for $s<d$, and 
when $a_{2\theta}$ is chosen as in \eqref{Ching-eq},\eqref{A(1)-eq} 
with $|\theta|=1$, then 
\begin{equation}
  a_{2\theta}(x,D)w(\theta,d;x)=w(-\theta,0;x).   
\end{equation}
Moreover,
\begin{align}
  \op{WF}(w_{\theta})&= \Rn\times (\R_+\theta),
  \label{WF-eq}
\\
   \op{WF}(a_{2\theta}(x,D)w(\theta,d;x))&=\Rn\times(\R_+(-\theta)),
\end{align}
so the wavefront sets of $w_{\theta}$ and $a_{2\theta}(x,D)w_{\theta}$ are
disjoint. 
\end{prop}

Since the above, as indicated, is a minor improvement of the result that
has been known since \cite{PaRo78},  
it should suffice here to refer to \cite[Prop.~3.3]{JJ08vfm} for details.

But a few remarks should be in order.
As $A$ vanishes around $2\theta$, 
it is easy to see that in this case every $(\xi,\eta)$ in 
$\supp \hat a_{2\theta}$ lies in the cone $|\eta|\le 2|\xi+\eta|$ so that
$a$ fulfils \eqref{tdc-cnd} for $B=2$.
Hence $a_{2\theta}(x,D)$ has a large domain containing $\bigcup H^s$ and fulfils
the twisted diagonal condition\,---\,but then neither of these properties 
can ensure the microlocal property of a type $1,1$-operator, according to
Proposition~\ref{WF-prop}.

There is a clear reason why the counter-example $w_{\theta}$ in
Proposition~\ref{WF-prop} is singular on all of $\Rn$:
the function  $w_{\theta}(x)$ equals $ v(x)f(x\cdot \theta)$ where 
$v\in \cal F^{-1}C^\infty_0$ is analytic whilst
$f(t)=\sum_{j=1}^\infty 2^{-jd}e^{\im 2^jt}$ is not just 
highly oscillating, but for $0<d\le1$ 
equal to Weierstrass's continuous nowhere differentiable function, here in a
complex version with its wavefront set along a half-ray. 
The link to this classical construction 
(that could have substantiated the
argument for formula \mbox{(19)} in \cite{PaRo78})
was first observed in \cite[Rem.~3.5]{JJ08vfm}.

More remarks on the above $f$, including a short, explicit analysis of its
regularity properties, can be found in Remarks~3.6--3.8 in \cite{JJ08vfm}.
In particular the nowhere differentiability was obtained with a short
microlocalisation argument (further explored in \cite{JJ10now}).

\section{The support rule and its spectral version}
\label{supp-ssect}\noindent
This subject has alreay been explained in the context of type
$1,0$-operators in Section~\ref{srule-ssect}. It therefore suffices to
comment on the modifications needed for type $1,1$-operators.

First of all there is a satisfactory result for the extended kernel formula
and the support rule, which for $u\in \cal E'(\Rn)$ is the well-known inclusion
\begin{equation}
  \supp Au\subset {\supp K\circ \supp u}.  
\end{equation}
Here $\supp K\circ \supp u$ stands for the set
$\{\,x\mid \exists y\in \supp u\colon (x,y)\in \supp K \,\}$ as usual.

\begin{thm}   \label{supprule11-thm}
When $a\in S^\infty_{1,1}(\Rn\times\Rn)$ 
and $K$ denotes its kernel, then 
$\dual{a(x,D)u}{v}=\dual{K}{v\otimes u}$ 
whenever $u\in D(a(x,D))$, $v\in C^\infty_0(\Rn)$ fulfil 
\begin{gather}
  \supp K\bigcap \supp  v\otimes u \Subset\Rn\times\Rn,
  \label{suppKuv-eq}
  \\
  \singsupp K\bigcap \singsupp  v\otimes u=\emptyset.
  \label{ssuppKuv-eq}
\end{gather}
And for all $u\in D(a(x,D))$ the support rule holds, ie
$\supp a(x,D)u\subset \overline{\supp K\circ \supp u}$.
\end{thm}

This was proved in \cite[Thm.~8.1]{JJ08vfm} by approaching $u$ by a
regularly converging sequence from $C^\infty_0$ and applying
Proposition~\ref{KmKS-prop}. 

The rule for spectra was amply described in
Section~\ref{srule-ssect}, so the full statement for type $1,1$-operators
is just given here. Unfortunately it contains an undesirable assumption,
which usually is redundant, cf the last part of The Spectral Support Rule:

\begin{thm}[{\cite[Thm.~8.4]{JJ08vfm}}]
  \label{supp-thm}
Let $a\in  S^\infty_{1,1}(\Rn\times\Rn)$ and let 
$u\in D(a(x,D))$ be such that $a(x,D)u$ is temperate and that, for some
$\psi\in C^\infty_0(\Rn)$ equalling $1$ around the origin,
the convergence of Definition~\ref{vfm-defn} 
holds in the topology of $\cal S'(\Rn)$, ie
\begin{equation}
  a(x,D)u=\lim_{m\to\infty} a^m(x,D)u^m \quad\text{in}\quad \cal S'(\Rn).
  \label{S'lim-cnd}
\end{equation}
Then \eqref{ssKau-eq} holds, that is with $\Xi=\supp\cal K\circ\supp\hat u$
one has
\begin{gather}
   \supp\cal F(a(x,D)u)\subset\overline{\Xi},
 \\
  \Xi=\bigl\{\,\xi+\eta \bigm| (\xi,\eta)\in \supp\hat a,\ 
     \eta\in \supp\hat u \,\bigr\}.
  \label{sFAu-eq}
\end{gather}
When $u\in \cal F^{-1}\cal E'(\Rn)$ then \eqref{S'lim-cnd} 
holds automatically and $\Xi$ is closed for such $u$.
\end{thm}

In the theory of type $1,1$-operators it may seem unmotivated that the
partially Fourier transformed symbol $\hat a(\xi,\eta)$
plays such a prominent role (cf the twisted diagonal condition). 
But the spectral support rule \eqref{srule-eq} 
gives an explanation as $\hat a(\xi,\eta)$ appears in $\Xi$ too; thence $\hat
a(\xi,\eta)$ should be a natural object for every pseudo-differential
operator, as in Littlewood--Paley analysis control of $\supp\cal
Fa(x,D)u$ is a central theme.

However, $\hat a(\xi,\eta)$ is particularly important for operators of type
$1,1$, as the spectral support rule \eqref{sFAu-eq}  
clearly shows that the role of the twisted 
diagonal condition \eqref{tdc-cnd} is to ensure that $a(x,D)$ cannot change
(large) frequencies in $\supp\hat u$ to $0$: \eqref{tdc-cnd} means
that $\xi$ cannot be close to $-\eta$ when $(\xi,\eta)\in \supp\hat a$,
which by \eqref{sFAu-eq} means that $\eta\in \supp\hat u$ will be changed
to $\xi+\eta\ne 0$.  

The proof of  Theorem~\ref{supp-thm} was given first in
\cite{JJ04DCR,JJ05DTL} in a special case, and the full result appeared 
in in \cite[Thm.~8.3]{JJ08vfm}. The main ingredient was to
obtain an extended version of the kernel formula for the conjugated operator
$\cal Fa(x,D)\cal F^{-1}$. This is recalled from \cite[Thm.~8.2]{JJ08vfm}
here.

\begin{thm}   \label{Kuv-thm}
Let $a\in S^\infty_{1,1}(\Rn\times\Rn)$ and let 
the distribution kernel of $\cal F a(x,D)\cal F^{-1}$ be
denoted by $\cal K(\xi,\eta)$.
When $u\in D(a(x,D))$ is such that, for some $\psi$
as in Definition~\ref{vfm-defn}, 
\begin{equation}
  a(x,D)u=\lim_{m\to\infty} a^m(x,D)u^m \quad\text{holds in}\quad \cal S'(\Rn),
  \label{S'conv-eq}
\end{equation}
and when $\hat v\in C^\infty_0(\Rn)$ satisfies
\begin{equation}
  \supp\cal K\bigcap \supp \hat v\otimes\hat u \Subset\Rn\times\Rn,
\qquad
  \singsupp\cal K\bigcap \singsupp \hat v\otimes\hat u=\emptyset,
  \label{suppFKuv-eq}
\end{equation}
then it holds
\begin{equation}
  \dual{\cal F a(x,D)\cal F^{-1}(\hat u)}{\hat v}
  =\dual{\cal K}{\hat v\otimes\hat u},
  \label{FaFK-eq}
\end{equation}
with extended action of $\dual{\cdot}{\cdot}$ on the right-hand side.
\end{thm}

This was also obtained using regular convergence of test functions; cf 
\cite[Thm.~8.2]{JJ08vfm}. The convergence in the topology of $\cal S'$ seems
to be necessary for technical reasons in this proof. Anyhow, the assumption
that $a(x,D)u$ be an element of $\cal S'$ of course serves the purpose of
making $\cal Fa(x,D)u$ defined.

The remarks made in Section~\ref{srule-ssect} also apply here:
elementary symbols are redundant because of
Theorem~\ref{supp-thm}, but in the type $1,1$-context this simplification is
particularly  important, for else Definition~\ref{vfm-defn} would contain a
double limit procedure, with severe complications for the entire theory.

\chapter{Continuity results}
\label{resultsII-sect}\noindent
Here the results summarised in
\eqref{LPa-pro}--\eqref{Fspq'-pro} in Section~\ref{dozen-ssect} are
described in detail.

\section{Littlewood--Paley decompositions of 
type $\mathbf{1},\mathbf{1}$-operators}
\label{LPa-ssect}\noindent
Here it is described how the definition by vanishing frequency
modulation leads at once to Littlewood--Paley analysis of
$a(x,D)u$, and in particular 
to the paradifferential decomposition 
mentioned in Section~\ref{uni-ssect}.

\bigskip

\subsection{Dyadic corona decompositions of symbols and operators}
  \label{dyadic-sssect}
The basic step is to obtain a Littlewood--Paley
decomposition from the 
modulation function $\psi$ used in Definition~\ref{vfm-defn}.

As a preparation one may obviously, 
to each $\psi\in C^\infty_0(\Rn)$ with
$\psi=1$ in a neighbourhood of $0$, fix $R>r>0$ satisfying
\begin{equation}
  \psi(\xi)=1\quad\text{for}\quad |\xi|\le r;
  \qquad
  \psi(\xi)=0\quad\text{for}\quad |\xi|\ge R\ge 1.
  \label{Rr-eq}
\end{equation}
It is also convenient to 
take an integer $h\ge 2$ so large that $2R< r2^h$.

As an auxiliary function one has
$\varphi=\psi-\psi(2\cdot )$. Dilations of this function are
supported in coronas, eg
\begin{equation}
  \supp \varphi(2^{-k}\cdot )\subset \bigl\{\,\xi \bigm| 
   r2^{k-1}\le|\xi|\le R2^k\,\bigr\},
  \qquad \text{for }k\ge 1,
  \label{phi-eq}
\end{equation}
and by calculating the telescopic sum
\begin{equation}
  \psi(\xi)+\varphi(\xi/2)+\dots+\varphi(\xi/2^m)=\psi(2^{-m}\xi),
  \label{tele-eq}
\end{equation}
it follows  by letting $m\to\infty $ that one has the Littlewood--Paley
partition of unity
\begin{equation}
  1=\psi(\xi)+\sum_{k=1}^\infty \varphi(2^{-k}\xi), 
  \quad\text{for each $\xi\in \Rn$}.
\end{equation}

Using this, $u(x)$ can now be (micro-)localised eg to
frequencies $|\eta|\approx 2^j$ by setting
\begin{equation}
  u_j=\varphi(2^{-j}D)u, \qquad
  u^j=\psi(2^{-j}D)u.
\end{equation}
Note that localisation to balls given by $|\eta|\le R2^j$ are written with
upper indices.
For symbols $a(x,\eta)$ similar conventions apply to the first variable,
\begin{align}
  a_j(x,\eta)&=\varphi(2^{-j}D_x)a(x,\eta)=
   \cal F^{-1}_{\xi\to x}(\varphi(2^{-j}\xi)\cal F_{x\to\xi}a(x,\eta)).
 \\
  a^j(x,\eta)&=\psi(2^{-j}D_x)a(x,\eta)=
   \cal F^{-1}_{\xi\to x}(\psi(2^{-j}\xi)\cal F_{x\to\xi}a(x,\eta)).
\end{align}
By convention $u^j=u_j$ and $a^j=a_j$ for $j=0$;  they are all
taken to equal $0$ for indices $j<0$. 
(To avoid having several meanings of sub-
and superscripts, dilations $\psi(2^{-j}\cdot  )$ are written as such,
with the corresponding Fourier multiplier as $\psi(2^{-j}D)$, and similarly
for $\varphi$). Note that as a consequence one has for operators that, eg,
$a_j(x,D)=\OP(\varphi(2^{-j}D_x)a(x,\eta))$.

Thus prepared with these classical dyadic corona decompositions, the point
of this section is to
insert the relation \eqref{tele-eq} twice in $a^m(x,D)u^m$, 
cf \eqref{amum-eq}, and apply bilinearity. This gives
\begin{equation}
  \begin{split}
  a^m(x,D)u^m&=
  \OP\big((\psi(D_x)+\varphi(2^{-1}D_x)+\dots+\varphi(2^{-m}D_x))a(x,\eta)\big)
(u_0+u_1+\dots +u_m)
\\
  &= \sum_{j,k=0}^m   a_j(x,D)u_k.
  \end{split}
  \label{bilin-eq}
\end{equation}
Of course this sum may be split in three groups in which $j\le k-h$,
$|j-k|<h$ and $k\le  j-h$, respectively.
Proceeding to the limit of vanishing frequency modulation, ie $m\to\infty $, 
one is therefore lead to consider the three infinite series 
\begin{align}
    a_{\psi}^{(1)}(x,D)u&=\sum_{k=h}^\infty \sum_{j\le k-h} a_j(x,D)u_k
  =\sum_{k=h}^\infty a^{k-h}(x,D)u_k
  \label{a1-eq}\\
  a_{\psi}^{(2)}(x,D)u&= \sum_{k=0}^\infty
               \bigl(a_{k-h+1}(x,D)u_k+\dots+a_{k-1}(x,D)u_k+a_{k}(x,D)u_k
\notag\\[-2\jot]
   &\qquad\qquad
                +a_{k}(x,D)u_{k-1} +\dots+a_k(x,D)u_{k-h+1}\bigr) 
  \label{a2-eq}\\
   a_{\psi}^{(3)}(x,D)u&=\sum_{j=h}^\infty\sum_{k\le j-h}a_j(x,D)u_k
   =\sum_{j=h}^\infty a_j(x,D)u^{j-h}.
  \label{a3-eq}
\end{align}

More precisely one has

\begin{prop}   \label{a123-prop}
If $a(x,\eta)$ is of type $1,1$ and $\psi$ is an arbitrary modulation
function, then the paradifferential decomposition
\begin{equation}
  a_{\psi}(x,D)u=
  a_{\psi}^{(1)}(x,D)u+a_{\psi}^{(2)}(x,D)u+a_{\psi}^{(3)}(x,D)u
  \label{a123-eq}
\end{equation}
holds for all $u\in \cal S'(\Rn)$ fulfilling that 
the three series converge in $\cal D'(\Rn)$.
\end{prop}

One should note the shorthand $a^{k-h}(x,D)$ for 
$\sum_{j\le k-h}a_j(x,D)=\op{OP}(\psi(2^{h-k}D_x)a(x,\eta))$ etc.
In this way \eqref{a2-eq} also has a brief form, namely
\begin{equation}
 a_{\psi}^{(2)}(x,D)u=\sum_{k=0}^\infty
((a^{k}-a^{k-h})(x,D)u_k+a_k(x,D)(u^{k-1}-u^{k-h})). 
  \label{a2two-eq}
\end{equation}

The importance of the decomposition in \eqref{a1-eq}--\eqref{a3-eq} 
lies in the fact that the summands have their spectra in balls and coronas.
This has been anticipated since the 1980's, and verified for elementary
symbols, eg in \cite[Thm.~1]{Bou88}. In general it follows directly from the
spectral support rule in Theorem~\ref{supp-thm}:

\begin{prop}  \label{corona-prop}
When $a\in S^d_{1,1}(\Rn\times \Rn)$ and $u\in \cal S'(\Rn)$, and 
$r$, $R$ are chosen as in \eqref{Rr-eq} for each modulation function $\psi$,
then every $h\in \N$ such that $2R< r2^h$ gives
\begin{align}
  \supp\cal F(a^{k-h}(x,D)u_k)&\subset
  \bigl\{\,\xi\bigm| 
  R_h2^k\le|\xi|\le \frac{5R}{4} 2^k\,\bigr\}
  \label{supp1-eq}  \\
  \supp\cal F(a_k(x,D)u^{k-h})&\subset
  \bigl\{\,\xi \bigm| 
  R_h2^k\le|\xi|\le \frac{5R}{4} 2^k\,\bigr\},
  \label{supp3-eq}
\end{align}
where $R_h=\tfrac{r}{2}-R2^{-h}>0$.
\end{prop}
\begin{proof}
Since the type $1,1$-operator $a^{k-h}(x,D)$ is defined on
$u_k\in \cal F^{-1}\cal E'$,
the last part of Theorem~\ref{supp-thm} and \eqref{phi-eq} give
\begin{equation}
  \supp\cal F(a^{k-h}(x,D)u_k)\subset
  \bigl\{\,\xi+\eta \bigm| (\xi,\eta)\in \supp
    (\psi_{h-k}\otimes1)
    \hat  a,\quad 
  r2^{k-1}\le |\eta|\le R2^k \,\bigr\}.
\end{equation}
So by the triangle inequality every $\zeta=\xi+\eta$ 
in the support fulfils
\begin{equation}
  r2^{k-1}-R2^{k-h}
 \le|\zeta|\le R2^{k-h}+R2^k\le \tfrac{5}{4}R2^k,
\end{equation}
as $h\ge 2$.
This shows \eqref{supp1-eq} and \eqref{supp3-eq} follows analogously.
\end{proof}

In contrast with this, the terms in $a^{(2)}(x,D)u$ only satisfy a dyadic ball
condition, as was first observed for functions $u$ in \cite{JJ05DTL}.
But when the twisted diagonal condition \eqref{tdc-cnd} holds, 
the situation improves for large $k$:

\begin{prop}   \label{ball-prop}
When $a\in S^d_{1,1}(\Rn\times \Rn)$, $u\in \cal S'(\Rn)$, and 
$r$, $R$ are chosen as in \eqref{Rr-eq} for each auxiliary function $\psi$,
then every $h\in \N$ such that $2R< r2^h$ gives
\begin{equation}
  \supp\cal F\big(a_k(x,D)(u^{k-1}-u^{k-h})+(a^k-a^{k-h})(x,D)u_k\big)\subset
  \bigl\{\,\xi\bigm| |\xi|\le 2R 2^k\,\bigr\}
  \label{supp2-eq}
\end{equation}
If $a(x,\eta)$ satisfies \eqref{tdc-cnd} for some $B\ge 1$, 
the support is eventually contained in the corona
\begin{equation}
  \bigl\{\,\xi \bigm| \frac{r}{2^{h+1}B} 2^k \le |\xi|\le 2R2^k\,\bigr\}.
  \label{supp2'-eq}
\end{equation}
\end{prop}

A proof of the general case above can be given with the same
techniques as for 
Proposition~\ref{corona-prop}; cf \cite[Prop.~5.3]{JJ10tmp}. 

The decomposition in Proposition~\ref{a123-prop} is also useful because the
terms of the series can be conveniently estimated at every $x\in \Rn$, using
the pointwise approach of Section~\ref{pe-ssect}. For later reference, these
estimates are collected in the next result.

\begin{prop}
  \label{a123pe-prop}
For every $a(x,\eta)$ in $S^d_{1,1}(\Rn\times \Rn)$ and $u\in \cal S'(\Rn)$
there are pointwise estimates for $x\in \Rn$,
\begin{align}
  |a^{k-h}(x,D)u_k(x)|&\le p(a) (R2^{k})^d u_k^*(N,R2^k;x),
  \label{a1-pe} \\
  |(a^k-a^{k-h})(x,D)u_k(x)|& \le  p(a) (R2^{k})^d u_k^*(N,R2^k;x),
  \label{a2-pe} \\
  |a^k(x,D)(u^{k-1}(x)-u^{k-h}(x))|& \le p(a) (R2^{k})^d 
  \sum_{l=1}^{h-1} 2^{-ld}u_{k-l}^*(N,R2^{k-l};x),
  \label{a2'-pe} \\
  |a_j(x,D)u^{j-h}(x)|&\le c_M 2^{-jM} p(a) 
  \sum_{k=0}^j (R2^{k})^{d+M} u_k^*(N,R2^k;x).
  \label{a3-pe}
\end{align}
Hereby $p(a)$ denotes a continuous seminorm on $S^d_{1,1}$ and $M\in \N$.
\end{prop}
\begin{proof}
From the factorisation inequality in Section~\ref{pe-ssect} it follows that
\begin{equation}
  |a^{k-h}(x,D)u_k(x)|\le F_{a^{k-h}}(N,R2^k;x)u_k^*(N,R2^k;x)
  \le c_1\nrm{\cal F^{-1}\psi}{1}p(a)(R2^k)^d u_k^*(x).
\end{equation}
Indeed, the estimate of $F_{a^{k-h}}$ builds on
Corollary~\ref{Fa-cor}. Since $\hat u_k$ is supported in a corona with outer
radius $R2^k$, this states that
\begin{equation}
  F_{a^{k-h}}(x)\le c p(a^{k-h})(R2^k)^d.
\end{equation}
Here it easy to see from a convolution estimate that
$p(a^{k-h})\le c'p(a)$, since eg
\begin{equation}
  |D^\alpha_\eta D^\beta_x a^{k}(x,\eta)|
\le \int |2^{nk}\check \psi(2^{k}y)
          D^\alpha_\eta D^\beta_x a(x-y,\eta)|\,dy
\le p_{\alpha,\beta}(a)(1+|\eta|)^{d-|\alpha|+|\beta|}\int |\check \psi|\,dy.
\end{equation}
Hence the first claim is obtained.

The estimate of $(a^k-a^{k-h})(x,D)u_k$ is highly similar. In fact the only
change is in the integrals, where by the definition of the symbol
$a^k-a^{k-h}$ the last integrand should be
$|\cal F^{-1}(\psi-\psi(2^h\cdot ))|$ instead of $|\cal F^{-1}\psi|$.

In the third line one has
\begin{equation}
  a^k(x,D)(u^{k-1}(x)-u^{k-h}(x))=a^k(x,D)u_{k-1}(x)+\dots 
   +a^k(x,D)u_{k-h+1}(x).
\end{equation}
Here each term is estimated as above, now with the factor $(R2^{k-l})^d$
as a result for $l=1,\dots ,h-1$; as this is $2^{-ld}$ times 
$(R2^k)^d$, the claim follows.

The last inequality results from the fact that, by \eqref{tele-eq} one has
$u^{j-h}=u_0+\dots +u_{j-h}$, whence a crude estimate yields
\begin{equation}
    |a_j(x,D)u^{j-h}(x)|\le \sum_{k=0}^{j-h} |a_j(x,D)u_k(x)| 
   \le \sum_{k=0}^{j} F_{a_j}(N,R2^k;x)u_k^*(N,R2^k;x).
  \label{aj-ineq}
\end{equation}
Since $\hat a_j(\xi,\eta)=\varphi(2^{-j}\xi)\hat a(\xi,\eta)$ vanishes
around the origin (there is nothing to show if $j<h$), 
Corollary~\ref{Fa0-cor} gives the estimate
$F_{a_j}(x)\le c_M p(a)2^{-jM}(R2^k)^{d+M}$ for $k\ge 1$, as for such $k$ an
auxiliary function supported in a corona may be used in $F_{a_j}$. 
The case $k=0$ is similarly estimated if
$c_M$ is increased by a power of $R$; cf Corollary~\ref{Fa0-cor}. This
completes the proof.
\end{proof}

\subsection{Calculation of symbols and remainder terms}
In connection with  the paradifferential splitting
\eqref{a1-eq}--\eqref{a3-eq} there is an extra task
for type $1,1$-operators, because the operator notation
$a^{(j)}(x,D)$ requires a more explicit justification in this context.

Departing from the right hand sides of \eqref{a1-eq}--\eqref{a3-eq} one
finds the following symbols,
\begin{align}
    a^{(1)}(x,\eta)&=\sum_{k=h}^\infty a^{k-h}(x,\eta)\varphi(2^{-k}\eta)
  \label{a1'-eq}
\\
  a^{(2)}(x,\eta)&=\sum_{k=0}^\infty
        \big((a^{k}(x,\eta)-a^{k-h}(x,\eta))\varphi(2^{-k}\eta)
       +a_k(x,\eta)(\psi(2^{-(k-1)}\eta)-\psi(2^{-(k-h)}\eta))\big)
  \label{a2'-eq}
\\
   a^{(3)}(x,\eta)& =\sum_{j=h}^\infty a_j(x,\eta)\psi(2^{-(j-h)}\eta).
  \label{a3'-eq}
\end{align}
Not surprisingly,
these series converge in the Fr\'echet space $S^{d+1}_{1,1}(\Rn\times
\Rn)$, for the sums are locally finite. 

More intriguingly, an inspection shows that 
both $a^{(1)}(x,\eta)$ and $a^{(3)}(x,\eta)$ fulfil the
twisted diagonal condition \eqref{tdc-cnd}; cf \cite[Prop.~5.5]{JJ10tmp}.

However, before this can be applied,
it is clearly necessary to
verify that the type $1,1$-operators corresponding to
\eqref{a1'-eq}--\eqref{a3'-eq} are in fact given by the
infinite series in \eqref{a1-eq}--\eqref{a3-eq}. 
In particular its is a natural programme to show that the
series for $a^{(j)}(x,D)u$, $j=1,2,3$, 
converges precisely when $u$ belongs to the domain of $a^{(j)}(x,D)$. 

But in view of the definition by
vanishing frequency modulation in \eqref{apsi-eq},
this will necessarily be lengthy because a second modulation function has to
be introduced.

To indicate the details for $a^{(1)}(x,\eta)$, let 
$\psi, \Psi\in C^\infty_0(\Rn)$ be
equal to $1$ around the origin, and let $\Psi$ be used as the fixed
auxiliary function entering the symbol $a^{(j)}_{\Psi}(x,\eta)$;
and set $\Phi=\Psi-\Psi(2\cdot )$. The numbers $r, R$ and $h$ are 
then chosen in relation to $\Psi$ as in \eqref{Rr-eq}; and one can take
$\lambda<\Lambda$ playing similar roles for $\psi$.
Moreover, $\psi$ is used for the frequency modulation entering the
definition of the \emph{operator} $a^{(1)}_{\Psi}(x,D)$.

As shown in \cite{JJ10tmp}, this gives the following identity for $u\in \cal
S'(\Rn)$, where prime indicates a finite sum and $\mu=[\log_2(\lambda/R)]$,
\begin{multline}
  \OP(\psi(2^{-m}D_x)a^{(1)}(x,\eta)\psi(2^{-m}\eta))u=
   \sum_{k=h}^{m+\mu}a^{k-h}(x,D)u_k
\\
   + \sideset{}{'}\sum_{\mu<l<1+\log_2(\Lambda/r)}
     \OP(\psi(2^{-m}D_x)\Psi(2^{h-l-m}D_x)a(x,\eta)
     \Phi(2^{-m-l}\eta)\psi(2^{-m}\eta))u.
  \label{a1u-eq}
\end{multline}

To complete the abovementioned programme,
it remains to let $m\to\infty $ in \eqref{a1u-eq}, whereby the first term
on the right-hand side converges to the limit in \eqref{a1-eq}.
But clearly it should also be shown that the remainder terms in the 
primed sum can be safely ignored.

This turns out to be possible for all $u\in \cal S'(\Rn)$, and as a result
it is true that the type $1,1$-operator $a^{(1)}_{\Psi}(x,D)u$ simply equals
the infinite series $   \sum_{k=h}^{\infty }a^{k-h}(x,D)u_k$, with
convergence for all $u$ in the domain $D(a^{(1)}_{\Psi}(x,D))=\cal S'(\Rn)$.
However, the details of this will be indicated during the disussion of 
Theorem~\ref{tdc-thm} below.

For the other operators $a^{(2)}_{\Psi}(x,D)$ and $a^{(3)}_{\Psi}(x,D)$
there are similar calculations, 
yielding remainder terms analogous to the primed
sum above; cf \cite[Sect.~5]{JJ10tmp}. The outcome is equally satisfying for
$a^{(3)}(x,D)$, which is also defined on $\cal S'(\Rn)$ and 
given by the infinite series. This extends to $a^{(2)}(x,D)$, at least if
\eqref{tdc-cnd} holds (for weaker conditions, cf Theorem~\ref{sigma-thm}). 

As indicated these issues will be taken up again in
Section~\ref{twist-ssect}.
A discussion of the applications of the paradifferential decomposition
\eqref{a123-eq} in Proposition~\ref{a123-prop} was begun already in
Section~\ref{para-sssect}, and it continues in
Section~\ref{twist-ssect}--\ref{Fspq-ssect} below.

\section{The twisted diagonal condition}  
\label{twist-ssect}\noindent
For convenience it is recalled that boundedness
$a(x,D)\colon H^{s+d}\to H^s$ for all real $s$
was proved by L.~H{\"o}rmander \cite{H88,H97}
for every symbol in $S^d_{1,1}$
fulfilling the twisted diagonal condition, mentioned already
in \eqref{tdc-cnd}; namely for some $B\ge1$,
\begin{equation}
  \hat a(\xi,\eta)=0\quad\text{when} \quad
  B(|\xi+\eta|+1)<|\eta|.
  \label{tdc'-cnd}
\end{equation}
It is of course natural to conjecture that \eqref{tdc'-cnd} 
also implies continuity $a(x,D)\colon \cal S'\to\cal S'$.
However, this question has neither been formulated nor treated 
before it was addressed in \cite{JJ10tmp}.

To prove the conjecture, one may argue by duality, which succeeds as
follows. First a lemma on the adjoint symbols is recalled
from \cite{H88} and \cite[Lem.~9.4.1]{H97}.

\begin{lem}
  \label{GB-lem}
When $a(x,\eta)$ is in $S^{d}_{1,1}(\Rn\times \Rn)$ and for some 
$B\ge 1$ satisfies the twisted diagonal condition \eqref{tdc'-cnd},
then the adjoint $a(x,D)^*=b(x,D)$ has the symbol 
\begin{equation}
  b(x,\eta)=e^{\im D_x\cdot  D_\eta}\overline{a(x,\eta)},  
\end{equation}
which is in $S^{d}_{1,1}(\Rn\times \Rn)$ in this case and
$\hat b(\xi,\eta)=0 $ when $|\xi+\eta|>B(|\eta|+1)$.
Moreover,
\begin{equation}
  |D^\alpha_\eta D^\beta_x b(x,\eta)|\le 
  C_{\alpha\beta}(a)B(1+B^{d-|\alpha|+|\beta|})(1+|\eta|)^{d-|\alpha|+|\beta|},
  \label{tdc-ineq}
\end{equation}
for certain continuous seminorms $C_{\alpha\beta}$ on 
$S^{d}_{1,1}(\Rn\times \Rn)$, that do not depend on $B$.
\end{lem}

The fact that $C_{\alpha,\beta}(a)$ is a continuous seminorm on $a(x,\eta)$
is added here (it is seen directly from H{\"o}rmander's proof). To emphasize
its importance, note that if $a_k\to a$ in the topology of $S^d_{1,1}$ and
they all fulfil \eqref{tdc'-cnd} with the same $B\ge 1$, then insertion of
$a_k-a$ in \eqref{tdc-ineq} yields that $b_k\to b$ in $S^d_{1,1}$.

In view of the lemma, it is clear that if $a(x,D)$ fulfils \eqref{tdc'-cnd}, 
it necessarily has continuous linear \emph{extension}
$\cal S'(\Rn)\to\cal S'(\Rn)$, namely $b(x,D)^*$. 
This moreover coincides with Definition~\ref{vfm-defn}
of $a(x,D)$ by vanishing frequency modulation:

\begin{prop}[{\cite[Prop.~4.2]{JJ08vfm}}]
  \label{GB-prop}
When $a(x,\eta)\in S^d_{1,1}(\Rn\times \Rn)$ fulfils 
\eqref{tdc'-cnd}, then $a(x,D)$ is a continuous linear map $\cal S'(\Rn)\to\cal
S'(\Rn)$ and it equals the adjoint of the mapping $b(x,D)\colon \cal S(\Rn)\to\cal
S(\Rn)$, when $b(x,\eta)$ is the adjoint symbol as in Lemma~\ref{GB-lem}.
\end{prop}
\begin{proof}
A simple convolution estimate, cf \cite[Lem.~2.1]{JJ08vfm},
gives that in $S^{d+1}_{1,1}$,
\begin{equation}
  \psi(2^{-m}D_x)a(x,\eta)\psi(2^{-m}\eta)\to a(x,\eta)
  \quad\text{for}\quad m\to\infty .
\end{equation}
As frequency modulation cannot increase supports, this sequence
also fulfils \eqref{tdc'-cnd} for the same $B$. 
So since the passage to adjoint symbols by \eqref{tdc-ineq}
is continuous from the metric subspace of $S^{d}_{1,1}$ 
fulfilling \eqref{tdc'-cnd} to $S^{d+1}_{1,1}$, 
\begin{equation}
  b_m(x,\eta):= e^{\im D_x\cdot D_\eta}
  (\overline{\psi(2^{-m}D_x)a(x,\eta)\psi(2^{-m}\eta)})
  \xrightarrow[m\to\infty ]{~}
     e^{\im D_x\cdot D_\eta}\overline{a(x,\eta)}=: b(x,\eta).
  \label{*conv-eq}
\end{equation}
Moreover, since $b(x,D)$ as an operator on the Schwartz
space depends continuously on the symbol, 
one has for $u\in \cal S'(\Rn)$, $\varphi\in \cal S(\Rn)$,
\begin{equation}
  \begin{split}
    \scal{b(x,D)^*u}{\varphi}&=\scal{u}{b(x,D)\varphi}
\\
  &= \scal{u}{\lim_{m\to\infty }\OP(b_m(x,\eta))\varphi}
\\
  &= \lim_{m\to\infty }
     \Scal{\OP(\psi(2^{-m}D_x)a(x,\eta)\psi(2^{-m}\eta))u}{\varphi}.
  \end{split}
\end{equation}
As the left-hand side is independent of $\psi$ the limit in \eqref{apsi-eq} is so,
hence the definition of $a(x,D)$ gives that every 
$u\in \cal S'(\Rn)$ is in $D(a(x,D))$ and $a(x,D)u=b(x,D)^*u$ as claimed.
\end{proof}

The mere extendability to 
$\cal S'$ under the twisted diagonal condition \eqref{tdc'-cnd} could
have been observed already in \cite{H88,H97}. The above result seems to be
the first giving a sufficient condition for a type $1,1$-operator to
be \emph{defined} on the entire $\cal S'(\Rn)$.

Despite its success, the above is of limited value when it comes to
continuity results, say in the $H^s_p$-scale.
On the positive side it does show that the first part of Theorem~\ref{Hsp-thm} 
implies the last part there; but the former is not in
itself related to duality.

For a direct proof of continuity it is convenient 
to invoke the paradifferential decomposition in 
Proposition~\ref{a123-prop}, and the outcome of this is satisfying 
inasmuch as 
the series in \eqref{a1-eq}--\eqref{a3-eq} can be shown to converge for
all temperate $u$, when $a(x,D)$ fulfils \eqref{tdc'-cnd} as above. 
In fact, by summing up one arrives at a contraction of Theorems~6.3 and 6.5 in
\cite{JJ10tmp} as a main theorem of the analysis:

\begin{thm}[{}]   \label{tdc-thm}
When $a\in S^d_{1,1}(\Rn\times \Rn)$ fulfils  the twisted diagonal condition
\eqref{tdc'-cnd}, then the associated
type $1,1$-operator  
$a(x,D)$ defined by vanishing frequency modulation
is an everywhere defined continuous linear map 
\begin{equation}
 a(x,D)\colon \cal S'(\Rn)\to\cal S'(\Rn),  
\end{equation}
with its adjoint $a(x,D)^*$ also in $\OP(S^d_{1,1}(\Rn\times \Rn))$.
For each modulation function $\psi$, the operator fulfils 
\begin{equation}
  a(x,D)u=a^{(1)}_\psi(x,D)u+a^{(2)}_\psi(x,D)u+a^{(3)}_\psi(x,D)u,
  \label{a123'-eq}
\end{equation}
where the operators on the right-hand side all belong to $\OP(S^d_{1,1})$
and likewise fulfil \eqref{tdc'-cnd} (hence have adjoints in
$\OP(S^d_{1,1})$); they are given by 
the series in \eqref{a1-eq}, \eqref{a2-eq}, \eqref{a3-eq} 
that converge rapidly in $\cal S'(\Rn)$ for every $u\in \cal S'(\Rn)$.
\end{thm}

The first statement is of course a repetition of Proposition~\ref{GB-prop}.
The rest of the proof relies on 
Proposition~\ref{corona-prop} and the second part of 
Proposition~\ref{ball-prop}. Indeed, these
show that the first assumptions in the
following lemma is fulfilled, when applied in the basic case with 
$\theta_0=1=\theta_1$
to the series in Proposition~\ref{a123-prop}:

\begin{lem}[{\cite[Lem.~A.1]{JJ10tmp}}]
  \label{corona-lem}
$1^\circ$~Let $(u_j)_{j\in \N_0}$ be a sequence in $\cal S'(\Rn)$ fulfilling 
that there exist $A>1$ and $\theta_1>\theta_0>0$ such that 
$\supp \hat u_0\subset\{\,\xi\mid |\xi|\le A\,\}$
while for $j\ge1$
   \begin{equation}
   \supp \hat u_j\subset\{\,\xi\mid 
       \tfrac{1}{A}2^{j\theta_0}\le |\xi|\le A2^{j\theta_1} \,\}, 
  \label{DAC-cnd}
   \end{equation}
and that for suitable constants $C\ge0$, $N\ge0$,
\begin{equation}
  |u_j(x)|\le C 2^{jN\theta_1}(1+|x|)^N \text{ for all $j\ge0$}.
  \label{CM-cnd}  
\end{equation}
Then $\sum_{j=0}^\infty u_j$ converges rapidly 
in $\cal S'(\Rn)$ to a distribution $u$,
for which $\hat u$ is of order $N$.

$2^\circ$~For every $u\in \cal S'(\Rn)$ both \eqref{DAC-cnd} and
\eqref{CM-cnd} are fulfilled for $\theta_0=\theta_1=1$ by the functions  
$u_0=\Phi_0(D)u$ and $u_j=\Phi(2^{-j}D)u$
when $\Phi_0,\Phi\in C^\infty_0(\Rn)$ and $0\notin\supp\Phi$. 
In particular this is the case for a 
Littlewood--Paley decomposition 
$1=\Phi_0+\sum_{j=1}^\infty \Phi(2^{-j}\xi)$.
\end{lem}

The second assumption of this lemma, cf \eqref{CM-cnd}, is verified for the
terms in \eqref{a123'-eq} by using the pointwise estimates. 
Indeed, for the general term in 
$a^{(1)}_\psi(x,D)u$ it is seen from Proposition~\ref{cutoff-prop} for
$\Phi=\psi(2^{-h}\cdot )$ and $\Psi=\varphi$ that
\begin{equation}
  |a^{k-h}(x,D)u_k(x)|=|\OP(\psi(2^{k-h}D_x)a(x,\eta)\varphi(2^{-k}\eta))u(x)| 
  \le c 2^{k(N+d_+)}(1+|x|)^{N+d_+}.
\end{equation}
Similar estimates can be found for $a^{(2)}_\psi(x,D)$ and
$a^{(3)}_\psi(x,D)$, cf \cite[Prop.~6.1--2]{JJ10tmp}, hence the three
series converge in $\cal S'$ according to the above lemma.

Finally, to conclude that the three terms in \eqref{a123'-eq} are type
$1,1$-operators, one can apply Proposition~\ref{cutoff-prop} once more, this
time with other choices of the two cut-off functions. For example, in the
primed sum of remainders in \eqref{a1u-eq}, they should for each $l$ be taken
as  $\psi\cdot \Psi(2^{h-l}\cdot )$ and $\Phi(2^{-l}\cdot )\psi$, respectively;
then Lemma~\ref{corona-lem} yields that the remainders can be summed over
$m\in\N$, hence that (each term in) 
the primed sum tends to $0$ in $\cal S'$ as $m\to\infty$.
In this way Theorem~\ref{tdc-thm} is proved.

\begin{rem}
Theorem~\ref{tdc-thm} and its proof generalises a result of 
R.~R.~Coifman and Y.~Meyer \cite[Ch.~15]{MeCo97} in three ways.
They stated Lemma~\ref{corona-lem} for $\theta_0=\theta_1=1$ 
and derived a version of Theorem~\ref{tdc-thm} for 
paramultiplication, though only 
with a treatment of the first and third term.
\end{rem}

Without the twisted diagonal condition \eqref{tdc'-cnd}, the techniques
behind Theorem~\ref{tdc-thm} at least show that  
$a^{(1)}(x,D)u$ and $a^{(3)}(x,D)u$ are always defined. Thus one has

\begin{cor}
  \label{a123-cor}
For $a(x,\eta)$ in $S^d_{1,1}(\Rn\times \Rn)$ a distribution $u\in \cal
S'(\Rn)$  belongs to the domain $D(a(x,D))$ if and only if the series for
$a^{(2)}(x,D)u$ in \eqref{a2-eq} converges in $\cal D'(\Rn)$.
\end{cor}

As a follow-up to this corollary, it is natural to ask whether less strong
assumptions on $a(x,D)$ along $\cal T$   
improves the picture. That is the topic of the next section.

\section{The twisted diagonal condition of order $\sigma$}  
\label{wtdc-ssect}\noindent
When the strict vanishing of $\hat a(\xi,\eta)$ in a conical neighbourhood
of $\cal T$ is replaced by vanishing to infinite order (in some sense)
\emph{at} $\cal T$, then
one arrives at a characterisation of the self-adjoint subclass 
$\tilde S^d_{1,1}$. This result of L.~H{\"o}rmander is recalled here and
shown to imply continuity on $\cal S'$.

\subsection{Localisation along the twisted diagonal}
   \label{micro-sssect}
As a weakening of the twisted diagonal condition \eqref{tdc'-cnd},
L.~H{\"o}rmander \cite{H88,H89,H97} introduced 
certain localisations of the symbol to conical neighbourhoods 
of $\cal T$. 
Specifically this was achieved by passing from $a(x,\eta)$ to
another symbol, that was denoted by 
$a_{\chi,\varepsilon}(x,\eta)$. This is defined by
\begin{equation}
  \hat a_{\chi,\varepsilon}(\xi,\eta)
  =\hat a(\xi,\eta)\chi(\xi+\eta,\varepsilon\eta),
  \label{axe-eq}
\end{equation}
whereby $\chi\in C^\infty (\Rn\times \Rn)$ is chosen so that
\begin{gather}
  \chi(t\xi,t\eta)= \chi(\xi,\eta)\quad\text{for}\quad t\ge 1,\ |\eta|\ge 2
\label{chi1-eq}   \\
  \supp \chi\subset \{\,(\xi,\eta)\mid 1\le |\eta|,\ |\xi|\le
|\eta|\,\}
\label{chi2-eq} \\
  \chi=1 \quad\text{in}\quad 
  \{\,(\xi,\eta)\mid 2\le |\eta|,\ 2|\xi|\le |\eta|\,\}.
  \label{chi3-eq}
\end{gather}
Here it should first of all be noted that by the choice of $\chi$,
\begin{equation}
  \supp \hat a_{\chi,\varepsilon}\subset 
  \{\,(\xi,\eta)\mid 1+|\xi+\eta|\le 2\varepsilon|\eta| \,\}.
  \label{axesupp-eq}
\end{equation}
So when $a(x,\eta)$ fulfils the
strict condition \eqref{tdc'-cnd}, then 
clearly $a_{\chi,\varepsilon}\equiv 0$ for $\varepsilon>0$ so small that
$B\le 1/(2\varepsilon)$. Therefore milder conditions will result by imposing
smallness requirements on $a_{\chi,\varepsilon}$ in the limit
$\varepsilon\to0$. 

As a novelty in the analysis, L.~H{\"o}rmander linked the above to
well-known 
multiplier conditions (of Mihlin--H{\"o}rmander type) by introducing the
condition that for some $\sigma\in\R$, 
it holds for all multiindices $\alpha$ and $0<\varepsilon<1$ that
\begin{equation}
  \sup_{R>0,\; x\in \Rn}R^{-d}\big(
  \int_{R\le |\eta|\le 2R} |R^{|\alpha|}D^\alpha_{\eta}a_{\chi,\varepsilon}
  (x,\eta)|^2\,\frac{d\eta}{R^n}
  \big)^{1/2}
  \le c_{\alpha,\sigma} \varepsilon^{\sigma+n/2-|\alpha|}.
  \label{Hsigma-eq}
\end{equation}
This is referred to as the twisted diagonal condition of order $\sigma\in
\R$. 

The above asymptotics for $\varepsilon\to0$ always holds for $\sigma=0$,
regardless of $a\in S^d_{1,1}$,  
as was proved in \cite[Prop.~3.2]{H89}, and given as a part of
\cite[Lem.~9.3.2]{H97}: 

\begin{lem}   \label{Heps-lem}
When $a\in S^d_{1,1}(\Rn\times \Rn)$ and $0<\varepsilon\le 1$, then
$a_{\chi,\varepsilon}\in C^\infty $ and 
\begin{gather}
  |D^\alpha_{\eta}D^\beta_x a_{\chi,\varepsilon}(x,\eta)|\le 
   C_{\alpha,\beta}(a)\varepsilon^{-|\alpha|}
  (1+|\eta|)^{d-|\alpha|+|\beta|}
  \\
  \big(
  \int_{R\le |\eta|\le 2R} |D^\alpha_{\eta}a_{\chi,\varepsilon}
  (x,\eta)|^2\,d\eta
  \big)^{1/2}
  \le C_{\alpha} R^d (\varepsilon R)^{n/2-|\alpha|}.
\end{gather} 
The map $a\mapsto a_{\chi,\varepsilon}$ is continuous in $S^d_{1,1}$.
\end{lem}
The last remark on continuity was added in \cite{JJ10tmp} 
because it plays a role in connection with the definition by
vanishing frequency modulation.
It is easily verified by deducing from the proof of \cite[Lem.~9.3.2]{H97}
that the constant $C_{\alpha,\beta}(a)$ is
a continuous seminorm in $S^d_{1,1}$.

In case $a(x,\eta)$ fulfils \eqref{Hsigma-eq} for some $\sigma>0$, 
the localised symbol  $a_{\chi,\varepsilon}$ tends faster to $0$, and this
was proved in \cite{H89,H97} to imply boundedness 
\begin{equation}
  a(x,D)\colon H^{s+d}(\Rn)\to H^s(\Rn) \quad\text{for}\quad s>-\sigma.
  \label{Hssigma-eq}
\end{equation}
The reader could consult \cite[Thm.~9.3.5]{H97} for this (and
\cite[Thm.~9.3.7]{H97} for four pages of proof of the sharpness of the
condition $s>-\sigma$).
Consequently, when $\hat a(\xi,\eta)$ satisfies
\eqref{Hsigma-eq} for all $\sigma\in \R$, 
then there is boundedness $H^{s+d}\to H^s$ for
all $s\in \R$. 

As accounted for below, $a(x,D)\colon \cal S'\to\cal S'$ is furthermore
everywhere defined and continuous when $a(x,\eta)$ fulfils 
\eqref{Hsigma-eq} for every $\sigma\in \R$. However, this relies on
L.~H{\"o}rmander's characterisation of such symbols as those having adjoints of
type $1,1$.

\subsection{The self-adjoint subclass $\tilde S^d_{1,1}$}
  \label{saSd-sssect}
The next result characterises the $a\in S^d_{1,1}$ for which the adjoint
symbol $a^*$ is again in $S^d_{1,1}$; cf the below condition \eqref{a*-cnd}. 
Since passage to the adjoint is an involution, such symbols constitute the
self-adjoint subclass 
\begin{equation}
  \tilde S^d_{1,1}:= S^d_{1,1}\cap (S^d_{1,1})^* .
\end{equation}

\begin{thm}   \label{a*-thm}
For every symbol $a(x,\eta)$ in $S^d_{1,1}(\Rn\times \Rn)$ the following
properties are equivalent:
\begin{rmlist}
    \item   \label{a*-cnd}
    The adjoint symbol $a^*(x,\eta)$ also belongs to 
    $S^d_{1,1}(\Rn\times \Rn)$.
  \item   \label{orderN-cnd}
  For arbitrary $N>0$ and $\alpha$, $\beta$ there is a constant
$C_{\alpha,\beta,N}$ such that 
\begin{equation}
  |D^\alpha_\eta D^\beta_x a_{\chi,\varepsilon}(x,\eta)|\le C_{\alpha,\beta,N}
  \varepsilon^{N}(1+|\eta|)^{d-|\alpha|+|\beta|}
  \quad\text{for}\quad 0<\varepsilon<1.
\end{equation} 
  \item   \label{sigma-cnd}
  For all $\sigma\in\R$ there is a constant $c_{\alpha,\sigma}$ such that
\eqref{Hsigma-eq} holds for $0<\varepsilon<1$, ie
\begin{equation}
  \sup_{R>0,\ x\in\Rn} R^{|\alpha|-d}
  \big( \int_{R\le |\eta|\le 2R}
   |D^\alpha_{\eta} a_{\chi,\varepsilon}(x,\eta)|^2\,\frac{d\eta}{R^n}
  \big)^{1/2}
  \le c_{\alpha,\sigma} \varepsilon^{\sigma+\tfrac{n}{2}-|\alpha|}.
\end{equation}
\end{rmlist}
In the affirmative case $a\in \tilde S^{d}_{1,1}$, and there is an estimate
\begin{equation}
  |D^\alpha_\eta D^\beta_x a^*(x,\eta)| \le 
  (C_{\alpha,\beta}(a)+C'_{\alpha,\beta,N})(1+|\eta|)^{d-|\alpha|+|\beta|}
  \label{CC'-eq}
\end{equation}
for a certain continuous seminorm 
$C_{\alpha,\beta}$ on $S^d_{1,1}(\Rn\times \Rn)$ and a finite sum
$C'_{\alpha,\beta,N}$ of constants fulfilling 
the inequalities in \eqref{orderN-cnd}.
\end{thm}
It deserves to be mentioned that
\eqref{a*-cnd} holds for $a(x,\eta)$ if and only if it holds
for $a^*(x,\eta)$ (neither \eqref{orderN-cnd} nor
\eqref{sigma-cnd} make this obvious).
But \eqref{orderN-cnd} immediately gives that the inclusion 
$\tilde S^d_{1,1}\subset \tilde S^{d'}_{1,1}$ holds for $d'>d$, as one would
expect. Condition \eqref{sigma-cnd} involves $L_2$-norms analogous to those
in the Mihlin--H{\"o}rmander multiplier theorem
and is useful for estimates.

As a small remark it is noted that \eqref{orderN-cnd}, \eqref{sigma-cnd} 
both hold either for all $\chi$ satisfying \eqref{Hsigma-eq} or for none, 
since \eqref{a*-cnd} does not depend on $\chi$.
It moreover suffices to verify \eqref{orderN-cnd} or \eqref{sigma-cnd} for
$0<\varepsilon<\varepsilon_0$ for some convenient 
$\varepsilon_0\in \,]0,1[\,$, as can be seen from the second inequality in  
Lemma~\ref{Heps-lem}, since every power $\varepsilon^p$ is bounded on the
interval $[\varepsilon_0,1]$. 

The theorem is essentially due to
L.~H{\"o}rmander, who stated the equivalence of \eqref{a*-cnd} 
and \eqref{orderN-cnd} explicitly in \cite[Thm.~4.2]{H88} and
\cite[Thm.~9.4.2]{H97}, in the latter 
with brief remarks on \eqref{sigma-cnd}. 
Equivalence with continuous extensions
$H^{s+d}\to H^s$ for all $s\in \R$ was also shown, whereas the estimate
\eqref{CC'-eq} was not mentioned.

However, a full proof of Theorem~\ref{a*-thm}
was given in \cite[Sec.~4.2]{JJ10tmp}, not only because
a heavy burden of verification was left to the reader in \cite{H97},
but more importantly because some consequences for operators defined by
vanishing frequency modulation were derived as corollaries to the proof.

It would lead too far here to give the details (that follow the line of
thought in \cite[Thm.~9.4.2]{H97}, and are available in
\cite[Sec.~4.2]{JJ10tmp});
it should suffice to mention that the final estimate \eqref{CC'-eq}
in the theorem
was derived from the proof that \eqref{orderN-cnd} implies \eqref{a*-cnd}.

In its turn, this estimate was shown to imply
that the self-adjoint class $\tilde S^d_{1,1}$, as envisaged, 
behaves nicely under full frequency modulation:

\begin{cor}[{\cite[Cor.~4.5]{JJ10tmp}}]   \label{a*-cor}
Whenever $a(x,\eta)$ belongs to $\tilde S^d_{1,1}(\Rn\times \Rn)$ and $\psi$
is a modulation function, then it holds for the adjoint symbols that
\begin{equation}
  \big(\psi(2^{-m}D_x)a(x,\eta)\psi(2^{-m}\eta) \big)^*
  \xrightarrow[m\to\infty ]{~}
  a(x,\eta)^*
\end{equation}
in the topology of $S^{d+1}_{1,1}(\Rn\times \Rn)$.
\end{cor}

Returning to the proof of Proposition~\ref{GB-prop}, it is clear that the
above establishes \eqref{*conv-eq} under the weaker assumption that $a\in
\tilde S^d_{1,1}$. So by repeating the rest of the proof there, one obtains
the first main result on $\tilde S^d_{1,1}$:

\begin{thm}[{\cite[Thm.~4.6]{JJ10tmp}}]   \label{tildeS-thm}
For every $a(x,\eta)$ of type $1,1$ belonging to the self-adjoint
subclass $\tilde S^d_{1,1}(\Rn\times \Rn)$, that is characterised in
Theorem~\ref{a*-thm}, the operator 
\begin{equation}
  a(x,D)\colon \cal S'(\Rn)\to \cal S'(\Rn)
\end{equation}
is everywhere defined and continuous, and it equals the adjoint of 
$\OP(e^{\im D_x\cdot D_\eta}\bar a(x,\eta))$.
\end{thm}

\subsection{Paradifferential decompositions for the self-adjoint subclass}
  \label{saSd123-sssect}

The result below gives a generalisation of Theorem~\ref{tdc-thm} to the
operators $a(x,D)$ that merely fulfil the twisted diagonal condition
\eqref{Hsigma-eq} for every real $\sigma$ instead of the strict condition
\eqref{tdc'-cnd}. 

\begin{thm}[{\cite[Thm.~6.7]{JJ10tmp}}]   \label{sigma-thm}
When $a(x,\eta)$ belongs to 
$\tilde S^d_{1,1}(\Rn\times \Rn)$, cf Theorem~\ref{a*-thm},
then $a(x,D)$
is an everywhere defined continuous linear map 
\begin{equation}
 a(x,D)\colon \cal S'(\Rn)\to\cal S'(\Rn),  
\end{equation}
with its adjoint $a(x,D)^*$ also in $\OP(\tilde S^d_{1,1}(\Rn\times \Rn))$.
For each modulation function $\psi$, the operator fulfils 
\begin{equation}
  a(x,D)u=a^{(1)}_\psi(x,D)u+a^{(2)}_\psi(x,D)u+a^{(3)}_\psi(x,D)u,
\end{equation}
where the operators on the right-hand side all belong to 
$\OP(\tilde S^d_{1,1})$; 
they are given by 
the series in \eqref{a1-eq}, \eqref{a2-eq}, \eqref{a3-eq} 
that converge rapidly in $\cal S'(\Rn)$ for every $u\in \cal S'(\Rn)$.
\end{thm}

The proof of this theorem is quite lengthy, and therefore only sketched
here. First of all the continuity on $\cal S'$ is obtained from
Theorem~\ref{tildeS-thm}, of course. The other statements involve a
treatment of 9 infinite series, 
departing from the splitting obtained from \eqref{axe-eq} with
$\varepsilon=1$, 
\begin{equation}
  a(x,\eta)=(a(x,\eta)-a_{\chi,1}(x,\eta))+ a_{\chi,1}(x,\eta).
\end{equation}
Here the difference $a-a_{\chi,1}$ 
fulfils the twisted diagonal condition \eqref{tdc'-cnd}
with $B=1$, so this term is covered by Theorem~\ref{tdc-thm} 
(that involves three series). 

The remainder $a_{\chi,1}(x,\eta)$ is treated through the
paradifferential decomposition \eqref{a123-eq}. However, to simplify
notation it is useful to note that $\hat a_{\chi,1}$ is supported by the set
\begin{equation}
  \bigl\{\,(\xi,\eta) \bigm| \max(1,|\xi+\eta|)\le |\eta|\,\bigr\}.
\end{equation}
It therefore suffices to prove the theorem for such symbols $a(x,\eta)$.
Here $a^{(1)}_\psi(x,D)$ and $a^{(3)}_\psi(x,D)$ are again covered by
Theorem~\ref{tdc-thm}; but the term $a^{(2)}_\psi(x,D)$ is first split into
two using \eqref{a2'-eq}, where each term is subjected to
H{\"o}rmander's localisation to $\cal T$ (now applied for the second time).

More specifically, \eqref{a2'-eq} gives rise to the series
\begin{equation}
  \sum_{k=0}^\infty (a^k-a^{k-h})(x,D)u_k,
\qquad
\sum_{k=1}^\infty a_k(x,D)(u^{k-1}-u^{k-h});
  \label{2a2-eq}
\end{equation}
they can be treated in much the same way, so only the latter will be
discussed here. Set 
\begin{equation}
  v_k=u^{k-1}-u^{k-h}
  =\cal F^{-1}((\Phi(2^{1-k}\cdot )-\Phi(2^{h-k}\cdot))\hat u).
\end{equation}
Using the cut-off function $\chi$ from \eqref{chi1-eq}--\eqref{chi3-eq} to set
\begin{equation}
  \hat a_{k,\chi,\varepsilon}(\xi,\eta)=
  \hat a(\xi,\eta)\Phi(2^{-k}\xi)\chi(\xi+\eta,\varepsilon\eta),
  \label{akke-id}
\end{equation}
the last of the above series is split into two sums by writing
\begin{equation}
  a_k(x,D)v_k=a_{k,\chi,\varepsilon}(x,D)v_k+b_k(x,D)v_k .
\end{equation}
Of course the remainder $b_k(x,D)v_k$ also depends on $\varepsilon$, which
with $\varepsilon=2^{-k\theta}$, say 
for $\theta=1/2$ is taken to vary with the summation index (as also done in
some proofs by L.~H{\"o}rmander). 

Convergence of $\sum b_k(x,D)v_k$ is obtained from Lemma~\ref{corona-lem},
that 
this time applies with $\theta_0=1-\theta=1/2$ and $\theta_1=1$, as can be
seen with a minor change of the argument for Proposition~\ref{corona-prop}.
The polynomial growth of its terms follows easily from similar
control of $a_k$ and of $a_{k,\chi,\varepsilon}$.

To control the $a_{k,\chi,\varepsilon}(x,D)v_k$ it is advantageous to adopt
the factorisation inequality,
\begin{equation}
  a_{k,\chi,\varepsilon}(x,D)v_k\le F_{a_{k,\chi,\varepsilon}}(N,R2^k;x)
  v_k^*(N,R2^k;x),
\end{equation}
where the symbol factor is controlled by certain $L_2$-norms according to 
Theorem~\ref{Fa-thm}.

Indeed, a direct comparison of the $L_2$-conditions 
in Theorem~\ref{Fa-thm} with those in Theorem~\ref{a*-thm} reveal that they
only differ in the domain of integration. But the
integration area in Theorem~\ref{Fa-thm} can be covered by a fixed finite
number of the annuli appearing in the $L_2$-conditions for 
$a\in \tilde S^d_{1,1}$\,---\,and these furthermore yield 
control by powers of $\varepsilon$; cf Theorem~\ref{a*-thm}.
It is therefore not surprising that the above gives an estimate of the form
\begin{equation}
  |a_{k,\chi,\varepsilon}(x,D)v_k(x)|\le  cv^*_k(N,R2^k;x)
  \big(\sum_{|\alpha|\le N+[n/2]+1}c_{\alpha,\sigma}
  \varepsilon^{\sigma+n/2-|\alpha|}\big)
  (R2^k)^{d}.
  \label{akkev-eq}
\end{equation}
Here $N=\order_{\cal S'} (\hat u)$, cf \eqref{Sorder-eq}, so for $\theta=1/2$ the polynomial growth of $v_k^*$,
cf \eqref{u*xN-eq}, entails
\begin{equation}
  |a_{k,\chi,2^{-k\theta}}(x,D)v_k(x)| 
  \le c(1+|x|)^N 2^{-k(\sigma-1-2d-3N)/2}.
  \label{akkev'-eq}
\end{equation}
Now since $a\in \tilde S^d_{1,1}$ one can take $\sigma$ such that
$\sigma>3N+2d+1$, whence
$\sum_k\dual{a_{k,\chi,\varepsilon}(x,D)v_k}{\varphi}$ 
converges rapidly for $\varphi\in \cal S$.
Thereby all series in the theorem converge as claimed.

Finally, it only remains to note that the remainder terms appearing in
connection with the operator $a^{(2)}_\psi(x,D)$ tend to $0$ for
$m\to\infty $, also under the weak assumption that $a\in \tilde
S^d_{1,1}$. This 
can be seen with an argument analogous to the one for Theorem~\ref{tdc-thm}, 
by using the full generality of Proposition~\ref{cutoff-prop}.

\bigskip

The paradifferential decomposition, that was analysed above, is applied in
the $L_p$-estimates in the following sections.

\section{Domains of type $\mathbf{1},\mathbf{1}$-operators} 
\label{dom-ssect}\noindent
For the possible domains of type $1,1$-operators, the scale $F^{s}_{p,q}(\Rn)$
of Lizorkin--Triebel
spaces  was recently shown to play a role, for
it was proved in \cite{JJ04DCR,JJ05DTL} that for all $p\in
[1,\infty[\,$, \emph{every} $a\in S^d_{1,1}$ gives a bounded linear map 
\begin{equation}
  F^d_{p,1}(\Rn)\xrightarrow[]{a(x,D)} L_p(\Rn).
  \label{Fdp1-eq}
\end{equation}
The reader may refer to Section~\ref{NN-ssect} for a review of the definition
and basic properties of the $F^{s}_{p,q}$ spaces and of the related Besov
spaces $B^{s}_{p,q}$.

The result in \eqref{Fdp1-eq} is a substitute of boundedness 
$H^d_p\to L_p$, or of $L_p$-boundedness for $d=0$
(neither of which can hold in general because of Ching's counter-example
\cite{Chi72}, recalled in Lemma~\ref{cex-lem}). Indeed, 
$H^s_p=F^s_{p,2}$ for $1<p<\infty$ whereas $F^{s}_{p,1}\subsetneq F^s_{p,q}$
for $q>1$, so \eqref{Fdp1-eq} means that $a(x,D)$ is bounded from a
sufficiently small subspace of $H^d_p$.

Moreover, inside the $F^{s}_{p,q}$ and $B^{s}_{p,q}$ scales,
\eqref{Fdp1-eq} gives \emph{maximal} domains for
$a(x,D)$ in $L_p$, for it was noted in \cite[Lem.~2.3]{JJ05DTL} 
that Ching's 
operator is discontinuous $F^d_{p,q}\to\cal D'$ 
and $B^{d}_{p,q}\to\cal D'$ as soon as $q>1$. This follows as in  
Lemma~\ref{cex-lem} above by simply calculating the norms of $v_N$ in these
spaces.

In comparison G.~Bourdaud \cite{Bou83,Bou88}
showed the borderline result that every $a(x,D)$ in $\OP(S^0_{1,1})$ 
has a bounded extension 
\begin{equation}
  A\colon B^0_{p,1}\to L_p \quad\text{for}\quad1\le p\le \infty .
\end{equation}
In view of the embedding $B^s_{p,1}\hookrightarrow F^{s}_{p,1}$ valid for
all finite $p\ge 1$, the result in \eqref{Fdp1-eq} is sharper already for
$d=0$ (unless $p=\infty $)\,---\,and for $1\le p<\infty $ 
the best possible within the scales $B^{s}_{p,q}$, $F^{s}_{p,q}$ 
for the full class $\OP(S^d_{1,1})$; cf the above.
Thus the $F^{s}_{p,q}$-spaces with $q=1$ are indispensable for a sharp
description of the borderline $s=0$ for type $1,1$-operators; cf 
\cite[Rem.~1.1]{JJ05DTL}.

The above considerations can be summed up thus:

\begin{thm}[{\cite{JJ05DTL}}]
  \label{Fdp1-thm}
Every $a\in S^{d}_{1,1}(\Rn\times\Rn)$, $d\in\R$, 
yields a bounded operator 
\begin{align}
  a(x,D)&\colon F^{d}_{p,1}(\Rn)\to L_p(\Rn)
         \quad\text{for}\quad p\in[1,\infty[,
  \label{Fp1-eq}\\
  a(x,D)&\colon B^{d}_{\infty,1}(\Rn)\to L_\infty(\Rn).
  \label{Bi-eq}
\end{align}
The class $\op{OP}(S^d_{1,1})$ contains operators 
$a(x,D)\colon \cal S(\Rn)\to\cal
D'(\Rn)$, that are discontinuous when $\cal S(\Rn)$ is given the induced
topology from any of the Triebel--Lizorkin spaces $F^d_{p,q}(\Rn)$
or Besov spaces $B^d_{p,q}(\Rn)$ with $p\in [1,\infty]$ and
$q>1$ (while $\cal D'$ has the usual topology).
\end{thm}

By the remarks prior to the statements, it remains to discuss the
boundedness. Unlike \cite{JJ05DTL} that was based on Marschall's inequality
mentioned in Remark~\ref{Marschall-rem},
the present exposition of the proof will be indicated with as much use of
the pointwise estimates in Section~\ref{pe-ssect} as possible. 

However, it is still necessary to invoke the Hardy--Littlewood maximal
function, which is given on locally integrable functions $f(x)$ by
\begin{equation}
  Mf(x)=\sup_{R>0}\frac{1}{\op{meas}(B(x,R))} \int_{B(x,R)} |f(y)| \,dy.
\end{equation}
This is needed for the following well-known inequality:

\begin{lem}   \label{u*M-lem}
When $0<p<\infty $, $0<q\le \infty $ and $N>n/\min(p,q)$ are
given, there is a constant $c$ such that
\begin{equation}
  \int_{\Rn} \nrm{u_k^*(N,R2^k;x )}{\ell_{q}}^p\,dx
  \le 
  c \int_{\Rn} \nrm{u_k(x)}{\ell_{q}}^p\,dx
  \label{u*M-eq}
\end{equation}
for every sequence $(u_k)_{k\in \N_0}$ in $L_p(\Rn)$ satisfying
$\supp\hat u_k\subset \Bbar(0,R2^k)$, $k\ge 0$.
\end{lem}
This inequality and its proof are due to J.~Peetre. It might be illuminating
to give an outline: taking $r\in \,]0,\min(p,q)[\,$ so that
$N\ge n/r$ it can be shown 
(eg as in the paper of M.~Yamazaki \cite[Thm.~2.10]{Y1}) that
\begin{equation}
  u_k^*(N,R2^k;x)\le   u_k^*(\tfrac{n}{r},R2^k;x)\le c (Mu_k^r(x))^{1/r};
\end{equation}
so it suffices to estimate $\int \nrm{(Mu_k^r(x))^{1/r}}{\ell_q}^p\,dx$, 
but since $r<\min(p,q)$ the fundamental inequality of C.~Fefferman and
E.~M.~Stein \cite{FeSt72} (cf also
\cite[Thm.~2.2]{Y1}) yields an estimate by the right-hand side of
\eqref{u*M-eq}. 

\subsection{Proof of Theorem~\ref{Fdp1-thm}}
To obtain the boundedness one may simply take a finite sum of 
the inequalities in Proposition~\ref{a123pe-prop} and integrate. Indeed, for 
\eqref{a1-pe} this gives according to Lemma~\ref{u*M-lem} that
\begin{equation}
  \int |\sum_k a^{k-h}(x,D)u_k(x)|^p\,dx
 \le c \int(\sum_k (R2^k)^{d} |u^*_k(N,R2^k;x)|)^p\,dx
  \le c'\int (\sum_k 2^{kd} |u_k(x)|)^p\,dx.
\end{equation}
In particular this holds if the modulation function used in the splitting of
Proposition~\ref{a123-prop} coincides with the function generating the
Littlewood--Paley decomposition in the definition of $F^{s}_{p,q}$, which
can be arranged by Remark~\ref{BF-rem}. 
Then the integral on the right-hand side above is less than
$\nrm{u}{F^{d}_{p,1}}^p$, so that the terms $2^{kd}| u_k(x)|$ tend to $0$
a.e.\ on $\Rn$ for $k\to\infty$.

So by taking $k$ in a set of the form $\{K+1,\dots ,K+K'\}$, it follows by
majorised convergence for $K\to\infty $
that $\sum a^{k-h}(x,D)u_k$ is a Cauchy series, hence converges
in $L_p(\Rn)$. The above inequalities therefore hold verbatim for the sum
over all $k\ge h$, that is, $\nrm{\sum a^{k-h}(x,D)u_k}{L_p}
\le c\nrm{u}{F^d_{p,1}}$.
This means that $a^{(1)}(x,D)$ is bounded $F^{d}_{p,1}\to L_p$.

The boundedness of $a^{(2)}(x,D)$ is analogous, for in view of
\eqref{a2two-eq} the above procedure need only be applied to each of the
inequalities \eqref{a2-pe} and \eqref{a2'-pe}. For the latter, the fixed sum
over $l=1,\dots ,h-1$ is easily treated via the triangle inequality.

In case of $a^{(3)}(x,D)$ the approach needs a minor modification since the
sum in 
\eqref{a3-pe} has its number of terms increasing with $j$. But this is
easily handled by taking $M>0$ there and using a
well-known summation lemma (that goes back at least to \cite{Y1}, 
where it was used for similar purposes): 
for $s<0$, $0<q\le \infty $ and $b_j\in \C$,
\begin{equation}
  \sum_{j=0}^\infty 2^{sjq}
  (\sum_{k=0}^j |b_k|)^q\le c \sum_{j=0}^\infty 2^{sjq}|b_j|^q.  
  \label{Ylem-eq}
\end{equation}
Indeed, from this and
Lemma~\ref{u*M-lem} it follows that
\begin{equation}
  \begin{split}
 \int|\sum_j a_j(x,D)u^{j-h}(x)|^p\,dx
&\le c\int(\sum_j 2^{-jM} (\sum_{k=0}^{j} 
      (R2^k)^{d+M} u_k^*(N,R2^k;x))^p\,dx
\\
&\le  c'\int (\sum_j 2^{jd} |u_j(x)|)^p\,dx.
  \end{split}
\end{equation}
Again this yields the convergence in $L_p$ of the series, hence 
boundedness of $a^{(3)}(x,D)$. By Proposition~\ref{a123-prop} the operator
$a_\psi(x,D)$ is therefore continuous $F^d_{p,1}\to L_p$, but since it
coincides with $a(x,D)$ on the dense subset $\cal S\subset F^d_{p,1}$, it
does not depend on $\psi$; therefore $a(x,D)\colon F^d_{p,1}\to L_p$ is
everywhere defined and bounded.

The Besov case with $p=\infty $ is analogous, although the boundedness of
$Mf$ on $L_\infty $ suffices (instead of Lemma~\ref{u*M-lem}) due to the
fact that the norms of $\ell_q$ and $L_p$ are interchanged for
$B^{s}_{p,q}$. Thus $a_{\psi}(x,D)\colon B^d_{\infty ,1}\to L_\infty $ is
bounded for each modulation function $\psi$. It is well known that 
$B^d_{p,1}$ has $\cal F^{-1}\cal E'$ as a dense subset, 
and since $a(x,D)$ by its extension to $\cal F^{-1}\cal E'$ coincides with
each $a_\psi(x,D)$ there, again there is no dependence on $\psi$; whence
$a(x,D)\colon B^d_{\infty ,1}\to L_\infty $ is everywhere defined and bounded.
This yields the proof.

\begin{rem}  \label{LpFE-rem}
  It should be mentioned that $L_p$-boundedness on $L_p\bigcap \cal F^{-1}\cal
  E'$ is much easier to establish, as indicated in \eqref{LpFE-eq}. In view of
  Ching's counter-example, this result is somewhat striking, hence was
  formulated as Theorem~6.1 in \cite{JJ10pe}.
\end{rem}

\section{General Continuity Results}
\label{Fspq-ssect}\noindent
In addition to the borderline case $s=0$, that was treated above, it is
natural to expect that the $F^{s}_{p,q}$ scale is invariant under $a(x,D)$
as soon as $s>0$. More precisely, the programme is to show that if
$a(x,\eta)$ is of order, or rather degree $d\in\R$, then $a(x,D)$ is
bounded from $F^{s+d}_{p,q}$ to $F^{s}_{p,q}$.

This is true to a wide extent, and in the verification one may by and large
use the procedure from Section~\ref{dom-ssect}. However, for $q\ne2$ 
one cannot just appeal to the completeness of $L_p$, but the spectral
properties in Proposition~\ref{corona-prop} and \ref{ball-prop} 
make it possible to apply the following lemma. 

\begin{lem}
  \label{Fspq-lem}
Let $s>\max(0,\fracc np-n)$ for $0<p<\infty$ and $0< q\le \infty$ and suppose
$u_j\in \cal S'(\Rn)$ such that, for some $A>0$,
\begin{equation}
  \supp\cal F u_j\subset B(0,A2^j),\qquad
  F(q):=
  \Nrm{(\sum_{j=0}^\infty 2^{sjq}|u_j(\cdot)|^q)^{\fracci1q}}{p}<\infty.
  \label{ballFq-eq}
\end{equation}
Then $\sum_{j=0}^\infty u_j$ converges in $\cal S'(\Rn)$ to some 
$u$ in the space $F^s_{p,r}(\Rn)$ for $ r\ge q$, $ r>\tfrac{n}{n+s}$,
fulfilling $\nrm{u}{F^s_{p,r}}\le cF(r)$ for some $c>0$ depending on
$n$, $s$, $p$ and $r$.

When moreover $\supp\cal Fu_j\subset \{\,\xi\mid A^{-1}2^j\le |\xi|\le
A2^j\,\}$ for $j\ge J$ for some $J\ge 1$, 
then the conclusions are valid for all $s\in \R$ and $r=q$.
\end{lem}

This has been known for $r=q$ at
least since the pseudo-differential $L_p$-estimates of M.~Yamazaki
\cite{Y1,Y2}, who proved it under the stronger assumption that
\begin{equation}
  s>\max(0,\fracnp-n,\fracc nq-n).
  \label{spq-eq}
\end{equation}
However, the above sharpening was derived as an addendum in 
\cite{JJ05DTL} by
noting that $F(r)\le F(q)<\infty $ for $r\ge q$. (The case $J>1$ is also an
addendum, which was used tacitly in \cite{Y1,Y2}.)

The following theorem is taken from \cite{JJ05DTL}, where the methods were
essentially the same as here, except that there was no explicit reference to
the definition by vanishing frequency modulation (this definition was first
crystallised in \cite{JJ08vfm}, inspired by the proof of the theorem in
\cite{JJ05DTL}). However, the proof was only sketched in \cite{JJ05DTL}, so
full explanations were given in \cite[Thm.~7.4]{JJ10tmp}.
 
Previous works on such $L_p$-results include those of G.~Bourdaud
\cite{Bou83,Bou88}, T.~Runst \cite{Run85} and R.~Torres \cite{Tor90}.
Besides the remarks in Section~\ref{review-ssect} it should be noted here,
that they worked under the assumption \eqref{spq-eq}, 
which has the disadvantage of excluding arbitrarily large values of $s$ 
(even for $p=2$) in the limit $q\to 0^+$. As indicated this is unnecessary:

\begin{thm}[{\cite[Cor.~6.2]{JJ05DTL},\cite[Thm.~7.4]{JJ10tmp}}]
  \label{FBspq-thm}
If $a\in S^d_{1,1}(\Rn\times\Rn)$ the corresponding operator $a(x,D)$ 
is a bounded map 
for $s>\max(0,\fracc np-n)$, $0<p,q\le\infty$,
\begin{align}
   a(x,D)&\colon F^{s+d}_{p,q}(\Rn)\to F^s_{p,r}(\Rn)
   \quad (p<\infty);
  \label{Fspr-eq} \\
   a(x,D)&\colon B^{s+d}_{p,q}(\Rn)\to B^s_{p,q}(\Rn).
  \label{Bspq-eq}
\end{align}
Hereby $r\ge q$ and $r>n/(n+s)$.
If \eqref{tdc'-cnd} holds, then \eqref{Fspr-eq} and \eqref{Bspq-eq} do so for
all $s\in \R$ and $r=q$.
\end{thm}

\begin{rem}
It should be noted that in the Banach space case (ie when $p\ge 1$ and $q\ge
1$), one can always take $r=q$ in \eqref{Fspr-eq}, since $q\ge 1>n/(n+s)$
then. 
\end{rem}

\begin{proof}
If $\psi$ is an arbitrary modulation function, then
Remark~\ref{BF-rem} shows that $\nrm{u}{F^{s}_{p,q}}$ can 
be calculated in terms of the Littlewood--Paley partition associated with
$\psi$; cf Section~\ref{LPa-ssect}. 

For the treatment of 
$a^{(1)}(x,D)u=\sum_{k=h}^\infty a^{k-h}(x,D)u_k$ with $u\in F^{s}_{p,q}$
note that an application of
the norms of $\ell_q$ and $L_p$ on both sides of the pointwise estimate
\eqref{a1-pe} gives, if $q<\infty $ for
simplicity's sake,
\begin{equation}
 \int_{\Rn}(\sum_{k=0}^\infty 2^{skq}
|a^{k-h}(x,D)u_k(x)|^q)^{\frac{p}{q}}\,dx 
\le 
  c_2\Nrm{(\sum_{k=0}^\infty 2^{(s+d)kq}
   u_k^*(N,R2^k;x)^q)^\frac1q}{p}^p.
  \label{a1Lplq-eq}
\end{equation}
If $N>n/\min(p,q)$ here, \eqref{u*M-eq} gives an estimate from above by
$\nrm{u}{F^{s+d}_{p,q}}$, so that one has the bound in 
Lemma~\ref{Fspq-lem}  for all $s\in \R$, whilst the corona condition holds
by Proposition~\ref{corona-prop}. Thus the lemma gives
\begin{equation}
  \nrm{a^{(1)}(x,D)u}{F^{s}_{p,q}}\le c 
  (\int_{\Rn}(\sum_{k=0}^\infty 2^{skq}
|a^{k-h}(x,D)u_k(x)|^q)^{\frac{p}{q}}\,dx)^{\fracpi}
   \le c'\nrm{u}{F^{s+d}_{p,q}}.
  \label{a1Fspq-eq}
\end{equation}

In the contribution $a^{(3)}(x,D)u=\sum_{j=h}^\infty a_j(x,D)u^{j-h}$ one
may apply \eqref{Ylem-eq} to the estimate \eqref{a3-pe} 
for $M>s$ to get that
\begin{equation}
  \begin{split}
 \sum_{j=0}^\infty 2^{sjq}   |a_j(x,D)u^{j-h}(x)|^q
&\le \sum_{j=0}^\infty 2^{(s-M)jq} (\sum_{k=0}^{j} 
      c_M(R2^k)^{d+M} u_k^*(N,R2^k;x))^q
\\
&\le  c \sum_{j=0}^\infty 2^{(s+d)jq} u_j^*(N,R2^j;x)^q,
  \end{split}
\end{equation}
which by integration entails
\begin{equation}
 (\int_{\Rn}(\sum_{j=0}^\infty 2^{sjq}
|a_j(x,D)u^{j-h}(x)|^q)^{\frac{p}{q}}\,dx )^{\fracpi}
\le 
  c_3\Nrm{(\sum_{j=0}^\infty 2^{(s+d)jq}u_j^*(x)^q)^\frac1q}{p}.
  \label{a3Lplq-eq}
\end{equation}
Repeating the argument for \eqref{a1Fspq-eq} 
one arrives at 
$\nrm{a^{(3)}(x,D)u}{F^{s}_{p,q}}\le c\nrm{u}{F^{s+d}_{p,q}}$.

In estimates of $a^{(2)}(x,D)u$ the various terms 
can be treated similarly, now departing from
\eqref{a2-pe} and \eqref{a2'-pe}. This gives
\begin{multline}
  \big( \int_{\Rn}(\sum_{k=0}^\infty 2^{skq}
|(a^k-a^{k-h})(x,D)u_k(x)+a_k(x,D)(u^{k-1}-u^{k-h})|^q
  )^{\frac{p}{q}}\,dx \big)^{\fracpi}
  \\
\le 
  c'_2\Nrm{(\sum_{k=0}^\infty 2^{(s+d)kq}u_k^*(x)^q)^\frac1q}{p}.
  \label{a2Lplq-eq}
\end{multline}
If moreover the twisted diagonal condition \eqref{tdc'-cnd} holds, 
the last part of Proposition~\ref{ball-prop} 
and \eqref{u*M-eq} show that the argument for \eqref{a1Fspq-eq},
mutatis mutandis, gives 
$\nrm{a^{(2)}(x,D)u}{F^{s}_{p,q}}\le c\nrm{u}{F^{s+d}_{p,q}}$.
Altogether one has  then, for all $s\in \R$,
\begin{equation}
  \nrm{a_{\psi}(x,D)u}{F^{s}_{p,q}}
\le  \sum_{j=1,2,3} \nrm{a^{(j)}(x,D)u}{F^{s}_{p,q}}
\le c p(a)\nrm{u}{F^{s+d}_{p,q}}.
  \label{apsi123-eq}
\end{equation}
Without \eqref{tdc'-cnd}, the spectra are by Proposition~\ref{ball-prop} only
contained in balls, but the condition $s>\max(0,\fracnp-n)$ and those on 
$r$ imply that 
$\nrm{a^{(2)}(x,D)u}{F^{s}_{p,r}}\le c \nrm{u}{F^{s+d}_{p,q}}$;
cf Lemma~\ref{Fspq-lem}.
This gives the above inequality with $q$ replaced by $r$ on 
the left-hand side.

Thus $a_{\psi}(x,D)\colon F^{s+d}_{p,q}\to F^{s}_{p,r}$ is continuous, but
since $\cal S$ is dense in $F^{s+d}_{p,q}$ for $q<\infty $ (and
$F^{s+d}_{p,\infty }\hookrightarrow F^{s'}_{p,1}$ for $s'<s+d$), there is no
dependence on $\psi$. Hence $u\in D(a(x,D))$ and the above 
inequalities hold for $a(x,D)u$, which proves \eqref{Fspr-eq} in all cases.

The Besov case is analogous; one can interchange the order of $L_p$ and
$\ell_q$ and refer to the maximal inequality for scalar functions:
Lemma~\ref{Fspq-lem} carries over to $B^{s}_{p,q}$ in a natural way for
$0<p\le \infty $ with $r=q$ in all cases; this is well known, cf
\cite{Y1,JJ05DTL,JoSi08}. 
\end{proof}

In the above proof, the Besov result \eqref{Bspq-eq} can also be
obtained by real interpolation of \eqref{Fspr-eq}, since the sum exponent is
inherited from the interpolation method; cf \cite[2.4.2]{T2}. However, this
will not cover $p=\infty $ in \eqref{Bspq-eq}, as this is excluded
in \eqref{Fspr-eq}.

By duality, Theorem~\ref{FBspq-thm} implies an extension to
operators that fulfil the twisted diagonal condition of arbitrary
real order. This is a main result.

\begin{thm}[{\cite[Thm.~7.5]{JJ10tmp}}]
  \label{FB8-thm}
Let $a(x,\eta)$ belong to the self-adjoint subclass
$\tilde S^d_{1,1}$, as characterised in 
Theorem~\ref{a*-thm}. Then $a(x,D)$ 
is a bounded map for all $s\in\R$,
\begin{align}
   a(x,D)&\colon F^{s+d}_{p,q}(\Rn)\to F^{s}_{p,q}(\Rn),
   \quad 1<p<\infty,\ 1<q\le \infty, 
  \label{F8'-eq} \\
   a(x,D)&\colon B^{s+d}_{p,q}(\Rn)\to B^{s}_{p,q}(\Rn),
   \quad 1<p\le \infty,\ 1<q\le \infty.
  \label{B8'-eq}
\end{align}
\end{thm}

\begin{proof}
When $p'+p=p'p$ and $q'+q=q'q$, then $F^{s}_{p,q}$ is the dual of 
$F^{-s}_{p',q'}$ since $1<p'<\infty $ and $1\le q'<\infty $; 
cf \cite[2.11]{T2}. The adjoint symbol
$a^*(x,\eta)$ is in $S^{d}_{1,1}$ by assumption, so
\begin{equation}
  a^*(x,D)\colon F^{-s}_{p',q'}(\Rn)\to F^{-s-d}_{p',q'}(\Rn)  
\end{equation}
is continuous whenever $-s-d>\max(0,\fracc n{p'}-n)=0$, ie
for $s<-d$; this follows from
Theorem~\ref{FBspq-thm} since $p'\ge 1$ and $q'\ge 1$. The adjoint 
$a^*(x,D)^*$ is therefore bounded $F^{s+d}_{p,q}\to F^{s}_{p,q}$, and it
equals $a(x,D)$ according to Theorem~\ref{tildeS-thm}.

For $s>0$ the property \eqref{F8'-eq} holds by Theorem~\ref{FBspq-thm}.
If $d\ge 0$ the gap with $s\in [-d,0]$ 
can be closed by a reduction to order $-1$; cf \cite{JJ10tmp}.

For the $B^{s}_{p,q}$ scale similar arguments apply, also for $p=\infty $.
\end{proof}

For symbols $a(x,\eta)$ in $\tilde S^d_{1,1}$
the special case $p=2=q$ of the above corollary, ie continuity 
$a(x,D)\colon H^{s+d}\to H^s$ for all $s\in \R$,
was obtained by H{\"o}rmander as an immediate consequence of
\cite[Thm.~4.1]{H89}, but first formulated in
\cite[Thm.~9.4.2]{H97}. 

In comparison Theorem~\ref{FB8-thm} may appear as a 
rather wide generalisation to the $L_p$-setting.
However, a specialisation of the two above results
to Sobolev and H{\"o}lder--Zygmund spaces gives the following result 
directly from the standard identifications \eqref{Hsp-id}, \eqref{Cs*-id}. 

\begin{cor}[{\cite[Cor.~7.6]{JJ10tmp}}]
  \label{HC-cor}
Every $a(x,D)\in \OP(S^d_{1,1}(\Rn\times\Rn))$ is bounded 
\begin{align}
   a(x,D)&\colon H^{s+d}_{p}(\Rn)\to H^{s}_{p}(\Rn),
   \quad s>0,\ 1<p<\infty,
  \label{Hsp-eq} \\
   a(x,D)&\colon C_*^{s+d}(\Rn)\to C_*^{s}(\Rn),
   \quad s>0.
  \label{Cs-eq}
\end{align}
This holds for all real $s$ whenever $a(x,\eta)$ belongs to the
self-adjoint subclass $\tilde S^d_{1,1}(\Rn\times \Rn)$. 
\end{cor}

The last extension to $p\ne 2$ of the results on $\tilde S^d_{1,1}$
in \cite{H89,H97} is also new.

\section{Direct estimates for the self-adjoint subclass}
\label{Fspq'-ssect}\noindent
To complement Theorem~\ref{FB8-thm} with similar
results valid for $p$, $q$ in $\,]0,1]$, it is natural to exploit 
the paradifferential decomposition \eqref{a123-eq} and the pointwise
estimates used above. 

However, in the results to follow below there will be an
arbitrarily small loss of smoothness. The reason is that the estimates of
$a^{(2)}_{\psi}(x,D)$ will be based on a corona condition, which is now
\emph{asymmetric} in the sense that the outer radii grow faster than the
inner ones.
That is, the last part of Lemma~\ref{Fspq-lem} will now be extended to
series $\sum u_j$ fulfilling 
the more general condition, where $0<\theta\le 1$ and $A>1$,
\begin{equation}
  \begin{aligned}
  \supp \cal F u_0&\subset \{\,\xi\mid |\xi|\le A2^j\,\},
\quad\text{for}\quad j\ge 0,
\\
    \supp \cal F u_j&\subset \{\,\xi\mid  \tfrac{1}{A}2^{\theta j}
   \le |\xi|\le A 2^{j} \,\}
 \quad\text{for $j\ge J\ge 1$}.
\end{aligned}
  \label{theta01-eq}
\end{equation}
This situation is probably known to experts in function spaces, 
but in lack of a reference it was analysed with
standard techniques from harmonic analysis in \cite{JJ10tmp}.

The main point of \eqref{theta01-eq} is that $\sum u_j$
still converges for $s\le 0$, albeit with a loss of smoothness
that arises in the cases below with $s'<s$.  
Actually the loss is proportional to $(1-\theta)/\theta$, 
hence tends to $\infty $ for $\theta\to0$, which reflects that convergence 
in some cases fails for $\theta=0$ 
(as can be easily seen for $\hat u_j=\tfrac{1}{j}\psi\in C^\infty_0$, 
$s=0$, $1<q\le \infty $).

\begin{prop}[{\cite[Prop.~7.7]{JJ10tmp}}] 
  \label{Fcor'-prop}
Let $s\in \R$, $0< p<\infty$, $0<q\le\infty$, $J\in\N$ 
and $0<\theta<1$ be given; with $q>n/(n+s)$ if $s>0$.
For each sequence
$(u_j)_{j\in \N_0}$ in $\cal S'(\Rn)$ fulfilling the corona
condition \eqref{theta01-eq}
together with the bound (usual modification for $q=\infty $)
\begin{equation}
  F:=
  \Nrm{(\sum_{j=0}^\infty |2^{sj}u_j(\cdot )|^q)^{\frac1q}}
       {L_p}<\infty,
  \label{F-id}
\end{equation}
the series $\sum_{j=0}^\infty u_j$ converges in $\cal S'(\Rn)$ to some
$u\in F^{s'}_{p,q}(\Rn)$ with
\begin{equation}
  \Nrm{u}{F^{s'}_{p,q}}\le cF,
\end{equation}
whereby the constant $c$ also depends on $s'$, that one can take to fulfil
\begin{align}
  s'&=s \quad\text{for}\quad s>\max(0,\fracnp-n),   
  \label{s's-eq}
\\
  s'&<s/\theta \quad\text{for}\quad s\le 0,\; p\ge 1,\; q\ge 1,
  \label{s's-ineq}\\
\intertext{or in general}
  s'&<s-\tfrac{1-\theta}\theta(\max(0,\fracnp-n)-s)_{+}.
  \label{gs's-ineq}
\end{align}
(Note that $s'=s$ is possible if the positive part 
$(\dots )_+$ has strictly negative argument.)

The conclusions carry over to $B^{s'}_{p,q}$ for any $q\in ]0,\infty ]$ when
$B:=(\sum_{j=0}^\infty 2^{sjq}\nrm{u_j}{p}^{q})^{\fracci 1q}<\infty $.
\end{prop}

\begin{rem}
The restriction above that $q> n/(n+s)$ for $s>0$ is not
severe, for if \eqref{F-id} holds for a sum-exponent in $\,]0,n/(n+s)]$,
then the constant $F$ is also finite for any $q>n/(n+s)$, which yields 
convergence and an estimate in a slightly larger space (but for
the same $s$ and $p$). 
\end{rem}

The proof of the proposition is omitted here, since it is lengthy and only has
little in common with the treatment of type $1,1$-operators.

However, thus prepared one can obtain the next result, which provides 
a general result for $0<p\le 1$.

\begin{thm}[{\cite[Thm.~7.9]{JJ10tmp}}]
  \label{FB8'-thm}
For every symbol $a(x,\eta)$ belonging to the self-adjoint subclass
$\tilde S^d_{1,1}(\Rn\times\Rn)$ the operator $a(x,D)$ 
is bounded for $0< p\le1$ and $0<q\le\infty$, 
\begin{align}
   a(x,D)&\colon F^{s+d}_{p,q}(\Rn)\to F^{s'}_{p,q}(\Rn),
   \quad \text{for } s'<s\le \fracnp-n,
  \label{F8-eq} \\
   a(x,D)&\colon B^{s+d}_{p,q}(\Rn)\to B^{s'}_{p,q}(\Rn)
   \quad\text{for } s'<s\le \fracnp-n.
  \label{B8-eq}
\end{align}
\end{thm}

\begin{proof}
The theorem follows by elaboration of the proof of Theorem~\ref{sigma-thm}.
By applying the last part of Theorem~\ref{FBspq-thm} to the difference
$a-a_{\chi,1}$, the question is reduced to the case of $a_{\chi,1}$.
Again this will be covered by treating general $a\in \tilde S^d_{1,1}$ for
which $\hat a(\xi,\eta)\ne0$
only holds for $\max(1,|\xi+\eta|)\le |\eta|$.

Under this assumption, $a^{(1)}(x,D)u$ and $a^{(3)}(x,D)u$ are covered by 
Theorem~\ref{FBspq-thm} for all
$s\in \R$; cf \eqref{apsi123-eq}. 
Thus it suffices to estimates the
series in \eqref{2a2-eq} for fixed $s'<s\le 0$.

Taking in \eqref{akkev-eq} the parameter $\varepsilon=2^{-k\theta}$ 
for $\theta\in \,]0,1[\,$ so small that $s'$ fulfils the last condition in
Proposition~\ref{Fcor'-prop} with $1-\theta$ instead of $\theta$
(cf remarks prior to \eqref{a22-eq}),
clearly
\begin{equation}
  2^{k(s+M)}|a_{k,\chi,\varepsilon}(x,D)v_k(x)| 
\le cv_k^*(N,R2^k;x)2^{k(s+d)} 2^{-k\theta(\sigma-1-N-M/\theta)}.
  \label{akkev''-eq}
\end{equation}
Here one may first of all take $N>n/\min(p,q)$ so that \eqref{u*M-eq}
applies.
Secondly, since by assumption $a(x,\eta)$ fulfils the twisted
diagonal condition \eqref{Hsigma-eq} of any real order, 
$\sigma$ can for any $M$ (with $\theta$ fixed as above) 
be chosen so that $2^{-k\theta(\sigma-1-N-M/\theta)}\le 1$.
This gives
\begin{equation}
  \begin{split}
  (\int\Nrm{2^{k(s+M)}a_{k,\chi,\varepsilon}(x,D)v_k(\cdot)}{\ell_q}^p
    \,dx)^{\fracpi}
  &\le c (\int\Nrm{2^{k(s+d)}v_k^*(N,R2^k;\cdot )}{\ell_q}^p\,dx)^{\fracpi}
\\
  &\le c' (\int\Nrm{2^{k(s+d)}v_k(\cdot )}{\ell_q}^p\,dx)^{\fracpi}
\le c''\Nrm{u}{F^{s+d}_{p,q}}.
  \label{akkeLp-eq}
  \end{split}
\end{equation}
Here the last inequality follows from the (quasi-)triangle inequality in
$\ell_q$ and $L_p$.

Since the spectral support rule and Proposition~\ref{ball-prop} imply
that $a_{k,\chi,\varepsilon}(x,D)v_k$ also has its spectrum in the ball
$B(0,2R2^k)$, the above estimate allows application of
Lemma~\ref{Fspq-lem}, if $M$ is so large that 
\begin{equation}
  M>0,\quad M+s>0,\quad M+s>\fracnp-n.
  \label{M-ineq}
\end{equation}
This gives convergence of 
$\sum a_{k,\chi,2^{-k\theta}}(x,D)v_k$ to a function in
$F^{s+M}_{p,\infty }$ fulfilling
\begin{equation}
  \Nrm{\sum_{k=1}^\infty a_{k,\chi,2^{-k\theta}}(x,D)v_k}{F^{s+M}_{p,\infty }}
  \le c\nrm{u}{F^{s+d}_{p,q}}.
  \label{a21-eq}
\end{equation}
On the left-hand side the embedding $F^{s+M}_{p,\infty }\hookrightarrow
F^{s}_{p,q}$ applies, of course.

For the remainder $\sum_{k=1}^\infty b_k(x,D)v_k$, cf \eqref{akke-id}
ff, note that \eqref{akkeLp-eq} holds for $M=0$ with the same $\sigma$. 
If combined with \eqref{a2Lplq-eq}, it follows by the (quasi-)triangle
inequality that 
\begin{equation}
  \int\Nrm{2^{ks}b_k(x,D)v_k(\cdot)}{\ell_q}^p
    \,dx
\le \int\Nrm{2^{ks}(a_k(x,D)-a_{k,\chi,2^{-k\theta}}(x,D))
        v_k(\cdot)}{\ell_q}^p \,dx
 \le c \nrm{u}{F^{s+d}_{p,q}}^p.
\end{equation} 
In addition the series can be shown (cf the proof of 
Theorem~\ref{sigma-thm}) to fulfil a corona condition
with inner radius $2^{(1-\theta)k}$ for all sufficiently large $k$,
so that Proposition~\ref{Fcor'-prop} applies. 
By the choice of $\theta$, this gives
\begin{equation}
  \Nrm{\sum_{k=1}^\infty b_k(x,D)v_k}{F^{s'}_{p,q }}
  \le c\nrm{u}{F^{s+d}_{p,q}}.  
  \label{a22-eq}
\end{equation}

In $a^{(2)}_{\psi}(x,D)u$ the other contribution $\sum
(a^k(x,D)-a^{k-h}(x,D))u_k$, cf \eqref{2a2-eq}, can be treated similarly.
This was also done in the proof of
Theorem~\ref{sigma-thm}, where in particular \eqref{akkev-eq} was shown to
hold for $(a^k-a^{k-h})_{\chi,\varepsilon}(x,D)u_k$, 
with just a change of the constant. Consequently \eqref{akkev''-eq} carries
over, and with \eqref{M-ineq} the same arguments as for \eqref{a21-eq},
\eqref{a22-eq} give 
\begin{equation}
  \Nrm{\sum_{k=h}^\infty 
        (a^k-a^{k-h})_{\chi,\varepsilon}(x,D)u_k}{F^{s+M}_{p,\infty }}
+  \Nrm{\sum_{k=h}^\infty \tilde b_k(x,D)u_k}{F^{s'}_{p,q }}
  \le c\nrm{u}{F^{s+d}_{p,q}}.  
  \label{a23-eq}
\end{equation}
Altogether the estimates \eqref{a21-eq}, \eqref{a22-eq}, \eqref{a23-eq}
show that
\begin{equation}
  \Nrm{a^{(2)}_{\psi}(x,D)u}{F^{s'}_{p,q}}\le c\Nrm{u}{F^{s+d}_{p,q}}.
\end{equation}
Via the paradifferential decomposition \eqref{a123-eq}, the operator
$a_{\psi}(x,D)$ is therefore a bounded linear map $F^{s+d}_{p,q}\to
F^{s'}_{p,q}$. Since $\cal S$ is dense for $q<\infty $ (a case one
can reduce to), there is no dependence on the 
modulation function $\psi$, so the type $1,1$-operator
$a(x,D)$ is defined and continuous on $F^{s+d}_{p,q}$ as stated.

The arguments are similar for the Besov spaces: it suffices to interchange
the order of the norms in $\ell_q$ and $L_p$, and to use the estimate in 
\eqref{u*M-eq} for each single $k$. 
\end{proof}

The proof above extends to cases with $0<p\le \infty $ when $s'<s\le
\max(0,\frac np-n)$.
But even though it like the proof of
Theorem~\ref{sigma-thm} uses a delicate splitting of $a(x,D)$ into 
9 infinite series, 
it barely fails to reprove Theorem~\ref{FB8-thm} due to the loss of
smoothness. Therefore only $p\le 1$ is
included in Theorem~\ref{FB8'-thm}.

When taken together, Theorems~\ref{FBspq-thm}, \ref{FB8-thm}
and \ref{FB8'-thm} give a satisfactory $L_p$-theory of operators
$a(x,D)$ in $\OP(\tilde S^{d}_{1,1})$, inasmuch as for the domain
$D(a(x,D))$ they cover all possible $s$, $p$. 
Only a few of the codomains seem barely
unoptimal, and these all concern cases with $0<q<1$ or $0<p\le 1$;
cf the parameters $r$ in Theorem~\ref{FBspq-thm} and $s'$ in Theorem~\ref{FB8'-thm}.

One particular interest of Theorem~\ref{FB8'-thm} concerns the well-known
identification of $F^0_{p,2}(\Rn)$ with the
so-called local Hardy space $h_p(\Rn)$ for $0<p\le 1$, which is described in
\cite{T2} and especially \cite[Ch.~1.4]{T3}. 
Here Theorem~\ref{FB8'-thm} gives
boundedness $a(x,D)\colon h_p(\Rn)\to F^{s'}_{p,2}(\Rn)$ for every $s'<0$,
but this can probably be improved in view of recent results:

\begin{rem}
Extensions to  $h_p(\Rn)$ of operators in the self-adjoint subclass 
$\OP(\tilde S^0_{1,1})$
were treated by J.~Hounie and R.~A.~dos~Santos Kapp \cite{HoSK09}, who
used atomic estimates to carry over the
$L_2$-boundedness of H{\"o}rmander \cite{H89,H97} to $h_p$, ie to obtain
estimates with $s'=s=0$. However, they
worked without a precise definition of type $1,1$-operators. 
\end{rem}

\begin{rem}
As an additional merit, the proof above gives that when $a(x,D)$ fulfils the
twisted diagonal condition of a specific order $\sigma>0$, then for $1\le
p\le \infty $ 
\begin{equation}
 B^{s}_{p,q}\cup F^{s}_{p,q}\subset D(a(x,D))  
\quad\text{for}\quad s>-\sigma+[n/2]+2.
\end{equation}
While this does provide a result
in the $L_p$ set-up, it is hardly optimal in view of  
L.~H{\"o}rmander's condition $s>-\sigma$ for $p=2$, that was recalled 
in \eqref{Hssigma-eq}.
\end{rem}

\chapter{Final remarks}
\label{final-sect}\noindent
In view of the satisfying results on type $1,1$-operators in 
Chapter~\ref{results-sect} and the continuity results in $\cal S'(\Rn)$ and in
the scales 
$H^s_p$, $C^s_*$, $F^{s}_{p,q}$ and $B^{s}_{p,q}$ presented in
Chapter~\ref{resultsII-sect},
their somewhat unusual definition by vanishing frequency modulation in
Definition~\ref{vfm-defn} should be well motivated.

As an open problem, it remains to characterise the type $1,1$-operators 
$a(x,D)$ that are
everywhere defined and continuous on $\cal S'(\Rn)$. For this it was
shown above to be sufficient that $a(x,\eta)$ is in $\tilde
S^d_{1,1}(\Rn\times \Rn)$, and it could of course be conjectured 
that this is necessary as well.

Similarly, since the works of G.~Bourdaud and L.~H{\"o}rmander, 
cf \cite[Ch.~IV]{Bou83}, \cite{Bou88}, \cite{H88,H89} and also \cite{H97}, 
it has remained an open problem to determine the operator class
\begin{equation}
  \BBb B(L_2(\Rn))\cap \OP(S^0_{1,1}).  
  \label{BO-eq}
\end{equation} 
Indeed, it was shown by G.~Bourdaud that this contains the self-adjoint subclass
$\OP(\tilde S^0_{1,1})$, and this sufficient condition has  
led some authors to somewhat misleading statements, eg that lack
of $L_2$-boundedness for $\OP(S^0_{1,1})$ is ``attributable to the
lack of self adjointness''. But self-adjointness is not necessary, since
already G.~Bourdaud, by modification of Ching's symbol \eqref{Ching-eq},
gave an example \cite[p.~1069]{Bou88} of an operator 
$\sigma(x,D)$ in  the subset
$ {\BBb B}(L_2)\bigcap\OP(S^0_{1,1}\setminus \tilde S^0_{1,1})$,
for which $\sigma(x,D)^*$ is not of type $1,1$.

\bigskip

As a summary, it is noted that the work grew out of the analysis of
semi-linear boundary value problems, which was reviewed in
Section~\ref{bvp-ssect}. In particular the need for a proof of
pseudo-locality of type $1,1$-operators was identified, as was the question
of how they can be defined at all. Subsequently Definition~\ref{vfm-defn} by
vanishing frequency modulation was
crystallised from the borderline analysis in Section~\ref{dom-ssect}.
And during investigation of the definition's consequences, the general
techniques of pointwise estimates, the spectral support rule and stability
under regular convergence was developed. With these tools, the properties 
\eqref{uni-pro}--\eqref{Fspq'-pro} in Chapter~\ref{vfm-sect} were derived for
pseudo-differential operators of type $1,1$.
%


\providecommand{\bysame}{\leavevmode\hbox to3em{\hrulefill}\thinspace}
\providecommand{\MR}{\relax\ifhmode\unskip\space\fi MR }
\providecommand{\MRhref}[2]{%
  \href{http://www.ams.org/mathscinet-getitem?mr=#1}{#2}
}
\providecommand{\href}[2]{#2}

\cleardoublepage
\begin{center}{\bf\large Resum\'e}\end{center}
\addcontentsline{toc}{chapter}{Resum\'e (Danish summary)}
\thispagestyle{plain}

\bigskip\noindent
Denne afhandling vedr{\o}rer den type af pseudodifferential operatorer,
der er kendt i litteraturen som type $1,1$-operatorer.  
Disse har v{\ae}ret kendt i moderne matematisk analyse is{\ae}r siden 1980, da det
blev vist, at de spiller en v{\ae}sentlig rolle for behandlingen af fuldt
ikke-line{\ae}re partielle differentialligninger.

Bidragene i denne afhandling skal ses p{\aa} baggrund af de fundamentale
resultater, der blev opn{\aa}et i 1988--89 af G.~Bourdaud og
L.~H{\"o}rmander. Disse viste at type $1,1$-operatorer har en r{\ae}kke
egenskaber, der afviger v{\ae}sentligt fra andre pseudodifferential
operatorers.

I afhandlingens arbejder fremf{\o}res for f{\o}rste gang en pr{\ae}cis
definition, af hvad 
en type $1,1$-operator er i almindelighed. Dette f{\o}lges op af en
redeg{\o}relse for, at definitionen indeholder flere af de tidligere 
udvidelser af begrebet.

Med udgangspunkt deri
bevises en tredive {\aa}r gammel formodning om, at type $1,1$-operatorer
er pseudolokale; det vil sige, at de ikke kan skabe nye singulariter i
de funktioner, de virker p{\aa}. Desuden almindeligg{\o}res et tidligere
resultat om, at disse operatorer kan {\ae}ndre eksisterende singulariteter;
dette g{\o}res ved at inddrage Weierstrass' intetsteds differentiable
funktion.

Det udledes ogs{\aa}, hvorledes type $1,1$-operatorer {\ae}ndrer st{\o}tten og
spektret af den funktion, der virkes p{\aa}. Sp{\o}rgsm{\aa}let, om hvilke
funktioner en given operator \emph{kan} virke p{\aa}, er ogs{\aa} diskuteret
indg{\aa}ende. Som et nyt resultat er det vist, at enhver type
$1,1$-operator kan anvendes p{\aa} alle glatte funktioner, der er
tempererede i L.~Schwartz' forstand.

For generelle tempererede distributioner er det vist, at type
$1,1$-operatorer kan virke p{\aa} dem, hvis deres symboler efter delvis
Fouriertransformering forsvinder i en kegleformet omegn af
sk{\ae}vdiagonalen i det fulde frekvensrum. Mere generelt er dette bevist
for de operatorer, der tilh{\o}rer den selvadjungerede delklasse.

Ydermere er tilsvarende egenskaber blevet givet en omfattende behandling 
i flere skalaer af
funktionsrum, s{\aa} som H{\"o}lderrum og Sobolevrum samt de mere
generelle Besov og Lizorkin--Triebelrum. Disse beskriver 
en r{\ae}kke forskellige differentiabilitets- og
integrabilitetsegenskaber.

Som en vigtig metode til opn{\aa}else af disse resultater er der blevet
indf{\o}rt en almen ramme for punktvise vurderinger af
pseudodifferential operatorer. Et generelt resultat er den s{\aa}kaldte
faktoriseringsulighed, som udsiger at virkningen af en operator p{\aa} en
funktion, med kompakt spektrum for eksempel, altid er mindre end en
vis symbolfaktor multipliceret med en maksimalfunktion af
Peetre--Fefferman--Stein type. Via en analyse af symbolfaktorens
asymptotiske egenskaber er uligheden vist at have et frugtbart
samspil med de dyadiske dekompositioner, der indg{\aa}r som en
hovedingrediens i kontinuitetsanalysen af type $1,1$-operatorer.

\end{document}